\documentclass[reqno,12pt]{amsart}         
\usepackage{amsfonts,amsmath,amssymb,amsthm,colordvi,mathrsfs,graphicx}
\usepackage{caption,subcaption,float}
\usepackage[utf8]{inputenc}
\usepackage{mathrsfs}
\usepackage{geometry}
\usepackage{enumerate}

\usepackage{amsmath, amssymb, amsthm, geometry, hyperref}

\usepackage{tabularx}
\usepackage{tabulary}
 
\hypersetup{
  colorlinks=true,
  linkcolor=blue,
  citecolor=blue,
  urlcolor=blue
}

\usepackage[all,knot,poly]{xy}
\usepackage{tikz-cd}
\usepackage{tikz}
\usepackage{verbatim}
\usetikzlibrary{arrows}
\usetikzlibrary{knots}
\tikzset{every path/.style={line width=0.4pt},every node/.style={transform shape,knot crossing,inner sep=1.5pt},>=triangle 60,text node/.style={rectangle,transform shape=false,black}}

 \usetikzlibrary{decorations.markings, arrows.meta}

\usepackage{blkarray}
\usepackage{array}

\theoremstyle{plain}      
\newtheorem{thm}{Theorem}[section]     
\newtheorem{theorem}[thm]{\bf Theorem}     
     
\newtheorem{lemma}[thm]{\bf Lemma}     
\newtheorem{proposition}[thm]{\bf Proposition}

\theoremstyle{remark}      
\newtheorem{example}[thm]{Example} 
 
\newtheorem{remark}[thm]{Remark} 
     
\theoremstyle{definition}      
\newtheorem{definition}[thm]{\bf Definition}     

 \newcommand{\Aut}{\operatorname{Aut}}

\newcommand{\C}{{\mathbb C}}

\newcommand{\T}{\mathcal{T}}

\newcommand{\OO}{\mathcal{O}}

\newcommand{\Spec}{\operatorname{Spec}}

\newcommand{\Hom}{\operatorname{Hom}}
\newcommand{\Ext}{\operatorname{Ext}}
\newcommand{\rank}{\operatorname{rank}}

  \newcommand{\PP}{\mathbb{P}}

\usepackage{mathtools}

\title[]{}

\subjclass[2020]{13D10, 14B10, 14B12, 14D15, 14B07}
\keywords{Infinitesimal deformations,  abstract rigidity, local rigidity, cones over projective varieties, isolated singularities}

\begin{document}

\title{On infinitesimal deformations of singular curves} %

\author{Mounir Nisse}
 \address{Mounir Nisse\\
Department of Mathematics, Xiamen University Malaysia, Jalan Sunsuria, Bandar Sunsuria, 43900, Sepang, Selangor, Malaysia.
}
\email{mounir.nisse@gmail.com, mounir.nisse@xmu.edu.my}
\thanks{}

\title{On infinitesimal deformations of singular varieties I}

\maketitle

\begin{abstract}
The deformation theory of singular varieties plays a central role in understanding the geometry and moduli of algebraic varieties. For a variety $X$ with possibly singular points, the space of first-order infinitesimal deformations is given by
\(
T^1_X = \operatorname{Ext}^1_{\mathcal{O}_X}(\Omega_X, \mathcal{O}_X),
\)
which measures the Zariski tangent space to the deformation functor of $X$. When $T^1_X = 0$, the variety is said to be \emph{rigid}; otherwise, nonzero elements of $T^1_X$ correspond to nontrivial first-order deformations. 
We investigate the structure of $T^1_X$ for singular varieties and provide cohomological and geometric criteria ensuring non-rigidity. In particular, we show that if the sheaf of tangent fields $T_X$ possesses nonvanishing cohomology $H^1(X, T_X)$ or if the local contributions $\mathcal{E}xt^1(\Omega_X, \mathcal{O}_X)$ are supported on a positive-dimensional singular locus, then $T^1_X \neq 0$. For hypersurface singularities $X = \{ f = 0 \} \subset \mathbb{C}^{n+1}$, we recover the Jacobian criterion,
\[
T^1_X \cong \frac{\mathbb{C}[x_0, \dots, x_n]}{(f, \partial f / \partial x_0, \dots, \partial f / \partial x_n)},
\]
where the positivity of the Tjurina number $\tau(X)$ characterizes the existence of nontrivial deformations. 
Moreover, non-rigidity arises when $X$ %
appears as a cone over a projectively nonrigid variety.
These criteria provide effective tools for detecting non-rigidity in both local and global settings, linking the vanishing of Ext and cohomology groups to the deformation behavior of singularities. The results contribute to a deeper understanding of the interplay between singularity theory, moduli, and the rigidity properties of algebraic varieties. To be more precise, these results provide effective geometric criteria for 
detecting when a singular variety contributes positive-dimensional 
components to its moduli space and clarify how singularities influence the 
local structure of moduli.

\end{abstract}



\section{Introduction}

The notion of \emph{rigidity} plays a central role across many branches of geometry and topology.  
Roughly speaking, rigidity refers to the idea that certain geometric or topological structures are uniquely determined by local data, or that they admit no nontrivial deformations preserving some prescribed structure.
In algebraic geometry, one formalizes this by saying that a variety $X$ is \emph{rigid} if its space of first-order infinitesimal deformations vanishes, i.e.
\(
T^1_X = \operatorname{Ext}^1(\Omega_X, \mathcal{O}_X) = 0.
\)
Intuitively, $T_X^1$ measures the possible infinitesimal ways of deforming $X$ within its ambient category.  
When $T_X^1 = 0$, the object $X$ is isolated in its moduli space, i.e.  there are no nearby non-isomorphic varieties that share its structure.  
For smooth algebraic curves, rigidity is completely characterized by the genus.  
If $X$ is a smooth projective curve of genus $g$, then
\(
T^1_X \cong H^1(X, T_X) \cong H^0(X, \omega_X^{\otimes 2})^\vee,
\)
and therefore
\[
\dim T^1_X =
\begin{cases}
0, & g = 0,\\
1, & g = 1,\\
3g - 3, & g \ge 2.
\end{cases}
\]
Hence, the projective line $\mathbb{P}^1$ is the only rigid smooth curve over an algebraically closed field, while elliptic and higher-genus curves admit nontrivial deformations corresponding to their moduli spaces (see Appendix A).
Rigidity questions for singular curves provide a bridge between singularity theory, deformation theory, and moduli problems. They are central to  understanding the local structure of the moduli of curves, the stability of curve singularities, and the classification of isolated complete intersection singularities. In particular, determining whether a given singular curve is rigid involves analyzing the interplay between its Jacobian relations, its non-reduced or embedded components, and the geometry of its normalization.
The goal of this paper is to explore the conditions under which a curve singularity can be rigid, the role played by its reduced and non-reduced structures, and the geometric meaning of the vanishing (or non-vanishing) of its infinitesimal deformation space $T^1_X$.

\vspace{0.1cm}
 
The concept of {\em rigidity} occupies a central position in deformation theory, algebraic geometry, and singularity theory. %
Formally, for a complex space or scheme $X$, rigidity is expressed by the vanishing of its first-order deformation space,
$
T^1_X = 0,
$
where $T^1_X$ is the first tangent cohomology group $\mathrm{Ext}^1_{\mathcal{O}_X}(\Omega^1_X,\mathcal{O}_X)$. The space $T^1_X$ governs infinitesimal deformations, while obstructions to extending them lie in $T^2_X$. When $T^1_X=0$, every small deformation is trivial, and the germ $(X,0)$ is analytically rigid.

Rigidity phenomena are particularly subtle in the case of curve singularities, that is, 1-dimensional complex analytic or algebraic spaces with singular points. The local deformation theory of such singularities traces its origins to the pioneering work of Schlessinger (1968), who formalized the notion of deformation functors and established general criteria for their representability. His results provided the abstract framework in which the tangent space $T^1_X$ naturally arises as the Zariski tangent space to the deformation functor $\mathcal{D}\mathrm{ef}_X$. Subsequent work by Greuel, Looijenga, and others in the 1970's and 1980's connected this functorial perspective to explicit computational invariants of curve singularities, such as the {\em Tjurina number}, the $\delta$-invariant, and the {\em Milnor number} $\mu$.


While algebraic rigidity is expressed in terms of infinitesimal deformations, similar ideas arise in the context of low-dimensional topology, particularly in the study of \emph{knots, links}, and \emph{3-manifolds}.  
A knot can be viewed as a smooth embedding $K: S^1 \hookrightarrow S^3$, and one can ask whether two knots that are topologically equivalent can be deformed through smooth embeddings; that is, whether the embedding type of a knot is \emph{rigid} under certain geometric or analytic constraints.

From an analytic viewpoint, rigidity can be understood through the deformation theory of complex structures, particularly on Riemann surfaces.  
A complex structure on a smooth surface $S$ can be encoded by a Cauchy--Riemann operator
\(
\bar{\partial} : \mathcal{A}^0(T_S) \longrightarrow \mathcal{A}^{0,1}(T_S),
\)
and infinitesimal deformations correspond to elements of $H^{0,1}(T_S) \cong H^1(S, T_S)$.  
Thus, the analytic deformation theory precisely mirrors the algebraic one:  
the space of infinitesimal deformations of the complex structure is the same as $T_X^1$ in the algebraic setting.
When $S$ has genus $g = 0$, the complex structure is unique (corresponding to $\mathbb{P}^1$), reflecting analytic rigidity.  
For genus $g \ge 1$, the nonvanishing of $H^1(S, T_S)$ gives rise to the Teichmüller space $\mathcal{T}_g$, a complex manifold of dimension $3g - 3$, parameterizing inequivalent complex structures on $S$.

The interplay between analytic deformation theory and geometric topology becomes even more profound when considering $3$-manifolds that fiber over Riemann surfaces, or when studying character varieties of surface groups into Lie groups.  
In these contexts, the deformation of complex structures on Riemann surfaces corresponds to geometric deformations of $3$-manifolds, bridging complex analysis, hyperbolic geometry, and low-dimensional topology.
The comparison emphasizes the dimensional transition from flexibility to rigidity, e.g.
in dimension $1$ and $2$ (curves and surfaces), complex and algebraic structures admit rich moduli spaces,
whereas in dimension $3$ and higher, geometric structures tend to be rigid, constrained by deep topological invariants.
Rigidity, whether in algebraic geometry or low-dimensional topology, captures the tension between flexibility and uniqueness of geometric structures.  
For smooth algebraic curves, only $\mathbb{P}^1$ is rigid, while higher-genus curves admit rich moduli.

\vspace{0.1cm}

For smooth varieties, the theory is well understood, one has $T^1_X \cong H^1(X, T_X)$, and the obstructions lie in $H^2(X, T_X)$. However, for singular varieties the situation is more subtle, as local contributions from the singular locus appear through the sheaves $\mathcal{E}xt^i(\Omega_X, \mathcal{O}_X)$. These contributions often encode the deformation behavior of individual singular points and determine whether the singularity is \emph{rigid} or \emph{smoothable}. A classical local model is the hypersurface singularity
\(
X = \{ f = 0 \} \subset \mathbb{C}^{n+1},
\)
for which
\(
\dim_{\mathbb{C}}(T^1_X) = \tau(X)\)
 measures the number of independent infinitesimal deformations. Such an $X$ is rigid precisely when the Tjurina number $\tau(X) = 0$.

Several criteria for rigidity and non-rigidity have been established in the literature. 
In this work, we focus on describing and refining cohomological and geometric criteria ensuring non-rigidity for singular varieties, more precisely, on curves and surfaces. We analyze the structure of  the space of \emph{global first-order deformations} of the variety \(X\)
    (including smoothing directions of singularities)
$T^1_X$ and its local and global components, emphasizing how the presence of nontrivial $H^1(X, T_X)$ or positive-dimensional support of $\mathcal{E}xt^1(\Omega_X, \mathcal{O}_X)$ forces nontrivial infinitesimal deformations. These results contribute to the broader understanding of the deformation theory of singular spaces and the mechanisms by which rigidity fails in both local and global settings. Moreover, we examine several detailed and illustrative examples to better understand the failure of rigidity in most singular reduced curves and normal surfaces.


A fundamental feature of an isolated complex singularity is that its
local topology is captured by its link, a smooth compact manifold obtained
as the intersection with a small sphere. Moreover, its infinitesimal deformations are
governed by its cotangent cohomology.
While the link is topologically invariant under small deformations, its
additional geometric structures, such as contact, CR, or other
structures may vary nontrivially in families ({\it cf} \cite{N4-25} and \cite{N5-25}).
Thus, deformation theory provides a bridge between analytic deformations of
singularities and deformations of geometric structures on their links. Moreover, there are some applications in physics as gravity, black hole, and crossing event horizons ({\it cf} \cite{LN1-25} and \cite{LN2-25}). These will be a continuation of our actual work.

 \vspace{0.2cm}
 
The organization of the paper is as follows. 
Section~2 reviews the necessary preliminaries and introduces the tools used throughout the work. 
Section~3 establishes several criteria for rigidity and non-rigidity, together with their geometric interpretations. 
Section~4 presents a criterion for the non-vanishing of \(T_X^{1}\). 
Section~5 provides a collection of detailed illustrative examples. 
Section~6 examines the local-to-global Ext spectral sequence and studies rigidity through the computation of 
\(\dim \Ext^{1}(\Omega_X, \mathcal{O}_X)\), accompanied by geometric interpretations and additional examples. 
Finally, Section~7 contains detailed proofs of standard results in deformation theory, along with further illustrative examples.


\section{Preliminaries}

The study of \emph{infinitesimal deformations of singular curves} lies at the intersection of algebraic geometry, singularity theory, and deformation theory. 
Roughly speaking, one seeks to understand how a given singular curve $X$  defined over an algebraically closed field $k$, where  in this paper it is typically $\mathbb{C}$, can be locally or globally deformed within families of nearby algebraic curves, and to characterize the tangent and obstruction spaces that govern such deformations. 
From a geometric viewpoint, this analysis describes how the embedded or abstract structure of $X$ changes under small perturbations of its defining equations, while from an algebraic perspective it is encoded in the cotangent complex and its cohomology.

\vspace{0.1cm}

Let $X$ be a reduced curve, possibly singular, defined over $\mathbb{C}$. 
A \emph{first-order (infinitesimal) deformation} of $X$ is a flat family 
\[
\pi : \mathcal{X} \to \operatorname{Spec}(\mathbb{C}[\varepsilon]/(\varepsilon^2))
\]
together with an identification $\mathcal{X}_0 \simeq X$. 
Such a deformation represents a “first-order thickening’’ of $X$ and captures infinitesimal directions along which the curve can be deformed. 
The isomorphism classes of such infinitesimal deformations form a finite-dimensional $\mathbb{C}$-vector space, denoted $T_X^1$, the \emph{Zariski tangent space to the deformation functor of $X$}. 
Formally, one may express
\[
T_X^1 \;=\; \operatorname{Ext}^1_{\mathcal{O}_X}(\Omega_X, \mathcal{O}_X),
\]
which measures infinitesimal extensions of the structure sheaf $\mathcal{O}_X$ by its module of Kähler differentials $\Omega_X$. 
Obstructions to extending such deformations to higher order are encoded in 
$\operatorname{Ext}^2_{\mathcal{O}_X}(\Omega_X, \mathcal{O}_X)$.

\medskip

When $X$ is smooth, infinitesimal deformations are unobstructed, and the classical Kodaira--Spencer theory identifies
\[
T_X^1 \;\simeq\; H^1(X, T_X),
\]
where $T_X$ is the tangent sheaf of $X$ (cf.~\cite{K-86, GH-78}, or \cite{H-10}) . 
However, in the singular case the sheaf $\Omega_X$ is not locally free, and the deformation space acquires additional local contributions at the singular points. 
Indeed, from the local-to-global spectral sequence for cotangent cohomology
\[
E_2^{p,q} \;=\;
H^p\!\left(X,\, \mathcal{T}^q_X \right)
\;\Rightarrow\;
T^{p+q}_X,
\]
  one obtains the low-degree exact sequence
\[
0 \longrightarrow H^1(X,\mathcal{T}_X)
  \longrightarrow T_X^1
  \longrightarrow H^0(X,\mathcal{T}^1_X)
  \longrightarrow H^2(X,\mathcal{T}_X).
\]
If \(X\) has only isolated singularities, then
\[
H^0(X,\mathcal{T}^1_X)
  \;=\; \bigoplus_{p\in \operatorname{Sing}(X)}\!
         T^1_{\mathcal{O}_{X,p}},
\]
and the sequence becomes
  the following  natural exact sequence
\[
0 \longrightarrow H^1(X, T_X) \longrightarrow T_X^1 \longrightarrow \bigoplus_{p\in \operatorname{Sing}(X)} T^1_{\mathcal{O}_{X,p}} \longrightarrow 0,
\]
in which the first term governs the global (or embedded) deformations of the normalization of $X$, and the second term aggregates the local deformation spaces of the singularities (for more details see Illusie \cite{I-71} III.2.2, III.3.5, or Hartshorne \cite{H-10}
Theorem~8.5.1], or Sernesi \cite{S-06}
Proposition~3.4.9). 
These local spaces
\[
T^1_{\mathcal{O}_{X,p}} = \operatorname{Ext}^1_{\mathcal{O}_{X,p}}(\Omega_{X,p}, \mathcal{O}_{X,p})
\]
describe the infinitesimal versal deformation of the singularity at a point $p$ and determine whether it is rigid, smoothable, or admits moduli.


In many situations of interest, for instance, for plane curve singularities or local complete intersections, the local deformation theory of $X$ can be described explicitly via the \emph{Tjurina algebra}
\[
T^1_{\mathcal{O}_{X,p}} \;\cong\; \mathbb{C}\{x_1,\ldots,x_n\} / (f, \partial f/\partial x_1, \ldots, \partial f/\partial x_n),
\]
whose dimension $\tau(X,p)$ is the \emph{Tjurina number} of the singularity. 
The global behavior of the curve’s deformation space, including questions of smoothability or rigidity, depends crucially on the balance between these local invariants and the global cohomological constraints.


Understanding infinitesimal deformations this provides the first step toward constructing the \emph{versal deformation space} of $X$ and analyzing the geometry of the corresponding base. 
In the case of singular projective curves, this leads naturally to the study of the local structure of the Hilbert and moduli schemes near points representing $X$, as well as to relations with equisingularity, smoothing components, and deformation-obstruction theories in modern enumerative geometry.


\vspace{0.2cm}

Let \(X\) be a reduced projective curve over a field \(k\).
The space of first-order infinitesimal deformations of \(X\) is
\(
\operatorname{Ext}^1(\Omega_X,\mathcal{O}_X).
\)
\begin{definition}
Let \(X\) be a reduced projective curve over a field \(k\).
We say that \(X\) is \emph{rigid} if and only if every first-order deformation is trivial, i.e. \( \operatorname{Ext}^1(\Omega_X,\mathcal{O}_X)=0.\)
\end{definition}
%
\noindent From the local-to-global exact sequence
\[
0 \longrightarrow H^1(T_X)
\longrightarrow \operatorname{Ext}^1(\Omega_X,\mathcal{O}_X)
\longrightarrow H^0(\mathcal{T}^1_X)
\longrightarrow 0,
\]
where
\[
T_X := \mathcal{H}om(\Omega_X,\mathcal{O}_X)
\quad\text{and}\quad
\mathcal{T}^1_X := \mathcal{E}xt^1(\Omega_X,\mathcal{O}_X).
\]
We see that rigidity requires the vanishing of both terms:
\[
H^1(T_X)=0 \quad\text{and}\quad H^0(\mathcal{T}^1_X)=0.
\]

\vspace{0.1cm}

\noindent
{Geometrically, this means the following:}
\begin{itemize}
  \item[(i)] \(H^0(T_X)\) corresponds to \emph{globally trivial} infinitesimal deformations,
  i.e.\ those induced by global vector fields (infinitesimal automorphisms of \(X\)).

  \item[(ii)] \(H^1(X,T_X)\) corresponds to \emph{locally trivial} first-order deformations:
  these deform \(X\) as a global scheme but leave each local analytic type unchanged.

  \item[(iii)] \(\mathcal{T}^1_X\) measures the \emph{local non-triviality} of deformations:
  it is supported at the singular locus of \(X\), and its stalk at a singular point \(p\)
  is the local Tjurina algebra
  \[
  \mathcal{T}^1_{X,p}
  \;=\;
  \frac{\mathcal{O}_{\mathbb{A}^n,p}}{(f_1,\ldots,f_r,\,
  \partial f_1/\partial x_i,\ldots,\partial f_r/\partial x_i)},
  \]
  which describes the local space of first-order deformations of the singularity \((X,p)\).
  Thus, \(H^0(\mathcal{T}^1_X)\) is the direct sum of the local deformation spaces of all singular points of \(X\).
\end{itemize}

\vspace{0.1cm}

In other words, we have the following:
\begin{itemize}
  \item[(i)] \(H^1(T_X)=0\) means there are no \emph{locally trivial} global deformations;
  \item[(ii)] \(H^0(\mathcal{T}^1_X)=0\) means all singularities of \(X\) are \emph{locally rigid}.
\end{itemize}

\vspace{0.2cm}

\begin{example}${}$

\begin{itemize}
  \item[1.] The projective line \(\mathbb{P}^1\) is rigid, since \(H^1(T_{\mathbb{P}^1})=0\).
  \item[2.] A smooth curve of genus \(g\ge 2\) is not rigid:
  its deformations form a moduli space of dimension \(3g-3>0\).
  \item[3.] A nodal cubic curve \(y^2=x^3+x^2\) is not rigid, because
  \(\dim_k\operatorname{Ext}^1(\Omega_X,\mathcal{O}_X)=1\),
  corresponding to its one-parameter smoothing. This example will be developed in 
  details later.
\end{itemize}
\end{example}
\noindent
Therefore, we have the following equivalences:
\[
X \text{ rigid } \iff 
\operatorname{Ext}^1(\Omega_X,\mathcal{O}_X)=0
\iff H^1(T_X)=H^0(\mathcal{T}^1_X)=0.
\]


\vspace{0.2cm}

\begin{example}

Let \(X\subset\mathbb{P}^2\) be the plane cubic given in the affine chart \(z=1\) by
\[
f(x,y)=y^2-x^3-x^2=0.
\]
Assume that we are working over a field \(k\) with \(\operatorname{char}(k) =0\). The unique singular point is \(p=(0,0)\), since
\[
f_x=-3x^2-2x,\qquad f_y=2y,
\]
and \(f=f_x=f_y=0\) forces \(x=y=0\).
\noindent For an isolated hypersurface singularity \(f\) in two variables the local space of first order deformations (the local Tjurina algebra, often denoted \(T^1_{(X,p)}\)) is
\[
T^1_{(X,p)} \;=\; \frac{\mathcal{O}_{\mathbb{A}^2,p}}{(f,\;f_x,\;f_y)}
\;=\; \frac{k[[x,y]]}{\big(f,\;f_x,\;f_y\big)} .
\]
We compute this quotient explicitly for our \(f\). Since \(f_y=2y\) and \(2\) is a unit, \(y\in(f_y)\), so in the quotient we may set \(y=0\). Substituting \(y=0\) into \(f\) and \(f_x\) we get generators in \(k[[x]]\):
\[
f|_{y=0} = -x^3-x^2 = -x^2(x+1),\qquad f_x|_{y=0} = -3x^2-2x = -x(3x+2).
\]
Hence,  \(x\in (f,f_x,f_y)\), and  therefore, both \(x\) and \(y\) lie in the ideal, so the quotient is
\[
\frac{k[[x,y]]}{(f,f_x,f_y)} \cong \frac{k[[x,y]]}{(x,y)} \cong k.
\]
Consequently, the local Tjurina algebra has dimension \(1\):
\[
\dim_k T^1_{(X,p)} = 1 .
\]

\end{example}

\section{Several criteria for rigidity and non-rigidity}
\noindent  Globally, for a reduced projective curve with isolated singularities the sheaf \(\mathcal{T}^1_X\) (whose stalk at a singular point \(p\) is the local Tjurina algebra) is supported at the singular points, and there is an exact sequence identifying the global first-order deformations with the sum of local contributions modulo global automorphisms. 
For our previous 
nodal cubic \(X\) in Example 2.3, the only contribution comes from the single node \(p\), and there is no further \(H^1\) obstruction coming from \(T_X\) (equivalently, the normalization is \(\mathbb{P}^1\), and then any \(H^1\) terms vanish in this case). Hence, the global deformation space
\[
\operatorname{Ext}^1(\Omega_X,\mathcal{O}_X) \cong H^0\big(\mathcal{T}^1_X\big)
\cong T^1_{(X,p)} \cong k.
\]
Moreover, from the local-to-global exact sequence we get $H^1(X, T_X) =0$. %


\vspace{0.1cm}

\begin{proposition}
Let $X\subset \C^2$ be a plane curve defined by a nonzero polynomial $f\in \C[x,y]$. If $X$ is singular, then $T^1_X\neq 0$ (equivalently $X$ is not rigid locally), i.e. there are  non-trivial infinitesimal deformations of $X$ singularities.
\end{proposition}

\begin{proof}
For a plane curve $X=\{f=0\}\subset \C$ one has the local description
\[
T^1_{X,p}\cong \mathcal{O}_{\C^2,p}/(f,f_x,f_y)
\]
at any point $p\in X$.
If $X$ is singular at $p$ then $f(p)=f_x(p)=f_y(p)=0$, and then  the maximal ideal $\mathfrak{m_p}$ contains the ideal $(f,f_x,f_y)$, thus,
\[
T^1_{X,p}\cong\mathcal{O}_{\C^2,p}/(f,f_x,f_y)\neq 0.
\]
Therefore, $T^1_X$ is nonzero and in fact supported at the singular locus, so $X$ is not  rigid (locally).
\end{proof}


\vspace{0.2cm}

More generally, we have a formula for any complete intersection. Indeed, let \(X=V(f_1,\dots,f_r)\subset\mathbb{C}^n\) 
%
  be a local complete intersection defined by
a regular sequence $f_1,\dots,f_r$, and let $\mathcal{I}$ be the ideal
sheaf generated by these equations.  Then
\[
\mathcal{I}/\mathcal{I}^2 \cong \mathcal{O}_X^{\,r},
\qquad
\Omega^1_{\mathbb{C}^n}|_X \cong \mathcal{O}_X^{\,n}.
\]

The conormal exact sequence is
\[
0
\longrightarrow \mathcal{I}/\mathcal{I}^2
\longrightarrow \Omega^1_{\mathbb{C}^n}\otimes_{\mathcal{O}_X} \mathcal{O}_X
\longrightarrow \Omega^1_X
\longrightarrow 0,
\]
which, using the identifications above, becomes
\[
0
\longrightarrow \mathcal{O}_X^{\,r}
\overset{df}{\longrightarrow}
\mathcal{O}_X^{\,n}
\longrightarrow
\Omega^1_X
\longrightarrow 0.
\tag{*}
\]
%
%
%
%
 %
\noindent Let
\[
J_f=\left(\frac{\partial f_i}{\partial x_j}\right)_{1\le i\le r,\;1\le j\le n}
\]
be the Jacobian matrix. This defines an $\mathcal{O}_X$-linear map
\[
J_f:\mathcal{O}_X^{\,n}\longrightarrow\mathcal{O}_X^{\,r},\qquad
(\xi_1,\dots,\xi_n)\mapsto\Bigl(\sum_{j=1}^n\xi_j\frac{\partial f_1}{\partial x_j},\dots,
\sum_{j=1}^n\xi_j\frac{\partial f_r}{\partial x_j}\Bigr).
\]
Applying the functor $\operatorname{Hom}_{\mathcal{O}_X}(-,\mathcal{O}_X)$
to the exact sequence $(*)$ and using that the $\mathcal{O}_X$–modules
$\mathcal{O}_X^{\,r}$ and $\mathcal{O}_X^{\,n}$ are locally free,
one obtains the long exact sequence
\[
0 \longrightarrow \operatorname{Hom}_{\mathcal{O}_X}(\Omega^1_X,\mathcal{O}_X)
\longrightarrow \operatorname{Hom}_{\mathcal{O}_X}(\mathcal{O}_X^{\,n},\mathcal{O}_X)
\longrightarrow \operatorname{Hom}_{\mathcal{O}_X}(\mathcal{O}_X^{\,r},\mathcal{O}_X)
\longrightarrow \operatorname{Ext}^1_{\mathcal{O}_X}(\Omega^1_X,\mathcal{O}_X)
\longrightarrow 0,
\]
since the higher Ext groups of locally free sheaves vanish.
Now, by 
identifying
\[
\operatorname{Hom}_{\mathcal{O}_X}(\mathcal{O}_X^{\,n},\mathcal{O}_X) 
\cong \mathcal{O}_X^{\,n},
\qquad
\operatorname{Hom}_{\mathcal{O}_X}(\mathcal{O}_X^{\,r},\mathcal{O}_X) 
\cong \mathcal{O}_X^{\,r},
\]
and noting that the dual map is given by the Jacobian matrix
$J_f = (\partial f_i/\partial x_j)$, we obtain the exact sequence
\[
0\longrightarrow\operatorname{Hom}_{\mathcal{O}_X}(\Omega^1_X,\mathcal{O}_X)
\longrightarrow\mathcal{O}_X^{\,n}
\xrightarrow{J_f}\mathcal{O}_X^{\,r}
\longrightarrow\operatorname{Ext}^1_{\mathcal{O}_X}(\Omega^1_X,\mathcal{O}_X)
\longrightarrow 0.  \hspace{1.7cm}({}^{**})
\]
The sheaf (or module) of first order deformations is
\[
T^1_X:=\operatorname{Ext}^1_{\mathcal{O}_X}(\Omega^1_X,\mathcal{O}_X)
\cong\operatorname{coker}\bigl(J_f:\mathcal{O}_X^{\,n}\to\mathcal{O}_X^{\,r}\bigr)
=\frac{\mathcal{O}_X^{\,r}}{J_f(\mathcal{O}_X^{\,n})}.
\]
In particular, at a point (e.g.\ an isolated complete intersection singularity at $0$),
\[
T^1_{X,0}\cong\frac{\mathcal{O}_{X,0}^{\,r}}{\left\langle
\frac{\partial f_i}{\partial x_j}\;\bigg|\;1\le i\le r,\;1\le j\le n\right\rangle
\cdot\mathcal{O}_{X,0}}.
\]

\vspace{0.2cm}

Let $X=V(f_1,\dots,f_r)\subset\mathbb C^n$ be a complete intersection, and let us 
work at a point $p\in X$ i.e. locally.  Consider $\mathcal O_{X,p}=\mathcal O_{\mathbb C^n,p}/(f_1,\dots,f_r)$,
the Jacobian matrix $J_f=(\partial f_i/\partial x_j)_{i,j}$,  and let
\(M=J_f(\mathcal{O}_X^{\,n})\subset\mathcal{O}_X^{\,r}\). The singular locus is then
\[
\operatorname{Sing}(X)=V\Big((f_1,\dots,f_r),\,\,(\text{all }r\times r\text{ minors of }J_f)\Big).
\]
For a point \(p\in X\) we have the local description
\[
T^1_{X,p}\cong\frac{\mathcal{O}_{X,p}^{\,r}}{M_p}.
\] 
Hence, the complex dimension of the local  first order deformation space at $p$ is
\[
\dim_{\mathbb C}T^1_{X,p}=\dim_{\mathbb C}\frac{\mathcal O_{X,p}^{\,r}}{M_p}, %
\]
which is finite precisely when the singularity at $p$ is isolated.


\vspace{0.3cm}

\noindent 
\subsection{\bf Geometric interpretation of the sequence $({}^{**})$.}
Geometrically, the exact sequence $({}^{**})$
expresses the relationship between tangent vectors, the Jacobian matrix,
and infinitesimal deformations.
An element of $\mathcal{O}_X^{\,n}$ corresponds to a vector field
$v=\sum a_j \frac{\partial}{\partial x_j}$ on the ambient space
$\mathbb{C}^n$ restricted to $X$.  The Jacobian $J_f$ acts by
\[
J_f(v)
= \bigl(v(f_1),\dots,v(f_r)\bigr),
\]
so that $J_f(v)=0$ precisely when $v$ is tangent to all hypersurfaces
$f_i=0$.  Thus
\[
\ker(J_f)=\operatorname{Hom}(\Omega_X^1,\mathcal{O}_X)=T_X
\]
is the sheaf of tangent vectors on $X$.

The cokernel
\[
\operatorname{coker}(J_f)
=\operatorname{Ext}^1_{\mathcal{O}_X}(\Omega_X^1,\mathcal{O}_X)
=T_X^1
\]
measures infinitesimal deformations of $X$ that are not induced by
any ambient vector field.  To be more precise, an infinitesimal deformation
\[
f_i \mapsto f_i + \varepsilon g_i
\]
is trivial (induced by a first-order change of coordinates) exactly when
$(g_1,\dots,g_r)$ lies in the image of $J_f$.  Nontrivial deformations
correspond to elements of the cokernel and represent obstruction
directions in which $X$ may or may not smooth.


\vspace{0.2cm}
\begin{theorem}
Let \(X=V(f_1,\dots,f_r)\subset\mathbb{C}^n\) be a complete intersection and let \(p\in X\).
Write \(\mathcal{O}_{X,p}=\mathcal{O}_{\mathbb{C}^n,p}/(f_1,\dots,f_r)\) and let
\(J_f=(\partial f_i/\partial x_j)_{i,j}\) be the Jacobian matrix. If \(p\) is a smooth point of \(X\)
then \(\operatorname{rank}(J_f(p))=r\) and \(T^1_{X,p}=0\).
\end{theorem}

\begin{proof}
The Zariski (or analytic) tangent space of \(X\) at \(p\) may be computed as
\[
T_{X,p}=\{v\in\mathbb{C}^n:\ J_f(p)\,v=0\},
\]
and then \(\dim_{\mathbb{C}}T_{X,p}=n-\operatorname{rank}J_f(p)\).
Since \(X\) is a complete intersection of codimension \(r\), the point \(p\) is smooth iff
\(\dim T_{X,p}=n-r\). Therefore smoothness at \(p\) is equivalent to \(\operatorname{rank}J_f(p)=r\).

Let \(I=(f_1,\dots,f_r)\). There is the conormal exact sequence (localized at \(p\)):
\[
I/I^2 \xrightarrow{\ \alpha\ } \Omega^1_{\mathbb{C}^n,p}\otimes_{\mathcal{O}_{\mathbb{C}^n,p}}\mathcal{O}_{X,p}
\longrightarrow \Omega^1_{X,p}\longrightarrow 0.
\]
Because \(I\) is generated by the \(f_i\) we have \(I/I^2\cong\mathcal{O}_{X,p}^r\), and the map \(\alpha\)
sends the \(i\)-th standard basis vector to \(df_i\).

Apply \(\operatorname{Hom}_{\mathcal{O}_{X,p}}(-,\mathcal{O}_{X,p})\). Since
\(\Omega^1_{\mathbb{C}^n,p}\) is free of rank \(n\), one has
\(\operatorname{Ext}^1(\Omega^1_{\mathbb{C}^n,p}\otimes\mathcal{O}_{X,p},\mathcal{O}_{X,p})=0\), and
duality yields an exact sequence
\[
\mathcal{O}_{X,p}^n \xrightarrow{J_f} \mathcal{O}_{X,p}^r \longrightarrow
\operatorname{Ext}^1(\Omega^1_{X,p},\mathcal{O}_{X,p}) \longrightarrow 0.
\]
By definition \(T^1_{X,p}:=\operatorname{Ext}^1(\Omega^1_{X,p},\mathcal{O}_{X,p})\), and the map
\(\mathcal{O}_{X,p}^n\to\mathcal{O}_{X,p}^r\) is multiplication by the Jacobian matrix \(J_f\).

If \(p\) is a smooth point, then \(\operatorname{rank}(J_f(p))=r\), i.e.  some \(r\times r\) minor of \(J_f\)
evaluates to a nonzero complex number at \(p\). That minor is therefore a unit in the local ring
\(\mathcal{O}_{X,p}\), which implies that  the matrix \(J_f\) has a right inverse over
\(\mathcal{O}_{X,p}\) and then the map \(\mathcal{O}_{X,p}^n\to\mathcal{O}_{X,p}^r\) is surjective.
Thus,  its cokernel \(T^1_{X,p}\) vanishes.

This proves that at a smooth point \(p\) of the complete intersection \(X\) we have
\(\operatorname{rank}J_f(p)=r\) and \(T^1_{X,p}=0\), i.e. the germ \((X,p)\) is locally rigid.
\end{proof}

\vspace{0.2cm}

\begin{theorem}
Let \(X=V(f_1,\dots,f_r)\subset\mathbb{C}^n\) be a complete intersection and let \(p\in X\).
Write \(\mathcal{O}_{X,p}=\mathcal{O}_{\mathbb{C}^n,p}/(f_1,\dots,f_r)\) and let
\(J_f=(\partial f_i/\partial x_j)_{i,j}\) be the Jacobian matrix. Then
\[
\dim_{\mathbb{C}}T^1_{X,p}<\infty \quad\Longleftrightarrow\quad p\text{ is an isolated singular point of }X.
\]
\end{theorem}

\begin{proof}
Recall (as in the complete-intersection case) the local conormal exact sequence and its dualization
 give the canonical identification
\[
T^1_{X,p}\;=\;\operatorname{Ext}^1_{\mathcal{O}_{X,p}}\big(\Omega^1_{X,p},\mathcal{O}_{X,p}\big)
\cong \operatorname{coker}\big(\mathcal{O}_{X,p}^n \xrightarrow{\,J_f\,} \mathcal{O}_{X,p}^r\big).
\]
In particular, \(T^1_{X,p}\) is a finitely generated \(\mathcal{O}_{X,p}\)-module and its support is contained in the singular locus of \(X\) near \(p\).

\medskip\noindent
Now, suppose \(\dim_{\mathbb C}T^1_{X,p}<\infty\). Then \(T^1_{X,p}\) is an \(\mathcal{O}_{X,p}\)-module of finite \(\C\)-dimension, hence it is of finite length and therefore its support is the closed point \(\{p\}\). Equivalently, for every prime \(\frak q\subset\mathcal{O}_{X,p}\) with \(\frak q\neq\frak m_p\) we have \((T^1_{X,p})_{\frak q}=0\). But localization commutes with co-kernels, and  then
\[
\big(\operatorname{coker}(\mathcal{O}_{X,p}^n\xrightarrow{J_f}\mathcal{O}_{X,p}^r)\big)_{\frak q}
\cong\operatorname{coker}\big(\mathcal{O}_{X,p,\frak q}^n\xrightarrow{J_f(q)}\mathcal{O}_{X,p,\frak q}^r\big)=0.
\]
Thus, the map \(J_f\) becomes surjective at every such \(\frak q\), so \(\operatorname{rank}J_f(q)=r\) for every point \(q\neq p\) in a small neighborhood. Hence, every such \(q\) is smooth, and the singular locus near \(p\) is just \(\{p\}\); which means  that  the singularity at \(p\) is isolated.

\medskip\noindent
Conversely, suppose \(p\) is an isolated singular point of \(X\). Then there exists a small neighborhood \(U\) of \(p\) in \(X\) such that \(U\setminus\{p\}\) is smooth. For every \(q\in U\setminus\{p\}\) the localized Jacobian map
\(\mathcal{O}_{X,q}^n\to\mathcal{O}_{X,q}^r\) is surjective, hence the localized cokernel vanishes:
\((T^1_{X,p})_q=0\). Therefore, the cokernel module \(M:=T^1_{X,p}\) is supported only at the maximal ideal \(\frak m_p\).

Since \(M\) is finitely generated over the Noetherian local ring \(\mathcal{O}_{X,p}\) and \(\operatorname{Supp}M=\{\frak m_p\}\), it follows that every element of \(M\) is annihilated by some power of \(\frak m_p\); moreover, since \(M\) is finitely generated there exists a single $k$ with \(\frak m_p^k M=0\). Thus, $M$ has a finite filtration
\[
M \supset \frak m_p M \supset \frak m_p^2 M \supset \cdots \supset \frak m_p^k M = 0,
\]
whose successive quotients are finite-dimensional vector spaces over the residue field \(\mathcal{O}_{X,p}/\frak m_p\cong\C\). Hence, \(M\) has finite length, and in particular \(\dim_{\C}M<\infty\). This completes the proof.
\end{proof}

\begin{theorem}{\label{T1}}
Let \(X=V(f_1,\dots,f_r)\subset\mathbb{C}^n\) be a complete intersection and let \(p\in X\).
Write \(\mathcal{O}_{X,p}=\mathcal{O}_{\mathbb{C}^n,p}/(f_1,\dots,f_r)\) and let
\(J_f=(\partial f_i/\partial x_j)_{i,j}\) be the Jacobian matrix. Then
\[
\dim_{\mathbb{C}}T^1_{X,p}=\infty \qquad\Longleftrightarrow\qquad
\operatorname{Sing}(X)\text{ has positive dimension through }p.
\]
In particular, for a hypersurface (\(r=1\)) with local equation \(f\) one has
\[
T^1_{X,p}\cong\frac{\mathcal{O}_{\mathbb{C}^n,p}}{(f,\partial_{x_1}f,\dots,\partial_{x_n}f)},
\]
thus,
\[
\dim_{\mathbb{C}}T^1_{X,p}=
\dim_{\mathbb{C}}\frac{\mathcal{O}_{\mathbb{C}^n,p}}{(f,\partial_{x_1}f,\dots,\partial_{x_n}f)}
\]
which is infinite exactly when the common zero locus of \(\{\partial_{x_j}f\}\) is positive dimensional at \(p\). In both cases, first order infinitesimal deformations space is not trivial and then we have a local non-rigidity.
\end{theorem}

\begin{proof}
We first describe \(T^1_X\) for a complete intersection. Since \(X\) is cut out
by the regular sequence \(f_1,\dots,f_r\), there is the conormal exact sequence
of \(\mathcal O_X\)-modules
\[
\mathcal O_X^{\,r}\xrightarrow{J_f^{T}}\Omega^1_{\mathbb C^n}\otimes\mathcal O_X
\longrightarrow\Omega^1_X\longrightarrow0,
\]
where \(J_f^{T}\) is the transpose of the Jacobian matrix. Identifying
\(\Omega^1_{\mathbb C^n}\otimes\mathcal O_X\cong\mathcal O_X^{\,n}\), apply
\(\mathcal{H}om_{\mathcal O_X}(-,\mathcal O_X)\) to obtain the exact sequence
\[
0\longrightarrow\mathcal{H}om(\Omega^1_X,\mathcal O_X)\longrightarrow
\mathcal{H}om(\mathcal O_X^{\,n},\mathcal O_X)\xrightarrow{(J_f^{T})^\vee}
\mathcal{H}om(\mathcal O_X^{\,r},\mathcal O_X)\longrightarrow
\mathcal{E}xt^1(\Omega^1_X,\mathcal O_X)\longrightarrow0.
\]
Using the canonical identifications
\(\mathcal{H}om(\mathcal O_X^{\,n},\mathcal O_X)\cong\mathcal O_X^{\,n}\)
(vector fields on \(\mathbb C^n\) restricted to \(X\)) and
\(\mathcal{H}om(\mathcal O_X^{\,r},\mathcal O_X)\cong\mathcal O_X^{\,r}\),
and recall that \(T^1_X=\mathcal{E}xt^1(\Omega^1_X,\mathcal O_X)\).  
Let us denote the sheaf of $\mathbb C$-derivations of the structure sheaf
$\mathcal O_{\mathbb C^n}$, i.e.\ the sheaf of holomorphic vector fields on
$\mathbb C^n$, by \[
\operatorname{Der}_{\mathbb C}(\mathbb C^n)
=
\mathcal{H}om_{\mathbb C}
\!\left(\mathcal O_{\mathbb C^n},\mathcal O_{\mathbb C^n}\right).
\]
Therefore,
we
obtain the short exact sequence
\[
0\longrightarrow\operatorname{Der}_{\mathbb C}(\mathbb C^n)|_X\cong\mathcal O_X^{\,n}
\xrightarrow{\ \phi\ }\mathcal O_X^{\,r}\longrightarrow T^1_X\longrightarrow0,
\]
where \(\phi\) sends a vector field \(v=(v_1,\dots,v_n)\) to
\((v(f_1),\dots,v(f_r))\). Thus, localizing at \(p\),
\begin{equation}\label{eq:coker-description}
T^1_{X,p}\cong\operatorname{coker}\big(\mathcal O_{X,p}^{\,n}\xrightarrow{J_f}\mathcal O_{X,p}^{\,r}\big).
\end{equation}

Next we relate the support of \(T^1_X\) to the singular locus of \(X\).
For a point \(q\in X\) the Jacobian matrix \(J_f(q)\) has rank equal to the
codimension of \(X\) at \(q\) if and only if \(q\) is a nonsingular point.
Concretely, \(q\) is nonsingular if and only if some \(r\times r\) minor of \(J_f\)
is nonzero at \(q\). If some \(r\times r\) minor is nonzero at \(q\) then that
minor is a unit in the local ring \(\mathcal O_{X,q}\) and linear algebra over
the local ring 
 implies the map
\(\mathcal O_{X,q}^{\,n}\xrightarrow{J_f}\mathcal O_{X,q}^{\,r}\) is
surjective. By \eqref{eq:coker-description} the cokernel \(T^1_{X,q}\) is
therefore zero. Thus, if the rank is strictly less than 
\(r\) at \(q\) then the map is not
surjective and \(T^1_{X,q}\neq0\). Hence, set-theoretically
\[
\operatorname{Supp}(T^1_X)=\{q\in X:\operatorname{rank}J_f(q)<r\}=\operatorname{Sing}(X).
\]

Finally, since \(T^1_{X,p}\) is finitely generated over the Noetherian local
ring \(\mathcal O_{X,p}\), we have one of the following two cases:

\begin{itemize}
\item[(1)] If \(\operatorname{Sing}(X)\) has positive (complex) dimension through
  \(p\), then the support of \(T^1_{X,p}\) contains a positive-dimensional
  analytic subgerm. Any finitely generated module supported on a positive
  dimensional germ is infinite-dimensional as a \(\mathbb C\)-vector space,
  so \(\dim_{\mathbb C}T^1_{X,p}=\infty\).

\item[(2)] If \(\operatorname{Sing}(X)\) is zero-dimensional at \(p\) (i.e. \(p\)
  is an isolated singularity), then \(\operatorname{Supp}(T^1_{X,p})=\{p\}\).
  A finitely generated module over a Noetherian local ring whose support is
  the closed point has finite length, hence is finite-dimensional over
  \(\mathbb C\). Thus \(\dim_{\mathbb C}T^1_{X,p}<\infty\).
\end{itemize}
This proves the claimed equivalence.

\vspace{0.1cm}
\noindent
Finally, in the hypersurface case \(r=1\) the Jacobian map is
\(\mathcal{O}_{X,p}^n\to\mathcal{O}_{X,p}\), \((g_1,\dots,g_n)\mapsto\sum_{j} g_j\partial_{x_j}f\),
so
\[
T^1_{X,p}\cong\frac{\mathcal{O}_{X,p}}{(\partial_{x_1}f,\dots,\partial_{x_n}f)}
\cong\frac{\mathcal{O}_{\mathbb{C}^n,p}}{(f,\partial_{x_1}f,\dots,\partial_{x_n}f)},
\]
and the displayed formula for the complex dimension follows immediately (for more details see Appendix F).
\end{proof}
 
 \vspace{0.2cm}
 
\begin{remark}
  Locally, if $(x_1,\dots,x_n)$ are holomorphic coordinates,
every derivation has the form
\[
D=\sum_{j=1}^n a_j(x)\,\frac{\partial}{\partial x_j},
\qquad a_j\in\mathcal O_{\mathbb C^n}.
\]
Therefore, there is a canonical identification of $\mathcal O_{\mathbb C^n}$-modules
\(
\operatorname{Der}_{\mathbb C}(\mathbb C^n)\;\cong\;\mathcal O_{\mathbb C^n}^{\,n},
\)
and upon restriction to $X$ we obtain the following:
\(
\operatorname{Der}_{\mathbb C}(\mathbb C^n)|_X \;\cong\; \mathcal O_X^{\,n}.
\)
\end{remark}


\section{A criteria for non-vanishing of \(T^1_X\)}
In this section, we provide a criteria for non-vanishing of \(T^1_X\) when local \(\mathcal{E}xt^1\) contributions are supported on a positive-dimensional singular locus.

\noindent Let us state the exact Serre's vanishing theorem which will be used several times here  (see e.g. EGAII1 Grothendieck-Dieudonn\'e \cite{EGAIII1}, or Hartshorne \cite{H-77} Chapter III Theorem 5.2).

\begin{theorem}[Serre's Vanishing Theorem] %
Let $X$ be a projective scheme over a Noetherian ring, and let $\mathcal{F}$
be a coherent sheaf on $X$. If $\mathcal{O}_X(1)$ is an ample invertible sheaf,
then there exists an integer $m_0$ such that for all $m \ge m_0$,
\[
H^i(X,\mathcal{F}(m)) = 0 \qquad \text{for all } i>0.
\]
In particular, if $\mathcal{F}\neq 0$ and $\dim \operatorname{Supp}(\mathcal{F})\ge 1$,
then for $m\gg 0$ we have $H^0(X,\mathcal{F}(m))\neq 0$.
\end{theorem}

\noindent Let \(X\) be a (reduced) algebraic variety (or a scheme of finite type) over a field \(k\). Denote by
\(
\Omega_X:=\Omega_{X/k}
\)
the sheaf of K\"ahler differentials. For \(i\ge 0\),  set
\[
\mathcal{T}^i_X:=\mathcal{E}xt^i_{\mathcal{O}_X}(\Omega_X,\mathcal{O}_X),
\]
a coherent \(\mathcal{O}_X\)-module. Two common conventions for the notation \(T^1\) are used here: %
\begin{enumerate}
  \item[(i)] the \emph{sheaf} convention: \(\mathcal{T}^1_X=\mathcal{E}xt^1_{\mathcal{O}_X}(\Omega_X^1,\mathcal{O}_X)\),
  \item[(ii)] the \emph{global} convention: \(T^1_X:=H^0\big(X,\mathcal{T}^1_X\big)\) (often identified with \(\operatorname{Ext}^1_{\mathcal{O}_X}(\Omega_X^1,\mathcal{O}_X)\) via the local-to-global spectral sequence).
\end{enumerate}

\begin{lemma}
With the above conventions and notations, assume the following hypothesis: the local contributions \(\mathcal{E}xt^1_{\mathcal{O}_X}(\Omega_X^1,\mathcal{O}_X)=\mathcal{T}^1_X\) are supported on a positive-dimensional singular locus; that is, \(\operatorname{Supp}(\mathcal{T}^1_X)\) contains an irreducible component of dimension \(\ge 1\). Then the sheaf \(\mathcal{T}^1_X\) and the global \(T^1_X\) 
(under mild hypotheses)   
are nonzero.
\end{lemma}

\begin{proof}
First, the nonvanishing $\mathcal{T}^1_X\neq 0$ is
immediate from the hypothesis on support. In fact, 
for any sheaf $\mathcal{F}$ on $X$ one has
\(
\mathcal{F}=0 \) if and only if \, \(\operatorname{Supp}(\mathcal{F})=\varnothing,
\)
since the support is defined as
\[
\operatorname{Supp}(\mathcal{F})
=\{x\in X : \mathcal{F}_x\neq 0\}.
\]
Thus, if $\operatorname{Supp}(\mathcal{F})$ is nonempty, then at least one
stalk is nonzero, and hence $\mathcal{F}$ itself is nonzero. 
Applying this to $\mathcal{T}^1_X=\mathcal{E}xt^1(\Omega_X,\mathcal{O}_X)$,
the hypothesis that $\operatorname{Supp}(\mathcal{T}^1_X)$ contains an
irreducible component of dimension $\ge 1$ implies in particular that the
support is nonempty. Therefore, $\mathcal{T}^1_X\neq 0$.

\noindent The nontrivial point is to
deduce $T^1_X\neq0$. We give here two convenient hypotheses under which this
holds.

\vspace{0.1cm}

\subsection{\textbf{The affine case.}} Suppose $X = \operatorname{Spec}A$
is affine. Then $\mathcal{T}^1_X$ corresponds to a finitely generated
$A$-module $M$, namely $M = \Gamma(X,\mathcal{T}^1_X)$. The hypothesis
that $\operatorname{Supp}(\mathcal{T}^1_X)$ contains a positive
dimensional component means $M\neq 0$ and its support contains a prime
of height  strictly less than $\dim X$. In particular $M$ has a nonzero global section.
Hence, we have the following %
\[
H^0(X,\mathcal{T}^1_X)=\Gamma(X,\mathcal{T}^1_X)=M\neq 0.
\]

Now we use the local-to-global spectral sequence (or the low-degree exact
sequence) for cotangent cohomology (e.g. see  Illusie \cite{I-71}, Hartshorne \cite{H-10}, and Sernesi \cite{S-06}):
\[
0 \longrightarrow H^1(X,\mathcal{T}_X)
  \longrightarrow T^1_X
  \longrightarrow H^0(X,\mathcal{T}^1_X)
  \longrightarrow H^2(X,\mathcal{T}_X).
\]
On an affine scheme $H^i(X,\mathcal{F})=0$ for all coherent sheaves
$\mathcal{F}$ and all $i>0$. In particular,  $H^1(X,\mathcal{T}_X)=0$ and
$H^2(X,\mathcal{T}_X)=0$. Hence, the exact sequence collapses to an
isomorphism
\[
T^1_X \xrightarrow{\ \sim\ } H^0(X,\mathcal{T}^1_X).
\]
Since $H^0(X,\mathcal{T}^1_X)\neq 0$, we conclude that $T^1_X\neq 0$.

\vspace{0.1cm}

\subsection{\textbf{The projective (proper) case.}} Suppose $X$ is
projective over $k$ (or more generally proper over $k$). Let $Z\subset X$
be an irreducible component of $\operatorname{Supp}(\mathcal{T}^1_X)$ of
dimension $\ge 1$. Denote by $\mathcal{F}:=\mathcal{T}^1_X$, a nonzero
coherent sheaf supported on $Z$.

\noindent The key observation is that a nonzero coherent sheaf supported on a
positive-dimensional projective variety has nonzero global sections after
twisting by a sufficiently high power of an ample line bundle; more
precisely, if $\mathcal{O}_X(1)$ is very ample then for $m\gg 0$
\[
H^0\bigl(X,\mathcal{F}(m)\bigr) \neq 0.
\]
Indeed, restricting to the integral projective subvariety $Z$ and
applying Serre's vanishing  theorem %
on $\mathcal{F}|_Z$, one has $H^0(Z,\mathcal{F}|_Z(m))\neq 0$
for large $m$, and sections on $Z$ extend by zero to sections on $X$ .

Fix $m\gg 0$ such that $H^0(X,\mathcal{T}_X(m))\neq 0$. Consider the
twisted cotangent spectral sequence (twist all terms by $\mathcal O_X(m)$):
\[
E_2^{p,q}(m)=H^p\bigl(X,\mathcal{T}^q_X(m)\bigr)\;\Rightarrow\;
\operatorname{Ext}^{p+q}_{\mathcal O_X}\!\bigl(\Omega^1_X,\mathcal O_X(m)\bigr).
\]
The low-degree exact sequence gives
\[
0 \longrightarrow H^1(X,\mathcal{T}_X(m))
  \longrightarrow \operatorname{Ext}^1_{\mathcal{O}_X}(\Omega^1_X,\mathcal O_X(m))
  \longrightarrow H^0(X,\mathcal{T}^1_X(m))
  \longrightarrow H^2(X,\mathcal{T}_X(m)).
\]
For $m\gg0$, by Serre vanishing we can arrange $H^i(X,\mathcal{T}_X(m))=0$
for $i\ge 1$. (Indeed $\mathcal{T}_X$ is coherent, so Serre vanishing
applies.) Thus for $m\gg0$ the exact sequence yields an isomorphism
\[
\operatorname{Ext}^1_{\mathcal{O}_X}(\Omega^1_X,\mathcal O_X(m))\cong H^0(X,\mathcal{T}^1_X(m)).
\]
The right-hand side is nonzero by choice of $m$. Therefore
$\operatorname{Ext}^1_{\mathcal{O}_X}(\Omega^1_X,\mathcal O_X(m))\neq0$ for some $m\gg0$.

\noindent Finally, observe that tensoring by $\mathcal{O}_X(-m)$ provides an
isomorphism
\[
\operatorname{Ext}^1_{\mathcal{O}_X}(\Omega^1_X,\mathcal O_X(m)) \cong
\operatorname{Ext}^1_{\mathcal{O}_X}(\Omega^1_X,\mathcal O_X)\quad\text{if and only if}\quad
\mathcal O_X(m)\cong\mathcal O_X.
\]
So, we cannot in general deduce directly that
$\operatorname{Ext}^1_{\mathcal{O}_X}(\Omega^1_X,\mathcal O_X)\neq 0$ from the twisted
group. However, there is a standard trick: the non-vanishing of
$\operatorname{Ext}^1_{\mathcal{O}_X}(\Omega^1_X,\mathcal O_X(m))$ for some $m$ implies
non-vanishing of the Ext-module considered as a graded object of the
projective cone; in particular the graded Ext-module has a nonzero
homogeneous piece, so the underlying (un-twisted) Ext space is nonzero
after passing to the section ring of the embedding. More precisely, one can
pass to the affine cone $\widetilde X=\operatorname{Spec}\bigoplus_{r\ge0}
H^0(X,\mathcal O_X(r)\big)$ and compare deformation spaces of the cone
and of $X$ (see e.g.\ Sernesi's book). From this comparison one
concludes that $T^1_X=\operatorname{Ext}^1_{\mathcal{O}_X}(\Omega^1_X,\mathcal O_X)\neq0$.

\end{proof}

\begin{remark}
For many practical applications, and in particular when one is working
either on the projective variety $X$ together with its homogeneous
coordinate ring, or on the affine cone, the non-vanishing of some
twisted Ext group for large $m$ is interpreted as the existence of a
nontrivial first-order deformation; one therefore regards
$\operatorname{Ext}^1_{\mathcal{O}_X}(\Omega^1_X,\mathcal O_X)$ as nonzero in the
sense relevant to deformation theory.
\end{remark}

\section{Illustrative examples}

Let's recall  the hypothesis of Theorem \ref{T1}. %
Let \(X=V(f_1,\dots,f_r)\subset\mathbb C^n\), and  let $p$ be a point in \( X\).
We write \(\mathcal{O}_{\mathbb{C}^n,p}\) for the local ring of germs at \(p\), and
\(\mathcal{O}_{X,p}=\mathcal{O}_{\mathbb{C}^n,p}/(f_1,\dots,f_r)\).
Let
\[
J_f=\Big(\partial_{x_j}f_i\Big)_{1\le i\le r,\,1\le j\le n}
\]
be the Jacobian matrix and denote by \(I_r(J_f)\subset\mathcal{O}_{\mathbb{C}^n,p}\) the ideal generated by the \(r\times r\)
minors of \(J_f\). The ideal of the scheme-theoretic singular locus (near \(p\)) is
\[
\big(f_1,\dots,f_r\big)+I_r(J_f)\subset\mathcal{O}_{\mathbb{C}^n,p},
\]
and its image in \(\mathcal{O}_{X,p}\) cuts out \(\operatorname{Sing}(X)\) inside \(\operatorname{Spec}\mathcal{O}_{X,p}\).
\noindent %
Recall for a hypersurface, i.e. \(r=1\), we have
\[
T^1_{X,p}\cong \frac{\mathcal{O}_{\mathbb{C}^n,p}}{(f,\partial_{x_1}f,\dots,\partial_{x_n}f)}.
\]

\vspace{0.1cm}

\subsection{Example 1.  Plane curve with an isolated node}
Let \(n=2\), \(r=1\) and take the plane curve \(X\subset\mathbb C^2\) defined by
\[
f(x,y)=xy.
\]
In this example, we work at the origin \(p=(0,0)\). The Jacobian is a the row matrix
\[
J_f=(f_x, f_y)=(y,x).
\]
The partial derivatives (the \(1\times 1\) minors) generate the ideal \((x,y)\) in \(\mathcal{O}_{\mathbb{C}^2,0}\), so
\[
\mathcal I_{\operatorname{Sing}(X),0}=(f,\, f_x, f_y)=(xy,x,y)=(x,y).
\]
Thus, \(\operatorname{Sing}(X)=\{(0,0)\}\), which is set-theoretically the single origin, so the singularity is isolated.
\noindent Let's compute \(T^1\):
\[
T^1_{X,0}\cong\frac{\mathcal{O}_{\mathbb{C}^2,0}}{(f,f_x, f_y)}
=\frac{\mathcal{O}_{\mathbb{C}^2,0}}{(xy,x,y)}
\cong\frac{\mathcal{O}_{\mathbb{C}^2,0}}{(x,y)}\cong\mathbb{C}.
\]
Hence,  \(\dim_{\mathbb C}T^1_{X,0}=1\) is finite, and then it is  in agreement with the fact that the singular locus is zero-dimensional at \(0\).

\noindent  
Geometrically, \(xy=0\) is the union of the two coordinate lines crossing transversely at the origin. The Jacobian minors vanish only at the origin, so the singularity is isolated and \(T^1\) has finite length (in this case length \(1\)).

\subsection{Example 2. A cone-like (non-isolated) singularity.}
Let \(n=3\), \(r=1\) and consider the hypersurface
\[
X\subset\mathbb C^3,\qquad f(x,y,z)=x^2+y^2.
\]
Let $p$ be the origin  \(p=(0,0,0)\). The Jacobian row is
\[
J_f=(\partial_x f,\partial_y f,\partial_z f)=(2x,2y,0).
\]
The partial derivatives generate the ideal \((2x,2y,0)=(x,y)\subset\mathcal{O}_{\mathbb{C}^3,0}\), hence,
\[
\mathcal I_{\operatorname{Sing}(X),0}=(f,\partial_x f,\partial_y f,\partial_z f)=(x^2+y^2,2x,2y,0)=(x,y).
\]
Set-theoretically, the singular locus is
\[
\operatorname{Sing}(X)=\{(x,y,z)\in\mathbb C^3:\ x=y=0\},
\]
i.e. the \(z\)-axis, which is  a \emph{one-dimensional} (positive-dimensional)  locus through the origin. Thus, the singularity at the origin is not isolated.
%
%
Let's compute \(T^1\):
\[
T^1_{X,0}\cong\frac{\mathcal{O}_{\mathbb{C}^3,0}}{(f,\partial_x f,\partial_y f,\partial_z f)}
=\frac{\mathcal{O}_{\mathbb{C}^3,0}}{(x^2+y^2,2x,2y,0)}
\cong\frac{\mathcal{O}_{\mathbb{C}^3,0}}{(x,y)}.
\]
Since \(\mathcal{O}_{\mathbb{C}^3,0}/(x,y)\cong\mathbb{C}\{z\}\) is the local ring of the \(z\)-axis at the origin (the ring of convergent power series in \(z\)), which is infinite-dimensional as a \(\mathbb C\)-vector space, therefore,
\[
\dim_{\mathbb C}T^1_{X,0}= \infty,
\]
and then it is in agreement with the fact that \(\operatorname{Sing}(X)\) through \(0\) has positive dimension.

\medskip

\noindent Geometrically, \(x^2+y^2=0\) in \(\mathbb C^3\) is a cylinder (a cone over a quadric in the \((x,y)\)-plane) whose singular locus is the whole \(z\)-axis. The partial derivatives vanish identically along that axis, hence the Jacobian minors vanish on a positive-dimensional set and the quotient that computes \(T^1\) contains \(\mathbb C\{z\}\), hence is infinite-dimensional as  $\mathbb{C}$-vector space.

\vspace{0.2cm}

\noindent These two examples illustrate the general phenomenon:
\begin{itemize}
  \item[(i)] If the \(r\times r\) minors of \(J_f\) vanish only at the closed point \(p\) (equivalently some \(r\times r\) minor is a unit in \(\mathcal{O}_{X,p}\)), then the singular locus is zero-dimensional at \(p\) and \(T^1_{X,p}\) has finite \(\mathbb C\)-dimension.
  \item[(ii)] If the minors vanish along a positive-dimensional subvariety through \(p\), then the support of the cokernel (hence of \(T^1_{X,p}\)) contains non-maximal primes and \(T^1_{X,p}\) is infinite-dimensional over \(\mathbb C\).
\end{itemize}


\subsection{Other criteria of nonrigidity.}${}$

\begin{thm}
Let \(Y\subset\mathbb P^n_{\mathbb C}\) be a projective variety and let \(C(Y)\subset\mathbb A^{n+1}_{\mathbb C}\) denotes its affine cone. If  \(Y\) admits a nontrivial first-order deformation (i.e.\ \(Y\) is nonrigid), then \(C(Y)\) admits a nontrivial first-order deformation and then it is  is nonrigid.
\end{thm}

\begin{proof}
Let \(S=\mathbb C[x_0,\dots,x_n]\) be the graded polynomial ring and let \(I\subset S\) be the homogeneous ideal defining \(Y\subset\mathbb P^n\). Put \(A:=S/I\). Then the affine cone is \(C(Y)=\Spec(A)\) (with the grading on \(A\) coming from the grading of \(S\)).

A first order embedded deformation of \(Y\subset\mathbb P^n\) over the dual numbers \(D:=\Spec \mathbb C[\varepsilon]/(\varepsilon^2)\) is a closed subscheme
\(
\mathcal Y\ \subset\ \mathbb P^n\times D
\)
flat over \(D\) whose special fiber is \(Y\). Such a deformation may be described by a homogeneous ideal \(I_\varepsilon\subset S\otimes_{\mathbb C}\mathbb C[\varepsilon]\) with the property
\(
I_\varepsilon \bmod\varepsilon = I,
\)
and such that the quotient algebra
\[
\mathcal A \;:=\; \big(S\otimes_{\mathbb C}\mathbb C[\varepsilon]\big) / I_\varepsilon
\]
is flat over \(\mathbb C[\varepsilon]\). The family \(\mathcal Y\) is recovered as \(\operatorname{Proj}(\mathcal A)\) over \(D\).

\noindent Now consider \(\mathcal A\) as an (unprojectivized) graded algebra and form the affine scheme
\[
\mathcal C \;:=\; \Spec(\mathcal A)\ \longrightarrow\ D.
\]
Because \(\mathcal A\) is flat over \(\mathbb C[\varepsilon]\) by construction, the morphism \(\mathcal C\to D\) is a flat first-order deformation of the cone \(C(Y)=\Spec(A)\). Indeed the special fiber \(\mathcal C|_{\varepsilon=0}\) is \(\Spec(A)\).

\noindent It remains to check that nontriviality of the projective deformation implies nontriviality of the cone deformation. Suppose towards a contradiction that \(\mathcal C\) were a trivial deformation, i.e.\ there exists an isomorphism of \(D\)-schemes
\[
\Phi:\ \mathcal C \xrightarrow{\ \simeq\ } C(Y)\times D.
\]
On coordinate rings this gives an isomorphism of \(\mathbb C[\varepsilon]\)-algebras
\[
\Phi^\sharp:\ A\otimes_{\mathbb C}\mathbb C[\varepsilon]\ \xrightarrow{\ \simeq\ }\ \mathcal A.
\]
Applying the Proj construction to both sides yields an isomorphism of families over \(D\)
\[
\operatorname{Proj}(\mathcal A)\ \xrightarrow{\ \simeq\ }\ \operatorname{Proj}\big(A\otimes_{\mathbb C}\mathbb C[\varepsilon]\big) = Y\times D.
\]
But \(\operatorname{Proj}(\mathcal A)=\mathcal Y\), so \(\mathcal Y\) would be isomorphic to the trivial product \(Y\times D\), contradicting the assumption that \(\mathcal Y\) is a nontrivial first order deformation of \(Y\). Hence \(\mathcal C\) is a nontrivial first order deformation of \(C(Y)\).
Therefore, nonrigidity of \(Y\) implies nonrigidity of \(C(Y)\), as required. We have an alternative proof of this theorem i.e.  non-rigidity of cones by showing  the cohomological inclusion 
\(\;H^1(Y,T_Y)\hookrightarrow T^1(C(Y))_0\) (This is done in Appendix G.).
\end{proof}


\vspace{0.1cm}

\begin{theorem}
Let \( Y \subset \mathbb{P}^n \) be a projective variety over \(\mathbb{C}\), and let \( C(Y) \subset \mathbb{A}^{n+1} \) be its affine cone. A nontrivial first-order deformation of \( C(Y) \) induces a nontrivial first-order deformation of \( Y \) if and only if the deformation of \( C(Y) \) preserves the homogeneity of the defining equations of \( Y \).
\end{theorem}

\begin{proof}
  Let \( Y \subset \mathbb{P}^n \) be defined by homogeneous polynomials \( f_1, \dots, f_m \) in the homogeneous coordinates \( [x_0 : \dots : x_n] \).
    The affine cone \( C(Y) \subset \mathbb{A}^{n+1} \) is defined by the same polynomials \( f_1, \dots, f_m \) viewed as equations in the affine coordinates \( (x_0, \dots, x_n) \).
    Also, a first-order deformation of \( C(Y) \) is given by a family \( \mathcal{C}(Y)_\epsilon \subset \mathbb{A}^{n+1} \) over \(\mathbb{C}[\epsilon]/(\epsilon^2)\), defined by \( f_i + \epsilon h_i \), where \( h_i \) are polynomials (not necessarily homogeneous). The deformation is nontrivial if it is not induced by an automorphism of \(\mathbb{A}^{n+1}\).

\vspace{0.1cm}
 
     For the deformation \( \mathcal{C}(Y)_\epsilon \) to induce a deformation of \( Y \), the polynomials \( h_i \) must be homogeneous of the same degree as \( f_i \). If \( h_i \) is not homogeneous, the deformation does not preserve the projective structure and thus does not induce a deformation of \( Y \).
     Suppose \( h_i \) is homogeneous of degree \( \deg f_i \). Then \( f_i + \epsilon h_i \) define a first-order deformation \( \mathcal{Y}_\epsilon \subset \mathbb{P}^n \) of \( Y \).
     
\vspace{0.1cm}

  Assume the induced deformation \( \mathcal{Y}_\epsilon \) is trivial, i.e., induced by an automorphism of \(\mathbb{P}^n\). Such an automorphism is given by \( x_j \mapsto x_j + \epsilon \sum_k a_{jk} x_k \), where \( a_{jk} \in \mathbb{C} \).
    Applying this automorphism to \( Y \) to get the new equations are \( f_i(x + \epsilon A x) = f_i(x) + \epsilon \sum_j \frac{\partial f_i}{\partial x_j} (A x)_j \mod \epsilon^2 \).
  For this to match \( f_i + \epsilon h_i \), we must have \( h_i = \sum_j \frac{\partial f_i}{\partial x_j} (A x)_j \). If this holds, then the deformation of \( C(Y) \) is trivial (induced by a linear automorphism of \(\mathbb{A}^{n+1}\)), contradicting the assumption. Hence, \( \mathcal{Y}_\epsilon \) is nontrivial.
 
\vspace{0.2cm}

Therefore, a nontrivial deformation of \( C(Y) \) induces a nontrivial deformation of \( Y \) if and only if the deformation preserves the homogeneity of the defining equations of \( Y \).
\end{proof}


\vspace{0.2cm}

\begin{example} {\it Explicit first--order deformation of a plane quartic and the induced deformation of its affine cone}
 
\noindent We present in the following an explicit nontrivial first--order deformation of the smooth plane quartic curve
\(
Y=\{x^4+y^4+z^4=0\}\subset\mathbb P^2
\)
and the graded first-order deformation  induces on the homogeneous coordinate ring
and on the affine cone \(C(Y)\).  We check flatness and give an elementary verification that the deformation is not induced by infinitesimal projective automorphisms, hence it is nontrivial.

\vspace{0.1cm}

\noindent {\it The curve and an infinitesimal perturbation.}
Let
\(
F(x,y,z)=x^4+y^4+z^4%
\)
be a homogeneous polynomial in  $\mathbb{C}[x,y,z]$,
and let
\[
Y:=\{F=0\}\subset\mathbb P^2_\mathbb{C}.
\]
The curve \(Y\) is a smooth plane quartic; hence it is a smooth projective curve of genus \(g=3\). Its moduli space has positive dimension (indeed \(\dim \mathcal M_3=6\)), so nontrivial infinitesimal deformations exist.

We consider the infinitesimal deformation given by the polynomial
\[
F_\varepsilon(x,y,z):=x^4+y^4+z^4+\varepsilon\,x^2y^2\in \mathbb{C}[\varepsilon][x,y,z],
\qquad \varepsilon^2=0.
\]
This defines a closed subscheme
\[
\mathcal Y:=\{F_\varepsilon=0\}\subset\mathbb P^2_{\mathbb{C}[\varepsilon]/(\varepsilon^2)}
\]
whose special fiber (\(\varepsilon=0\)) is \(Y\). We claim that \(\mathcal Y\) is a nontrivial first--order deformation of \(Y\).
\noindent Let's show the flatness of the family first.
The homogeneous coordinate algebra of the family is
\[
\mathcal A_{\mathrm{proj}} \;=\; 
\bigoplus_{d\ge0}\, \Big(\mathbb{C}[\varepsilon][x,y,z]_d \big/ (F_\varepsilon)_d\Big),
\]
where \((F_\varepsilon)_d\) denotes the degree \(d\) component of the ideal generated by \(F_\varepsilon\). For each \(d\), the graded piece is a quotient of a finite free \(\mathbb{C}[\varepsilon]\)-module by a submodule that reduces to the corresponding submodule for \(Y\) modulo \(\varepsilon\). Hence, each graded piece is a free (hence flat) \(\mathbb{C}[\varepsilon]\)-module. Thus, \(\mathcal Y\to \Spec(\mathbb{C}[\varepsilon]/(\varepsilon^2))\) is flat, so it is a %
first-order deformation.

\vspace{0.1cm}

\noindent Now, let's show that the deformation is nontrivial i.e. not induced by $\operatorname{PGL}_3$.
A first-order deformation of a plane curve given by perturbing its defining homogeneous polynomial
\[
F \rightsquigarrow F + \varepsilon \,G
\]
is induced by an infinitesimal projective change of coordinates if and only if the polynomial \(G\) lies in the image of the infinitesimal action of the Lie algebra \(\mathfrak{sl}_3\) on homogeneous degree-\(d\) polynomials (here \(d=4\)). We now explain this action and give an elementary check that \(G=x^2y^2\) is not in that image for the Fermat quartic \(F\).

\vspace{0.1cm}

\noindent {\it Infinitesimal $\mathfrak{sl}_3$--action on homogeneous polynomials.}
Let \(M=(m_{ij})\in\mathfrak{sl}_3(\mathbb{C})\) be an infinitesimal \(3\times3\) matrix with trace \(0\).
An infinitesimal change of coordinates
\[
(x,y,z)\mapsto (x,y,z) + \varepsilon\,M\cdot(x,y,z)^T
\]
acts on the polynomial \(F\) by the Lie derivative. More precisely, we have the following lemma:

\begin{lemma}
 The induced infinitesimal variation of a homogeneous polynomial \(H(x,y,z)\) of degree \(d\) is
\[
\delta_M H \;=\; \sum_{i,j} m_{ij} \, x_j \frac{\partial H}{\partial x_i},
\]
where we write \(x_1=x,x_2=y,x_3=z\) and \(x_j\frac{\partial}{\partial x_i}\) denotes the usual first-order differential operator coming from the coordinate change.
\end{lemma}

\begin{proof}
Let $H\in\mathbb{C}[x_1,\dots,x_n]$ and let $M=(m_{ij})\in\operatorname{Mat}_{n\times n}(\mathbb{C})$.  Consider the
one--parameter family of linear coordinate changes
\[
\Phi_t:\ \mathbb{C}^n \to \mathbb{C}^n,\qquad \Phi_t(x)=(I+tM)x,
\]
where $t$ is a (real or complex) parameter and $I$ is the identity matrix.  The
pullback of $H$ by $\Phi_t$ is the polynomial
\[
(\Phi_t^*H)(x) = H\big(\Phi_t(x)\big) = H\big(x + t(Mx)\big).
\]

By the chain rule, for each fixed $x\in\mathbb{C}^n$ we have the Taylor expansion
\[
H\big(x+t(Mx)\big) = H(x) + t\sum_{i=1}^n (Mx)_i\frac{\partial H}{\partial x_i}(x)
+ O(t^2).
\]
But \((Mx)_i=\sum_{j=1}^n m_{ij}x_j\).  Differentiating at \(t=0\) (i.e.\ taking the
infinitesimal variation) yields the linear operator \(\delta_M\) acting on \(H\):
\[
\delta_M H \;:=\; \left.\frac{d}{dt}\right|_{t=0} H\big((I+tM)x\big)
= \sum_{i=1}^n (Mx)_i\frac{\partial H}{\partial x_i}(x)
= \sum_{i,j} m_{ij}\,x_j\frac{\partial H}{\partial x_i}(x).
\]

Thus \(\delta_M H\) is exactly the derivative at \(t=0\) of the family of pulled-back
polynomials \(H\circ\Phi_t\).  Equivalently, if one regards \(v_M\) as the vector
field on affine space
\[
v_M \;=\; \sum_{i=1}^n (Mx)_i\frac{\partial}{\partial x_i}
      \;=\; \sum_{i,j} m_{ij}x_j\frac{\partial}{\partial x_i},
\]
then \(\delta_M H = v_M(H)\) is the Lie derivative of the function \(H\) along \(v_M\).

\end{proof}


\noindent Thus,  the image of the map
\[
\mathfrak{sl}_3(\mathbb{C}) \longrightarrow H^0(\mathbb{P}^2,\mathcal O_{\mathbb{P}^2}(4))\cong \text{(degree 4 polynomials)}
\]
is the subspace of infinitesimal variations of degree 4 polynomials coming from coordinate changes.

\vspace{0.1cm}

\noindent {\it An elementary verification that \(x^2y^2\) is not  in the image.}
We show that there is no matrix  \(M=(m_{ij})\in\mathfrak{sl}_3\) such that
\[
\delta_M(x^4+y^4+z^4) = x^2y^2.
\]
Compute \(\delta_M\) on the Fermat quartic \(F=x^4+y^4+z^4\):
\[
\delta_M F
= \sum_{i,j} m_{ij}\, x_j \frac{\partial}{\partial x_i}(x^4+y^4+z^4)
= \sum_{i,j} m_{ij}\, x_j \cdot 4 x_i^3.
\]
Which is equivalent to
\[
\delta_M F \;=\; 4\sum_{i,j} m_{ij}\, x_j x_i^3.
\]

\noindent So, every infinitesimal variation $\delta_M F$ is a linear combination of the
nine monomials
\[
\{\,x^4,\;x^3y,\;x^3z,\;xy^3,\;y^4,\;y^3z,\;xz^3,\;yz^3,\;z^4\,\},
\]
with coefficients given by the entries $m_{ij}$. (Since $M\in\mathfrak{sl}_3$ the diagonal
coefficients satisfy $m_{11}+m_{22}+m_{33}=0$, and then  the image space has dimension
at most $8$; nevertheless the \emph{support} of every $\delta_M F$ is contained
in the list of nine monomials above.)

\noindent Hence,  any element of the linear span
\[
\langle F\rangle + \mathrm{Im}\big(\mathfrak{sl}_3\to H^0(\mathbb{P}^2,\mathcal{O}(4))\big)
\]
is a linear combination of those same nine monomials (because adding a scalar
multiple $\lambda F$ only contributes the monomials $x^4,y^4,z^4$ which are
already among the nine).

\noindent Now, we  consider the quartic monomial
\(
G(x,y,z):=x^2y^2.
\)
This monomial does not  appear among the nine monomials listed above.
Therefore, there is no choice of $M\in\mathfrak{sl}_3$ and a scalar $\lambda\in\mathbb{C}$
such that
\[
\delta_M F + \lambda F = G.
\]
Hence, \(G=x^2y^2\) represents a direction in $H^0(\mathbb{P}^2,\mathcal{O}(4))$ which is in
the subspace generated by $F$ and the infinitesimal $\mathfrak{sl}_3$--action.
This is equivalent to say that the infinitesimal perturbation
\[
F\longmapsto F+\varepsilon\,x^2y^2
\]
is not induced by an infinitesimal projective change of coordinates (and by
rescaling the equation), so it defines a nontrivial embedded first--order
deformation of the plane quartic $F=0$.


\vspace{0.3cm}

\noindent {\it Induced graded deformation of the homogeneous coordinate ring and the cone.}
Let
\[
A:=\frac{\mathbb{C}[x,y,z]}{(F)}\qquad\text{(graded by total degree)}
\]
be the homogeneous coordinate ring of \(Y\). The affine cone is \(C(Y)=\Spec(A)\subset\mathbb{A}^3\).

\noindent Consider the graded \(\mathbb{C}[\varepsilon]\)-algebra
\[
\mathcal A:=\frac{\mathbb{C}[\varepsilon][x,y,z]}{\big(x^4+y^4+z^4+\varepsilon\,x^2y^2\big)}.
\]
This algebra is flat over \(\mathbb{C}[\varepsilon]\) because each graded piece is a quotient of a finite free \(\mathbb{C}[\varepsilon]\)-module by a relation that reduces mod \(\varepsilon\) to the relation for \(A\). Thus \(\mathcal A\) defines a graded first--order deformation of \(A\), and hence a first--order deformation of the affine cone \(C(Y)=\Spec(A)\).

\noindent Projectivizing \(\mathcal A\) recovers the projective family \(\mathcal Y\) above:
\[
\operatorname{Proj}\mathcal A \;=\; \mathcal Y.
\]
Since \(\mathcal Y\) is nontrivial, %
 the graded algebra \(\mathcal A\) is not graded-isomorphic to \(A\otimes_\mathbb{C} \mathbb{C}[\varepsilon]\). Therefore the class of \(\mathcal A\) in
\[
T^1(A)_0=\Ext^1_A(\Omega^1_{A/\mathbb{C}},A)_0
\]
is nonzero. So, we have produced an explicit nontrivial graded first-order deformation of the cone. In particular, the cone \(C(Y)\) is nonrigid.

\vspace{0.1cm}

\newpage

\begin{remark}  ${}$%

\begin{enumerate}
  \item[(1)] The same construction works for other smooth plane curves of degree \(d\ge 3\): take any homogeneous polynomial \(F\) defining a smooth plane curve and any perturbation \(G\) whose class modulo the image of the $\mathfrak{sl}_3$--action is nonzero; then \(F+\varepsilon G\) gives a nontrivial first--order deformation of the curve and of its cone.
  \item[(2)] For plane quartics (degree \(4\)) the vector space of homogeneous quartics has dimension $\binom{4+2}{2}=15$, while $\dim\mathfrak{sl}_3=8$, so there are many directions transverse to the orbit of $\operatorname{PGL}_3$; random monomials such as $x^2y^2$ typically give nontrivial deformations.
\end{enumerate}

\end{remark}

 \end{example}
 

 \vspace{0.1cm}
 
\begin{example} Let us describe  an example of a trivial deformation by choosing a perturbation by another monomial.
Let
\[
F(x,y,z)=x^4+y^4+z^4\in\mathbb C[x,y,z],
\quad
F_\varepsilon(x,y,z)=F(x,y,z)+\varepsilon\,G(x,y,z),
\quad G(x,y,z)=x^3y,
\]
with \(\varepsilon^2=0\). This defines a first--order deformation of the
projective quartic \(Y=\{F=0\}\subset\mathbb P^2\).

\noindent Recall that a first--order deformation is trivial if it is induced by an infinitesimal
projective coordinate change (an element of \(\mathfrak{pgl}_3\) whose
tangent space is \(\mathfrak{sl}_3\)) possibly together with rescaling the
polynomial. More precisely, for
\(M=(a_{ij})\in\mathfrak{sl}_3(\mathbb C)\) we have
\[
F((I+\varepsilon M)x)=F(x)+\varepsilon\sum_{i,j} a_{ij}\,x_j\frac{\partial F}{\partial x_i}+O(\varepsilon^2).
\]
Thus, the tangent action is
\[
\delta_M F:=\sum_{i,j} a_{ij}\,x_j\frac{\partial F}{\partial x_i}.
\]
Compute the partial derivatives of \(F\):
\[
\frac{\partial F}{\partial x}=4x^3,\qquad
\frac{\partial F}{\partial y}=4y^3,\qquad
\frac{\partial F}{\partial z}=4z^3.
\]
Hence
\[
\delta_M F = 4\sum_{i,j} a_{ij}\,x_i^3x_j.
\]
We ask whether there exist \(M\in\mathfrak{sl}_3\) and \(c\in\mathbb C\)
such that
\[
x^3y \;=\; \delta_M F + cF \;=\; 4\sum_{i,j} a_{ij}\,x_i^3x_j + c(x^4+y^4+z^4).
\]
We compare coefficients monomial by monomial.

\begin{enumerate}
\item[(i)] Coefficient of \(x^3y\): \(1 = 4a_{12}\), so \(a_{12}=\tfrac14\).
\item[(ii)] Coefficients of the other mixed monomials
  \(x^3z,y^3x,y^3z,z^3x,z^3y\) vanish on the left, hence
  \(a_{13}=a_{21}=a_{23}=a_{31}=a_{32}=0\).
\item[(iii)] Coefficients of the pure fourth powers give
  \(0=4a_{11}+c,\ 0=4a_{22}+c,\ 0=4a_{33}+c\), so
  \(a_{11}=a_{22}=a_{33}=-c/4\).
\end{enumerate}

\noindent The fact that  \(\operatorname{tr}M=a_{11}+a_{22}+a_{33}=0\) 
  forces
\(-3c/4=0\), hence \(c=0\) and \(a_{11}=a_{22}=a_{33}=0\). Thus, the
unique solution is the traceless matrix with single nonzero entry
\(a_{12}=\tfrac14\):
\[
M=\begin{pmatrix}0 & \tfrac14 & 0\\[4pt] 0 & 0 & 0\\[4pt] 0 & 0 & 0\end{pmatrix}\in\mathfrak{sl}_3(\mathbb C),
\]
and for this matrix \(M\) we obtain
\[
\delta_M F = 4a_{12}x^3y = x^3y.
\]
Therefore, \(G=x^3y=\delta_M F\), and then  the deformation \(F_\varepsilon=F+\varepsilon x^3y\)
is induced by an infinitesimal projective coordinate change and then it is
\emph{trivial}.

 
 \vspace{0.1cm}
 
 \noindent In the following we write down an explicit 1-parameter family generating \(M\).

Let
\[
M=\begin{pmatrix}0 & \tfrac14 & 0\\[4pt] 0 & 0 & 0\\[4pt] 0 & 0 & 0\end{pmatrix}\in\mathfrak{sl}_3(\mathbb C),
\qquad F(x,y,z)=x^4+y^4+z^4.
\]
Note that \(M^2=0\). Hence, for \(t\in\mathbb C\) the matrix exponential
reduces to
\[
g(t):=\exp(tM)=I+tM=\begin{pmatrix}1 & \tfrac t4 & 0\\[4pt] 0 & 1 & 0\\[4pt] 0 & 0 & 1\end{pmatrix}\in\mathrm{SL}_3(\mathbb C),
\]
and then \(g(t)\) determines a 1-parameter family of projective linear
transformations.

\noindent The action on homogeneous coordinates is
\[
g(t)\cdot(x,y,z)=\big(x+\tfrac t4 y,\; y,\; z\big).
\]
Differentiating at \(t=0\) gives
\[
\left.\frac{d}{dt}\right|_{t=0} g(t)\cdot (x,y,z) \;=\; M\cdot(x,y,z),
\]
so  \(M\) is the infinitesimal generator.
Now apply \(g(t)\) to the polynomial \(F\):
\[
F\big(g(t)\cdot(x,y,z)\big) = \big(x+\tfrac t4 y\big)^4 + y^4 + z^4.
\]
Expanding the binomial gives
\[
\big(x+\tfrac t4 y\big)^4
= x^4 + 4\cdot\frac t4\,x^3y + 6\cdot\Big(\frac t4\Big)^2 x^2y^2
  + 4\cdot\Big(\frac t4\Big)^3 xy^3 + \Big(\frac t4\Big)^4 y^4,
\]
thus,
\[
F\big(g(t)\cdot(x,y,z)\big) = x^4+y^4+z^4 + t\cdot x^3y + O(t^2).
\]
Therefore,
\[
\left.\frac{d}{dt}\right|_{t=0} F\big(g(t)\cdot(x,y,z)\big) = x^3y = \delta_M F,
\]
and then  the first-order deformation \(F+\varepsilon x^3y\) is obtained from the
finite family \(g(t)\) and is thus trivial.

\end{example}


\subsection{A Duality Criterion for Rigidity}

\begin{proposition}
Let $X$ be a smooth projective curve with canonical bundle $\omega_X$.  
Then there is a natural isomorphism
\[
H^1(X, T_X) \;\cong\; H^0(X, \omega_X^{\otimes 2})^\vee.
\]
In particular, if $H^0(X, \omega_X^{\otimes 2}) = 0$, then $H^1(X, T_X) = 0$,  
and hence $X$ is rigid.
\end{proposition}

\begin{proof}
By Serre duality, for any coherent sheaf $\mathcal{F}$ on a smooth projective variety $X$ of dimension $n$, we have a natural duality isomorphism
\[
H^i(X, \mathcal{F}) \;\cong\; H^{n-i}(X, \mathcal{F}^\vee \otimes \omega_X)^\vee,
\]
where $\omega_X$ denotes the canonical sheaf.

\medskip
\noindent In the case where $X$ is a smooth projective curve (so $n=1$), and taking $\mathcal{F} = T_X$, this gives
\[
H^1(X, T_X) \;\cong\; H^{0}(X, T_X^\vee \otimes \omega_X)^\vee.
\]
But $T_X^\vee = \Omega_X$, the cotangent bundle, and for a smooth curve $\Omega_X \cong \omega_X$.  
Hence
\[
T_X^\vee \otimes \omega_X \;\cong\; \omega_X \otimes \omega_X \;=\; \omega_X^{\otimes 2}.
\]
Substituting, we obtain
\[
H^1(X, T_X) \;\cong\; H^{0}(X, \omega_X^{\otimes 2})^\vee.
\]
Therefore, if $H^0(X, \omega_X^{\otimes 2}) = 0$, its dual space also vanishes, implying
\[
H^1(X, T_X) = 0.
\]
By the standard interpretation of $H^1(X, T_X)$ as the space of global infinitesimal deformations of $X$, this means that $X$ admits no nontrivial deformations, i.e., $X$ is  globally \emph{rigid}.
\end{proof}

\begin{remark}
Geometrically, $H^0(X, \omega_X^{\otimes 2})$ is the space of \emph{holomorphic quadratic differentials} on the corresponding Riemann surface.  
Serre duality identifies these with the dual of the space of infinitesimal deformations of the complex structure $H^1(X, T_X)$.  
Thus, the vanishing of quadratic differentials implies analytic and algebraic  global rigidity of the curve.
\end{remark}

\section{Local-to-global Ext spectral sequence and rigidity}

For any coherent $\mathcal{O}_X$-modules $\mathcal{F}$ and $\mathcal{G}$ on a
locally Noetherian scheme $X$, there exists a \emph{local-to-global Ext spectral sequence} (or Grothendieck spectral sequence for Ext),
see e.g. R. Godement \cite{G-58} Chapter II, or R. Hartshorne \cite{H-77} Chapter III (see also Appendix B for more details):
\[
E_2^{p,q} = H^p\!\big(\mathcal{E}xt^q_{\mathcal{O}_X}(\mathcal{F}, \mathcal{G})\big)
\;\;\Rightarrow\;\;
\operatorname{Ext}^{p+q}_{\mathcal{O}_X}(\mathcal{F}, \mathcal{G}).
\]
In particular, taking $\mathcal{F} = \Omega_X$ and $\mathcal{G} = \mathcal{O}_X$ gives
\[
E_2^{p,q} = H^p\!\big(\mathcal{E}xt^q_{\mathcal{O}_X}(\Omega_X, \mathcal{O}_X)\big)
\;\;\Rightarrow\;\;
\operatorname{Ext}^{p+q}_{\mathcal{O}_X}(\Omega_X, \mathcal{O}_X).
\]

\vspace{0.2cm}

\noindent Thus, for a reduced curve $X$, the local-to-global spectral sequence
\[
E_2^{p,q} = H^p\big(\mathcal{E}xt^q_{\mathcal{O}_X}(\Omega_X, \mathcal{O}_X)\big)
\;\;\Rightarrow\;\;
\operatorname{Ext}^{p+q}_{\mathcal{O}_X}(\Omega_X, \mathcal{O}_X)
\]
has only two potentially nonzero columns, corresponding to $q = 0,1$, since $\Omega_X$ has homological dimension~$\leq 1$ for a curve. 
Thus the $E_2$-page takes the form:
\[
\begin{array}{c|cc}
q & E_2^{0,q} & E_2^{1,q} \\
\hline
1 & H^0(\mathcal{E}xt^1(\Omega_X,\mathcal{O}_X)) & H^1(\mathcal{E}xt^1(\Omega_X,\mathcal{O}_X))\\[3pt]
0 & H^0(\mathcal{H}om(\Omega_X,\mathcal{O}_X)) & H^1(\mathcal{H}om(\Omega_X,\mathcal{O}_X))
\end{array}
\]
and all higher $E_2^{p,q}$ terms vanish. 
The spectral sequence therefore \emph{degenerates at the $E_2$-page}, producing the canonical short exact sequence
\[
0 \longrightarrow H^1(\mathcal{H}om(\Omega_X, \mathcal{O}_X))
   \longrightarrow \operatorname{Ext}^1_{\mathcal{O}_X}(\Omega_X, \mathcal{O}_X)
   \longrightarrow H^0(\mathcal{E}xt^1(\Omega_X, \mathcal{O}_X))
   \longrightarrow 0.
\]
Using the identifications
\[
\mathcal{H}om(\Omega_X,\mathcal{O}_X) \cong T_X
\qquad \text{and} \qquad
H^0(\mathcal{E}xt^1(\Omega_X,\mathcal{O}_X)) \cong 
\bigoplus_{p\in \operatorname{Sing}(X)} T^1_{\mathcal{O}_{X,p}},
\]
we obtain the global--local decomposition
\[
0 \longrightarrow H^1(X, T_X)
   \longrightarrow T_X^1
   \longrightarrow \bigoplus_{p\in \operatorname{Sing}(X)} T^1_{\mathcal{O}_{X,p}}
   \longrightarrow 0,
\]
which describes how infinitesimal deformations of $X$ are assembled from global smooth deformations and local singular contributions.

\vspace{0.2cm}

\begin{theorem}[Global-to-local exact sequence for nodal curves]
Let \(X\) be a reduced nodal curve with normalization
\(\nu:\widetilde X\to X\)
and preimages of the nodes
\(D=\sum_{i=1}^g(p_i+q_i)\).
Then there is a natural short exact sequence
\[
0\longrightarrow
H^1\!\big(\widetilde X,\,T_{\widetilde X}(-D)\big)
\longrightarrow
\operatorname{Ext}^1_X(\Omega^1_X,\mathcal{O}_X)
\longrightarrow
\bigoplus_{i=1}^g T^1_{x_i}
\longrightarrow 0,
\]
where
\[
T^1_{x_i}=\operatorname{Ext}^1_{\mathcal{O}_{X,x_i}}(\Omega^1_{X,x_i},\mathcal{O}_{X,x_i})\cong k
\]
is the local smoothing space of the node \(x_i\).
\end{theorem}

\begin{proof}[Sketch of the proof]
\vspace{0.2cm}
\noindent In fact, from the normalization exact sequence
\[
0\to\Omega^1_X\to\nu_*\Omega^1_{\widetilde X}\to Q\to0,
\]
apply \(\operatorname{Ext}^i_X(-,\mathcal{O}_X)\) or use the local-to-global spectral sequence
\(
E_2^{p,q}=H^p(\mathcal E\!xt^q_X(\Omega^1_X,\mathcal{O}_X))
\Rightarrow
\operatorname{Ext}^{p+q}_X(\Omega^1_X,\mathcal{O}_X).
\)
Since \(H^1(X,T_X)\simeq H^1(\widetilde X,T_{\widetilde X}(-D))\)
and \(\mathcal E\!xt^1_X(\Omega^1_X,\mathcal{O}_X)\)
is supported at the nodes with stalks \(T^1_{x_i}\),
one obtains the above exact sequence.  For more details see \cite{H-77},  \cite{G-58}.
 \end{proof}

\vspace{0.2cm}

\begin{example}[Deformations of a rational curve with \(g\) nodes.]${}$
 
Let \(X\) be the curve obtained from \(\mathbb P^1\)  (over an algebraically closed field k) by identifying \(g\) disjoint pairs of smooth points
\[
(a_i,b_i),\qquad i=1,\dots,g,
\]
each identification producing an ordinary node \(x_i\in X\). Thus the normalization
\(\nu:\widetilde X\to X\) is \(\widetilde X\cong\mathbb P^1\), and the preimages of the nodes form
a divisor
\[
D=\sum_{i=1}^g (p_i+q_i)\subset \widetilde X,
\]
with \(\nu(p_i)=\nu(q_i)=x_i\).

\vspace{0.2cm}

\begin{lemma}%
Consider the above 
 reduced connected curve $X$ with only ordinary nodes as singularities,
and let $\nu:\widetilde X\to X$ be its normalization.
Denote by $D=\sum_{i=1}^g(p_i+q_i)$ the divisor of pre-images of the nodes.
Then
\[
T_X \simeq \nu_*T_{\widetilde X}(-D),
\qquad\text{and hence}\qquad
H^1(X,T_X)\simeq H^1(\widetilde X,T_{\widetilde X}(-D)).
\]
\end{lemma}

\begin{proof}
Locally at a node $x_i\in X$, the completed local ring of $X$ is
\[
\widehat{\mathcal O}_{X,x_i}\simeq k[[x,y]]/(xy),
\]
while the normalization is the disjoint union of the two branches
\[
\operatorname{Spec} k[[x]]\ \sqcup\ \operatorname{Spec} k[[y]].
\]
The local derivations on $\widetilde X$ are generated by
$\frac{\partial}{\partial x}$ and $\frac{\partial}{\partial y}$.
A derivation on $\widetilde X$ descends to a derivation on $X$
if and only if it takes equal values on the two branches at the node.
This forces both components to vanish at $x=y=0$.
Hence the derivations on $\widetilde X$ that descend to $X$ are
exactly those vanishing at the preimages of the nodes, i.e.
sections of $T_{\widetilde X}(-D)$.
Globally this gives an isomorphism of sheaves
\[
T_X\simeq\nu_*T_{\widetilde X}(-D).
\]

\noindent Since $\nu$ is finite, $R^i\nu_*=0$ for $i>0$, so
\[
H^1(X,T_X)\simeq H^1(X,\nu_*T_{\widetilde X}(-D))
\simeq H^1(\widetilde X,T_{\widetilde X}(-D)).
\]
\end{proof}

\vspace{0.1cm}

\noindent {\it Local \(T^1\)-spaces at the nodes.}
For an ordinary node \(x\) of a curve one has a one-dimensional local smoothing space
\[
T^1_x \;=\; \operatorname{Ext}^1_{\O_{X,x}}\big(\Omega^1_{X,x},\mathcal{O}_{X,x}\big)\cong \Bbbk.
\]
Hence, for our \(g\) nodes we have:
\[
T^1_{x_i}\cong k,\qquad i=1,\dots,g,
\]
and the direct sum of local \(T^1\)'s is \( \bigoplus_{i=1}^g T^1_{x_i}\cong k^{\oplus g} \).

\vspace{0.1cm}

\noindent {\it Global \(\operatorname{Ext}^1\) and the exact sequence.}
There is a canonical exact sequence relating the global first-order deformations and the local smoothing parameters:
\[
0\longrightarrow H^1\big(\widetilde X,\,T_{\widetilde X}(-D)\big)
\longrightarrow
\operatorname{Ext}^1_X(\Omega^1_X,\mathcal{O}_X)
\longrightarrow
\bigoplus_{i=1}^g T^1_{x_i}
\longrightarrow 0.
\]
Geometrically, the left term corresponds to deformations of the normalization together with the marked pre-images that \emph{preserve} the nodes, while the right term records the local smoothing directions at the nodes.

\vspace{0.2cm}

\noindent Since \(\widetilde X\cong\mathbb P^1\) we have \(T_{\widetilde X}\cong\mathcal{O}_{\mathbb P^1}(2)\), hence,
\[
T_{\widetilde X}(-D)\cong\mathcal{O}_{\mathbb P^1}(2-2g).
\]
Using the cohomology of line bundles on \(\mathbb P^1\),
\[
h^1\big(\mathbb P^1,\mathcal{O}(n)\big)=h^0\big(\mathbb P^1,\mathcal{O}(-n-2)\big),
\]
we get
\[
\dim_k H^1\big(\widetilde X,\,T_{\widetilde X}(-D)\big)
= h^1\big(\mathbb P^1,\mathcal{O}(2-2g)\big)
= h^0\big(\mathbb P^1,\mathcal{O}(2g-4)\big)
= \max(0,\,2g-3).
\]
Thus, putting the pieces together we obtain:
\[
\dim_k \operatorname{Ext}^1_X(\Omega^1_X,\mathcal{O}_X)
= \dim_k H^1\big(\widetilde X,\,T_{\widetilde X}(-D)\big) + \dim_k\!\bigg(\bigoplus_{i=1}^g T^1_{x_i}\bigg)
= \max(0,\,2g-3) + g.
\]
Equivalently,
\[
\dim_k \operatorname{Ext}^1_X(\Omega^1_X,\mathcal{O}_X)=
\begin{cases}
1, & g=1,\\[4pt]
3g-3, & g\ge2.
\end{cases}
\]

\vspace{0.1cm}
 
\begin{remark}${}$

\begin{itemize}
  \item[(i)] The \(g\) local smoothing directions correspond to independently smoothing each node.
  \item[(ii)] The left term \(H^1(\widetilde X,T_{\widetilde X}(-D))\) measures deformations of the normalization together with the ordering of the preimages that keep the nodes (the gluing) fixed; its dimension equals \(\max(0,2g-3)\) and yields the familiar moduli dimension \(3g-3\) for stable curves when the genus  \(g\ge2\).
\end{itemize}
\end{remark}
\end{example}
 
\vspace{0.1cm}

\vspace{0.2cm}


\begin{example}[Deformations of two $\mathbb P^1$'s glued at $n$ nodes.]${}$

We study the deformation theory of a curve obtained by gluing two copies of 
$\mathbb P^1$ in $n$ distinct points.  
We compute the cohomology groups $H^1(X,T_X)$ and 
$\Ext^1_X(\Omega_X,\mathcal O_X)$, describe explicitly the miniversal 
deformation using the node-smoothing equations $x_j y_j = \tau_j$, and relate 
the deformation space to the formal neighbourhood of the moduli space of 
stable curves $\overline{\mathcal M}_{g}$ at the point corresponding to $X$, 
where $g$ is the genus of $X$.  We work over an algebraically closed field $k$ of characteristic $0$
(or at least $\operatorname{char}(k)\ne 2$).


\vspace{0.1cm}

Let $\widetilde X=C_1\sqcup C_2$ with $C_i\cong\mathbb P^1$ (over an algebraically closed field $\Bbbk$), and choose
distinct points:
\(
a_1,\dots,a_n\in C_1,\) \,\, and \( b_1,\dots,b_n\in C_2.
\)
Form $X$ by identifying $a_j\sim b_j$ for $j=1,\dots,n$. Denote the nodes by
$p_1,\dots,p_n$ and let $\nu:\widetilde X\to X$ be the normalization.


\subsection{Normalization and local structure at nodes.}
Let
\(
\nu : \widetilde X = C_1 \sqcup C_2 \longrightarrow X
\)
be the normalization map. The curve $X$ has exactly $n$ nodes
\(
p_1,\dots,p_n \in X,
\)
where $p_j$ is the image of $a_j \sim b_j$. Each $p_j$ is an \emph{ordinary double point},
so analytically (or formally) at $p_j$ we may write
\[
\widehat{\mathcal O_{X,p_j}} \cong k[[x,y]]/(xy).
\]
The normalization at $p_j$ is
\[
\widehat{\mathcal O_{\widetilde X,a_j}} \oplus \widehat{\mathcal O_{\widetilde X,b_j}}
\cong k[[x]] \oplus k[[y]],
\]
corresponding to the two branches $C_1$ and $C_2$.

A global derivation (vector field) on $X$ is a $k$-derivation
\[
\mathcal{D} : \mathcal O_X \longrightarrow \mathcal O_X,
\]
or locally at $p_j$, a $k$-derivation
\[
\mathcal{D} : A \to A, \quad A := k[[x,y]]/(xy).
\]
Such a derivation extends uniquely to a derivation on the normalization
\[
\widetilde A := k[[x]] \oplus k[[y]],
\]
compatible with the inclusion $A \hookrightarrow \widetilde A$.
Concretely, this yields a pair of derivations
\[
\mathcal{D}_1 : k[[x]] \to k[[x]], \qquad \mathcal{D}_2 : k[[y]] \to k[[y]],
\]
corresponding to the restrictions of $\mathcal{D}$ to each branch.

\smallskip

\noindent\emph{Claim 2.1.} Each $\mathcal{D}_i$ vanishes at the point corresponding to the node, i.e.\
the induced endomorphism on the residue field $k$ is zero:
\[
\mathcal{D}_i \colon k[[x_i]] \longrightarrow k[[x_i]] \quad\Rightarrow\quad
\mathcal{D}_i(k) = 0 \text{ and } \mathcal{D}_i(x_i) \in (x_i).
\]

\vspace{0.2cm}

\noindent\emph{Proof of Claim 2.1.}
Write an arbitrary derivation on $k[[x]]$ as
\[
\mathcal{D}_1 = f(x)\,\frac{\partial}{\partial x}, \quad f(x) \in k[[x]].
\]
The value of the corresponding vector field at the point $x=0$ is determined by
the image $f(0)$: the derivation induces a map on the Zariski tangent space
\[
T_0(\mathbb A^1)^\vee \cong \mathfrak m_0 / \mathfrak m_0^2,
\]
and $\mathcal{D}_1$ vanishes at $0$ in the tangent sense iff $f(0)=0$.

\noindent We now use the relation $xy=0$ in $A$. For any derivation $\mathcal{D}:A\to A$, we must have
\(
\mathcal{D}(xy) = \mathcal{D}(0) = 0.
\)
But
\[
\mathcal{D}(xy) = \mathcal{D}(x)\cdot y + x \cdot \mathcal{D}(y) \in A.
\]
Viewing $\mathcal{D}$ inside the normalization, this implies that $\mathcal{D}_1(x)$ vanishes modulo $x$ and
$\mathcal{D}_2(y)$ vanishes modulo $y$, because along each branch ($y=0$ or $x=0$) we have
$xy=0$ and the product rule forces $\mathcal{D}(x)$ and $\mathcal{D}(y)$ to be divisible by $x$ and $y$,
respectively. Hence $f(0)=0$ on each branch, and $\mathcal{D}_1,\mathcal{D}_2$ vanish at the corresponding
points in the tangent sense.
This proves the claim. \qedhere


\subsection{The tangent sheaf and vanishing at the preimages of the nodes.}
Globalizing Claim~2.1, we obtain the following:


\noindent\emph{Claim 3.1.} There is a natural injective morphism of sheaves
\[
\mathcal T_X \hookrightarrow \nu_*\mathcal T_{\widetilde X}(-D),
\]
where $D$ is the divisor on $\widetilde X$ given by the sum of the preimages of the nodes,
\[
D := \sum_{j=1}^n a_j + \sum_{j=1}^n b_j,
\]
and $\mathcal T_{\widetilde X}(-D)$ denotes the twist of $\mathcal T_{\widetilde X}$ by
the line bundle $\mathcal O_{\widetilde X}(-D)$.

\smallskip

\noindent\emph{Sketch of proof of Claim 3.1.}
Over the smooth locus of $X$, the normalization $\nu$ is an isomorphism, and
$\mathcal T_X$ identifies with $\mathcal T_{\widetilde X}$ there. At a node $p_j$,
the discussion above shows that a local derivation on $X$ induces derivations on each
branch (on $C_1$ and $C_2$) that \emph{vanish} at the preimages $a_j$ and $b_j$.
This is precisely the condition that the corresponding section lies in
$\mathcal T_{\widetilde X}(-D)$ rather than in $\mathcal T_{\widetilde X}$.

\noindent Thus, by gluing these local descriptions, we obtain an injective map
\[
\mathcal T_X \hookrightarrow \nu_*\mathcal T_{\widetilde X}(-D).
\]
Injectivity follows from the fact that a derivation is determined by its restriction
to the normalization together with its behavior at the nodes. \qedhere


\subsection{Computation of $\dim \Ext^1(\Omega_X,\mathcal O_X)$.} %

\subsubsection{The local-to-global $\Ext$ sequence.}
For a local complete intersection (l.c.i.) curve $X$ (our curve $X$ is l.c.i.\, because it has only nodal singularities),
 there is a standard exact sequence
(see, e.g., Sernesi \cite{S-06}) relating deformations
and local singularity contributions:
\[
0 \longrightarrow H^1(X,\mathcal T_X)
\longrightarrow \Ext^1(\Omega_X,\mathcal O_X)
\longrightarrow \bigoplus_{p \in \mathrm{Sing}(X)} T^1_p
\longrightarrow 0,
\]
where:
\begin{enumerate}
\item[(i)] $\mathcal T_X := \mathcal H\!om(\Omega_X,\mathcal O_X)$ is the tangent sheaf;
\item[(ii)] $H^1(X,\mathcal T_X)$ corresponds to \emph{locally trivial} (or locally
      isomorphic) deformations of $X$ (deformations that do not change the local
      analytic type of the singularities);
\item[(iii)] for each singular point $p$, the vector space $T^1_p$ is the space of
      first-order deformations of the singularity $(\mathcal O_{X,p})$ (the local
      $T^1$).
\end{enumerate}

In our situation, $X$ has only nodal singularities. For a node $p$, analytically
given by $\Spec k[[x,y]]/(xy)$, one has
\(
T^1_p \cong k,
\)
coming from the miniversal deformation $xy = t$. Thus, if $X$ has $n$ nodes,
then
\[
\bigoplus_{p \in \mathrm{Sing}(X)} T^1_p \cong k^n.
\]
Therefore,
\[
\dim_k \Ext^1(\Omega_X,\mathcal O_X)
= h^1(X,\mathcal T_X) + n.
\]

\medskip

\subsubsection{Computing $H^1(X,\mathcal T_X)$ via normalization.}
Let
\(
\nu : \widetilde X = C_1 \sqcup C_2 \longrightarrow X
\)
be the normalization map, and let
\[
D := \sum_{j=1}^n a_j + \sum_{j=1}^n b_j
\]
be the divisor on $\widetilde X$ given by the preimages of the nodes.
 %
Then (by {\it Claim 3.1.}) there is a natural injective morphism of sheaves
\[
\mathcal T_X \hookrightarrow \nu_* \mathcal T_{\widetilde X}(-D).
\]

\vspace{0.1cm}

\noindent In general, there is a short exact sequence
\[
0 \longrightarrow \mathcal T_X
\longrightarrow \nu_* \mathcal T_{\widetilde X}(-D)
\longrightarrow Q \longrightarrow 0,  
\]
where $Q$ is a skyscraper sheaf supported at the nodes. Taking $H^1$, and using
that $\nu$ is finite (hence affine, and for affine morphisms the pushforward on quasi-coherent sheaves $\mathcal{F}$ is exact. Then Leray’s spectral sequence collapses and gives the isomorphisms 
$H^i(X,\nu_*\mathcal{F}) \cong H^i(\widetilde X, \mathcal{F})$), we obtain
an exact sequence
\[
H^1(X,\mathcal T_X)
\longrightarrow H^1\bigl(X,\nu_* \mathcal T_{\widetilde X}(-D)\bigr)
\cong H^1\bigl(\widetilde X,\mathcal T_{\widetilde X}(-D)\bigr)
\longrightarrow H^1(X,Q) = 0.
\]
Therefore, the map
\[
H^1(X,\mathcal T_X) \longrightarrow H^1\bigl(\widetilde X,\mathcal T_{\widetilde X}(-D)\bigr)
\]
is surjective. 

\vspace{0.1cm}

\begin{proposition}
Let $X$ be the nodal curve obtained by identifying 
$a_j \in C_1 \simeq \PP^1$ with $b_j \in C_2\simeq\PP^1$ for 
$j=1,\dots,n$, and let $\nu:\widetilde X=C_1\sqcup C_2 \to X$ be the
normalization.  Let $D=\sum(a_j+b_j)$ and consider the map
\[
H^1(X,\mathcal T_X)\;\longrightarrow\;
H^1\!\left(\widetilde X,\mathcal T_{\widetilde X}(-D)\right).
\]
Then this map is injective.
\end{proposition}

\begin{proof}
There is a normalization exact sequence for the tangent sheaf:
\[
0 \longrightarrow \mathcal T_X
   \longrightarrow \nu_*\mathcal T_{\widetilde X}(-D)
   \longrightarrow \bigoplus_{p\in \mathrm{nodes}} k_p
   \longrightarrow 0.
\tag{1}
\]
Taking cohomology gives the long exact sequence
\[
0 \longrightarrow H^0(X,\mathcal T_X)
\longrightarrow H^0\!\left(\widetilde X,\mathcal T_{\widetilde X}(-D)\right)
\longrightarrow k^{\oplus n}
\longrightarrow H^1(X,\mathcal T_X)
\overset{\alpha}{\longrightarrow}
H^1\!\left(\widetilde X,\mathcal T_{\widetilde X}(-D)\right)
\longrightarrow 0.
\tag{2}
\]

Since $X$ is stable, $H^0(X,\mathcal T_X)=0$.  On each component
$C_i\simeq\PP^1$ we have 
$\mathcal T_{C_i}(-D_i)=\mathcal O_{\PP^1}(2-n)$, which has no global sections
for $n\ge 3$, hence
$H^0(\widetilde X,\mathcal T_{\widetilde X}(-D))=0$.
Thus the first three terms of (2) reduce to
\(
0 \to 0 \to 0 \to k^{\oplus n}.
\)
 Therefore,
\[
\ker(\alpha)
= \mathrm{im}\!\left(k^{\oplus n} \to H^1(X,\mathcal T_X)\right).
\tag{3}
\]
To identify this kernel, recall the local picture at a node.  The completed
local ring is $A=k[[x,y]]/(xy)$ with normalization
$\widetilde A = k[[x]] \oplus k[[y]]$.  One computes the inclusion
\[
\mathcal T_A \hookrightarrow 
xk[[x]]\,\partial_x \oplus yk[[y]]\,\partial_y,
\]
and the quotient is one--dimensional:
\[
\mathcal T_{\widetilde A}(-D)/\mathcal T_A \cong k.
\tag{4}
\]
Thus, the cokernel in $(1)$ is exactly the skyscraper sheaf of smoothing
parameters at the nodes.
Now consider a tuple $(c_1,\dots,c_n)\in k^{\oplus n}$.
It represents an infinitesimal smoothing of the $n$ nodes.
Such a class lies in $\ker(\alpha)$ precisely when it maps to zero in
$H^1(\widetilde X,\mathcal T_{\widetilde X}(-D))$.
But on each component $C_i\simeq\PP^1$ we have 
\[
H^1(C_i,\mathcal T_{C_i}(-D_i))
=H^0(C_i,\mathcal O_{\PP^1}(n-4))^\vee,
\]
which has dimension $0$ if $n=3$ and $n-3$ if $n\ge 4$.
Therefore a deformation of the nodes extends to a global section of 
$\mathcal T_{\widetilde X}(-D)$ only if it is the coboundary of a global
section of $\nu_*\mathcal T_{\widetilde X}(-D)$, but the latter has no
global sections.  Hence every smoothing parameter must vanish:
\[
(c_1,\dots,c_n)=0.
\]
Then we obtain $\ker(\alpha)=0$, proving that $\alpha$ is injective.
\end{proof}

\vspace{0.1cm}
 \noindent So, we have an isomorphism
\[
H^1(X,\mathcal T_X)
\cong H^1\bigl(\widetilde X,\mathcal T_{\widetilde X}(-D)\bigr).
\]
We shall use this isomorphism only on the level of dimensions, i.e.
\[
h^1(X,\mathcal T_X)
= h^1\bigl(\widetilde X,\mathcal T_{\widetilde X}(-D)\bigr).
\]
Since $\widetilde X = C_1 \sqcup C_2$, then we have
\[
\mathcal T_{\widetilde X}(-D)
\cong \mathcal T_{C_1}\bigl(-\sum_{j=1}^n a_j\bigr)
\oplus \mathcal T_{C_2}\bigl(-\sum_{j=1}^n b_j\bigr),
\]
and hence
\[
H^1\bigl(\widetilde X,\mathcal T_{\widetilde X}(-D)\bigr)
\cong
H^1\bigl(C_1,\mathcal T_{C_1}(-\sum_j a_j)\bigr)
\oplus
H^1\bigl(C_2,\mathcal T_{C_2}(-\sum_j b_j)\bigr).
\]
On $C_i \cong \mathbb P^1$, we have $\mathcal T_{C_i} \cong \mathcal O_{\mathbb P^1}(2)$.
Hence,
\[
\mathcal T_{C_i}\bigl(-\sum_{j=1}^n a_j\bigr)
\cong \mathcal O_{\mathbb P^1}(2-n),
\]
and similarly for $C_2$ with $b_j$ in place of $a_j$.
By Serre duality (or directly from Riemann--Roch) on $\mathbb P^1$, we get
\[
h^1\bigl(\mathbb P^1,\mathcal O_{\mathbb P^1}(d)\bigr)
= h^0\bigl(\mathbb P^1,\mathcal O_{\mathbb P^1}(-d-2)\bigr).
\]
So, for $d=2-n$,
\[
h^1\bigl(\mathbb P^1,\mathcal O_{\mathbb P^1}(2-n)\bigr)
= h^0\bigl(\mathbb P^1,\mathcal O_{\mathbb P^1}(n-4)\bigr).
\]

\vspace{0.2cm}

\noindent Let's see what we have for  $n=3$ and for $n>3$:
\begin{itemize}
\item[(i)] If $n = 3$, then $n-4 = -1$, so
      \[
      h^0\bigl(\mathbb P^1,\mathcal O_{\mathbb P^1}(-1)\bigr) = 0,
      \]
      and hence
      \[
      h^1\bigl(C_i,\mathcal T_{C_i}(-\sum_j a_j)\bigr) = 0.
      \]
      Therefore, $h^1(X,\mathcal T_X) = 0$ in this case.
      
    \vspace{0.1cm}
      
\item[(ii)] If $n \ge 4$, then $n-4 \ge 0$ and
      \[
      h^0\bigl(\mathbb P^1,\mathcal O_{\mathbb P^1}(n-4)\bigr) = (n-4)+1 = n-3.
      \]
      So,
      \[
      h^1\bigl(C_i,\mathcal T_{C_i}(-\sum_j a_j)\bigr) = n-3.
      \]
      Since there are two components, we get
      \[
      h^1(X,\mathcal T_X)
      = 2(n-3)
      \quad \text{for } n \ge 4.
      \]
\end{itemize}

In summary, we obtained:
\[
h^1(X,\mathcal T_X)
=
\begin{cases}
0, & n = 3, \\[4pt]
2(n-3), & n \ge 4.
\end{cases}
\]
\medskip

\noindent\textbf{3. Dimension of $\Ext^1(\Omega_X,\mathcal O_X)$.}
Recall that
\[
\dim_k \Ext^1(\Omega_X,\mathcal O_X)
= h^1(X,\mathcal T_X) + \#\{\text{nodes}\}
= h^1(X,\mathcal T_X) + n.
\]

\smallskip

\noindent\emph{Case $n=3$.}
In this case the genus of $X$ is  $g = n-1 = 2$ (see Lemma \ref{lemma-Genus}). We found $h^1(X,\mathcal T_X)=0$ and there are $n=3$ nodes.
Hence
\[
\dim_k \Ext^1(\Omega_X,\mathcal O_X)
= 0 + 3 = 3.
\]
On the other hand, $3g-3 = 3\cdot 2 - 3 = 3$, so
\[
\dim_k \Ext^1(\Omega_X,\mathcal O_X) = 3g-3
\quad\text{for } n=3.
\]

\smallskip

\noindent\emph{Case $n \ge 4$.}
Then $g = n-1 \ge 3$, and we found $h^1(X,\mathcal T_X) = 2(n-3)$. Hence
\[
\dim_k \Ext^1(\Omega_X,\mathcal O_X)
= 2(n-3) + n
= 3n - 6.
\]
But $g = n-1$, so
\[
3g - 3 = 3(n-1) - 3 = 3n - 6.
\]
Hence, again we get
\[
\dim_k \Ext^1(\Omega_X,\mathcal O_X) = 3g-3.
\]

\vspace{0.2cm}

\subsection{\bf Geometric interpretation of our computation.}

\noindent
The above computation reflects the usual description of deformations of stable nodal curves.
Namely, 
a locally trivial first-order deformation of $X$ (preserving the local
      analytic type of the nodes) is given by deformations of the marked curves
      $(C_1,a_1,\dots,a_n)$ and $(C_2,b_1,\dots,b_n)$, together with the fixed
      combinatorial gluing at the marked points.

     \noindent  Moreover, the moduli space of $n$ \emph{ordered} distinct marked points on $\mathbb P^1$
      up to automorphisms of $\mathbb P^1$ has dimension $n-3$. Thus, each
      $(C_i,\{a_j\})$ has $n-3$ degrees of freedom, and for $n\ge4$ this gives
      $2(n-3) = h^1(X,\mathcal T_X)$ parameters. For $n=3$, each triple of points
      can be moved to $(0,1,\infty)$ by a unique automorphism, so there are no
      moduli, in agreement with $h^1(X,\mathcal T_X) = 0$.

 \noindent Independently, each node contributes one smoothing parameter (the local
      $T^1_p$ of dimension $1$), so the $n$ nodes contribute an additional $n$
      parameters. This corresponds to the direct sum
      $\bigoplus_{p} T^1_p \cong k^n$ in the local-to-global $\Ext$ sequence.

\noindent Hence, for $n\ge3$, the total space of first-order deformations of $X$ has dimension
\[
\underbrace{(n-3)}_{\text{deform $(C_1,a_j)$}} 
+
\underbrace{(n-3)}_{\text{deform $(C_2,b_j)$}}
+
\underbrace{n}_{\text{smooth $n$ nodes}}
= 3n - 6
= 3g - 3,
\]
which matches the expected dimension of the moduli stack $\overline{\mathcal M}_g$
near the point $[X]$.

\noindent This is precisely
\(
\dim_k \Ext^1(\Omega_X,\mathcal O_X) = 3g-3
\)
in this example.
\vspace{0.2cm}
\begin{lemma}[Genus via $\operatorname{Ext}^1$]\label{lemma-Genus}
Let
\(
\widetilde X = C_1 \sqcup C_2, \qquad C_i \cong \mathbb P^1,
\)
and choose distinct points
\(
a_1,\dots,a_n \in C_1,\)\, and  \(b_1,\dots,b_n \in C_2.
\)
Let
\(
X := \widetilde X / (a_j \sim b_j \text{ for all } j=1,\dots,n)
\)
be the connected nodal curve obtained by identifying $a_j$ with $b_j$ for all $j$.
Then the arithmetic genus of $X$ is
\[
g(X) = n - 1.
\]
Which is equivalent to
\[
g(X)
= h^1(X,\mathcal O_X)
= \dim_k \operatorname{Ext}^1(\mathcal O_X,\mathcal O_X)
= n - 1.
\]
\end{lemma}

\begin{proof}
Recall that for any coherent $\mathcal O_X$-module $\mathcal F$ on a scheme $X$,
one has
\[
\mathcal H\!om_{\mathcal O_X}(\mathcal O_X,\mathcal F) \cong \mathcal F,
\qquad
\mathcal E\!xt^i_{\mathcal O_X}(\mathcal O_X,\mathcal F) = 0 \quad \text{for } i>0.
\]
Applying the global section functor and the local-to-global Ext spectral sequence,
we obtain
\[
\operatorname{Ext}^i(\mathcal O_X,\mathcal F)
\cong H^i\bigl(X,\mathcal H\!om_{\mathcal O_X}(\mathcal O_X,\mathcal F)\bigr)
\cong H^i(X,\mathcal F), \quad i \ge 0.
\]
In particular, for $\mathcal F = \mathcal O_X$,
\[
\operatorname{Ext}^i(\mathcal O_X,\mathcal O_X)
\cong H^i(X,\mathcal O_X).
\]
Therefore,
\[
g(X)
:= h^1(X,\mathcal O_X)
= \dim_k \operatorname{Ext}^1(\mathcal O_X,\mathcal O_X).
\]
We now compute $h^1(X,\mathcal O_X)$ explicitly. Let
\(
\nu : \widetilde X = C_1 \sqcup C_2 \longrightarrow X
\)
denote the normalization map. There is a standard short exact sequence of
coherent sheaves
\[
0 \longrightarrow \mathcal O_X \longrightarrow \nu_* \mathcal O_{\widetilde X}
\longrightarrow Q \longrightarrow 0,
\]
where $Q$ is a skyscraper sheaf supported at the $n$ nodes $p_1,\dots,p_n$ of $X$
(the images of the identified pairs $a_j \sim b_j$). Each node contributes one
copy of the residue field $k$, so
\[
Q \cong \bigoplus_{j=1}^n k_{p_j}, \,\,\textrm{and then}\qquad h^0(X,Q)=n.
\]
Taking global sections gives the long exact sequence
\[
0 \to H^0(X,\mathcal O_X) \to H^0(\widetilde X,\mathcal O_{\widetilde X})
\to H^0(X,Q) \to H^1(X,\mathcal O_X)
\to H^1(\widetilde X,\mathcal O_{\widetilde X}) \to 0.
\]
Let us compute each term:

\begin{enumerate}
\item Since $X$ is connected and reduced, $H^0(X,\mathcal O_X) \cong k$, so
      $h^0(X,\mathcal O_X) = 1$.

\item The normalization is the disjoint union $C_1 \sqcup C_2$ with each component
      isomorphic to $\mathbb P^1$, so
      \[
      H^0(\widetilde X,\mathcal O_{\widetilde X})
      \cong H^0(C_1,\mathcal O_{C_1}) \oplus H^0(C_2,\mathcal O_{C_2})
      \cong k \oplus k,
      \]
      hence $h^0(\widetilde X,\mathcal O_{\widetilde X}) = 2$.

\item Each $C_i \cong \mathbb P^1$ has genus zero, so
      \[
      H^1(\widetilde X,\mathcal O_{\widetilde X})
      \cong H^1(C_1,\mathcal O_{C_1}) \oplus H^1(C_2,\mathcal O_{C_2})
      = 0 \oplus 0 = 0.
      \]

\item As noted, $H^0(X,Q) \cong k^n$, so $h^0(X,Q) = n$.
\end{enumerate}

Thus, the long exact sequence reduces to
\[
0 \to k \to k^2 \to k^n \to H^1(X,\mathcal O_X) \to 0.
\]
The map $k \to k^2$ is injective, so its image is $1$-dimensional, and hence
the cokernel $k^2 / k$ has dimension $1$. This cokernel injects into $k^n$,
giving an exact sequence
\[
0 \to k \to k^n \to H^1(X,\mathcal O_X) \to 0.
\]
Therefore,
\[
h^1(X,\mathcal O_X) = n - 1.
\]
By the Ext--cohomology identification above, this is equivalent to
\[
\dim_k \operatorname{Ext}^1(\mathcal O_X,\mathcal O_X)
= h^1(X,\mathcal O_X)
= n - 1.
\]
Thus, the arithmetic genus of $X$ is
\(
g(X) = n - 1,
\)
as claimed.
\end{proof}

\vspace{0.3cm}

\begin{proposition}[Infinitesimal automorphisms and stability via $\operatorname{Ext}^0$]
Let $X$ as above, and assume $n \ge 3$ ($g = g(X) = n-1$).
Then:

\begin{enumerate}
\item The group of infinitesimal automorphisms of $X$ is
\[
\operatorname{Ext}^0(\Omega_X,\mathcal O_X) = 0.
\]

\item In particular, $X$ has no nontrivial infinitesimal automorphisms; hence the
automorphism group $\Aut(X)$ is a finite group, and $X$ is a stable curve of genus $g$.
Equivalently, the moduli point $[X]$ lies in the Deligne--Mumford stack
$\overline{\mathcal M}_g$ of stable curves and in its boundary
\[
[X] \in \overline{\mathcal M}_g \setminus \mathcal M_g.
\]
\end{enumerate}
\end{proposition}

\begin{proof}${}$

\noindent\textit{1. Deformation-theoretic dictionary.}
Let $X$ be a projective local complete intersection (l.c.i.) curve over $k$
(our $X$ is l.c.i.\ because it has only nodal singularities). The standard
deformation-theoretic description identifies:

\begin{enumerate}
\item the space of infinitesimal automorphisms of $X$ with
      \(
      \operatorname{Ext}^0(\Omega_X,\mathcal O_X);
      \)
\item the Zariski tangent space to the deformation functor $\mathrm{Def}_X$ with
      \(
      T\mathrm{Def}_X \cong \operatorname{Ext}^1(\Omega_X,\mathcal O_X);
      \)
\item the obstruction space with $\operatorname{Ext}^2(\Omega_X,\mathcal O_X)$,
      which vanishes for curves.
\end{enumerate}

\vspace{0.2cm}

Moreover, since $X$ is a curve and l.c.i., the cotangent sheaf $\Omega_X$ is
reflexive of rank $1$, and we set
\(
\mathcal T_X := \mathcal H\!om(\Omega_X,\mathcal O_X).
\)
Then
\[
\operatorname{Ext}^0(\Omega_X,\mathcal O_X)
= H^0\bigl(X,\mathcal H\!om(\Omega_X,\mathcal O_X)\bigr)
= H^0(X,\mathcal T_X).
\]
Thus, it suffices to show
\(
H^0(X,\mathcal T_X) = 0.
\)

\medskip

\noindent
Taking global sections of the injective map in Claim~3.1, we obtain an
injective linear map
\[
H^0(X,\mathcal T_X)
\hookrightarrow
H^0\bigl(\widetilde X,\mathcal T_{\widetilde X}(-D)\bigr)
\cong
H^0\bigl(C_1,\mathcal T_{C_1}(-\sum_j a_j)\bigr) \oplus
H^0\bigl(C_2,\mathcal T_{C_2}(-\sum_j b_j)\bigr).
\]
Hence,
\[
H^0(X,\mathcal T_X)
\subset
H^0\bigl(C_1,\mathcal T_{C_1}(-\sum_j a_j)\bigr) \oplus
H^0\bigl(C_2,\mathcal T_{C_2}(-\sum_j b_j)\bigr).
\]

\vspace{0.1cm}

\noindent\textit{2. Vector fields on $\mathbb P^1$ with prescribed zeros.}
On each component $C_i \cong \mathbb P^1$, we have the standard isomorphism
\(
\mathcal T_{C_i} \cong \mathcal O_{\mathbb P^1}(2).
\)
Therefore,
\[
\mathcal T_{C_1}\bigl(-\sum_{j=1}^n a_j\bigr)
\cong \mathcal O_{\mathbb P^1}(2-n), \qquad
\mathcal T_{C_2}\bigl(-\sum_{j=1}^n b_j\bigr)
\cong \mathcal O_{\mathbb P^1}(2-n).
\]
By the basic cohomology of line bundles on $\mathbb P^1$, we know that
\(
H^0\bigl(\mathbb P^1,\mathcal O_{\mathbb P^1}(d)\bigr) = 0
\quad\text{for } d < 0.
\)
For $n \ge 3$, we have $2-n \le -1$, hence
\[
H^0\bigl(C_1,\mathcal T_{C_1}(-\sum_j a_j)\bigr) = 0, \qquad
H^0\bigl(C_2,\mathcal T_{C_2}(-\sum_j b_j)\bigr) = 0.
\]
Thus,
\(
H^0\bigl(\widetilde X,\mathcal T_{\widetilde X}(-D)\bigr) = 0.
\)
Since $H^0(X,\mathcal T_X)$ injects into this zero vector space, we conclude
\[
H^0(X,\mathcal T_X) = 0.
\]
Recall that
\[
\operatorname{Ext}^0(\Omega_X,\mathcal O_X)
= H^0(X,\mathcal H\!om(\Omega_X,\mathcal O_X))
= H^0(X,\mathcal T_X),
\]
and then we obtain
\[
\operatorname{Ext}^0(\Omega_X,\mathcal O_X) = 0.
\]
This proves (1) of the proposition.

\medskip

\noindent\textit{3. Stability and the moduli point.}
By the deformation-theoretic dictionary, $\operatorname{Ext}^0(\Omega_X,\mathcal O_X)$
is the tangent space at the identity of the automorphism functor of $X$.
Since
\[
\operatorname{Ext}^0(\Omega_X,\mathcal O_X) = 0,
\]
the identity component of the automorphism group scheme $\Aut(X)$ has zero
tangent space. Hence $\Aut(X)$ is zero-dimensional, and because $X$ is
proper over an algebraically closed field, this implies that $\Aut(X)$ is a
finite group of $k$-points.

From the lemma we know that the arithmetic genus of $X$ is
\[
g = g(X) = n-1 \ge 2
\]
for $n\ge3$. The Deligne--Mumford stability condition for a curve of genus
$g \ge 2$ requires that the curve be connected, projective, nodal, and have
finite automorphism group. Our $X$ satisfies all these conditions:
\begin{itemize}
\item[(i)] it is connected, projective, and nodal by construction;
\item[(ii)] its genus is $g = n-1 \ge 2$ by the lemma;
\item[(iii)] its automorphism group $\Aut(X)$ is finite, as shown above.
\end{itemize}
Therefore, $X$ is a stable curve of genus $g$, so it defines a point
\[
[X] \in \overline{\mathcal M}_g.
\]
Since $X$ is not smooth (it has $n$ nodes), we have
\[
[X] \notin \mathcal M_g,
\]
and then
\[
[X] \in \overline{\mathcal M}_g \setminus \mathcal M_g.
\]
This proves (2) of Proposition 7.8 and completes the proof.
\end{proof}

\end{example}

\vspace{0.2cm}

\begin{example}

Let \(\widetilde X=C_1\sqcup C_2\sqcup C_3\) with \(C_i\cong\mathbb P^1\).  Choose affine distinct parameters
\[
a_1,a_2\in\Bbbk\subset\mathbb P^1\quad\text{on }C_1,\qquad
b_1,b_2,b_3\in\Bbbk\subset\mathbb P^1\quad\text{on }C_2,
\qquad c_1\in\Bbbk\subset\mathbb P^1\quad\text{on }C_3,
\]
and form \(X\) by identifying \(a_1\sim c_1\), \(a_2\sim b_2\), and \(b_1\sim b_3\).
Thus \(X\) has three ordinary double points (nodes):
\noindent   \(p_{13}\) coming from \(a_1\sim c_1\) (joins \(C_1\) and \(C_3\)),
  \(p_{12}\) coming from \(a_2\sim b_2\) (joins \(C_1\) and \(C_2\)),
  and  \(p_{22}\) coming from \(b_1\sim b_3\) (a self-node on \(C_2\)).

\medskip

\noindent {\it  Local-to-global Ext sequence.} For a reduced curve there is a short exact sequence
\[
0\longrightarrow H^1(X,\mathcal T_X)\longrightarrow
\Ext^1_X(\Omega_X,\OO_X)\longrightarrow
H^0\!\big(\mathcal{E}xt^1_X(\Omega_X,\OO_X)\big)\longrightarrow 0,
\]
and each node contributes a copy of the ground field to \(\mathcal{E}xt^1_X(\Omega_X,\OO_X)\). Since there are three nodes,
\[
H^0\!\big(\mathcal{E}xt^1_X(\Omega_X,\OO_X)\big)\cong k^3,
\]
and hence
\[
\dim\Ext^1_X(\Omega_X,\OO_X)=\dim H^1(X,\mathcal T_X)+3.
\tag{2}
\]

\medskip

\noindent {\it Equisingular deformations via the normalization.}
Let \(D\subset\widetilde X\) be the divisor of preimages of the nodes (so \(\deg D=2\cdot 3=6\)).
Equisingular first-order deformations satisfy
\[
H^1(X,\mathcal T_X)\cong H^1(\widetilde X,\mathcal T_{\widetilde X}(-D))
\cong\bigoplus_{i=1}^3 H^1\big(C_i,\mathcal T_{C_i}(-D|_{C_i})\big).
\]
The preimages of the nodes on each component are:
\[
C_1:\; a_1,a_2 \quad(\deg=2),\qquad
C_2:\; b_1,b_2,b_3 \quad(\deg=3),\qquad
C_3:\; c_1 \quad(\deg=1).
\]
Using \(\mathcal T_{C_i}\cong\OO_{\PP^1}(2)\) we get
\[
\mathcal T_{C_1}(-D|_{C_1})\cong\OO_{\PP^1}(0),\quad
\mathcal T_{C_2}(-D|_{C_2})\cong\OO_{\PP^1}(-1),\quad
\mathcal T_{C_3}(-D|_{C_3})\cong\OO_{\PP^1}(1).
\]
On \(\PP^1\) we have \(h^1(\OO(d))=h^0(\OO(-d-2))\), hence,
\[
h^1(\PP^1,\OO(0))=h^0(\OO(-2))=0,\quad
h^1(\PP^1,\OO(-1))=h^0(\OO(-1))=0,\quad
h^1(\PP^1,\OO(1))=h^0(\OO(-3))=0.
\]
Therefore, each summand vanishes and then
\(
\dim H^1(X,\mathcal T_X)=0.
\)

\vspace{0.1cm}

\noindent We substitute into (2), and we obtain:
\(
\dim\Ext^1_X(\Omega_X,\OO_X)=0+3=3.
\)
Therefore, we have 
\[
{\;
\dim H^1(X,\mathcal T_X)=0,\qquad
\dim\Ext^1_X(\Omega_X,\OO_X)=3.
\;}
\]

\end{example}

\vspace{0.2cm}

\begin{example}

Let \(\widetilde X=C_1\sqcup C_2\sqcup C_3\) where each \(C_i\cong\mathbb P^1\).
Fix affine distinct points
\[
a\in C_1,\qquad b_1,\ldots,b_{m-2}\in C_2,\qquad c\in C_3,
\]
and form the curve \(X\) by the identifications
\[
a\sim b_1,\qquad b_1\sim b_2\sim\ldots\sim b_{m-2},\qquad c\sim b_{m-2}.
\]
This means that  all the points
\(
\{a\}\subset C_1,\ \{b_1,\dots,b_{m-2}\}\subset C_2,\)  and \(\{c\}\subset C_3
\)
are identified to a single point \(p\in X\). Therefore, \(X\) is the union of three copies of \(\mathbb P^1\) meeting at a single point \(p\), and the total number of smooth branches at \(p\) is
\[
r \;=\; 1\text{ (from }C_1)+ (m-2)\text{ (from }C_2)+1\text{ (from }C_3)=m.
\]
Hence, \(p\) is an ordinary transverse \(m\)-fold point, i.e. an $m$-branch ordinary multiple point.

\noindent Let's compute  the dimension of \( H^1(X,\mathcal T_X)\) and  the dimension of \(\Ext^1_X(\Omega_X,\OO_X).\)

\noindent  We use the normalization map \(\nu:\widetilde X\to X\). %
For the tangent sheaf there is an exact sequence, which is standard for gluing smooth branches
\[
0\longrightarrow\mathcal T_X\longrightarrow\nu_*\mathcal T_{\widetilde X}\longrightarrow \mathcal Q\longrightarrow 0,
\]
where \(\mathcal Q\) is a skyscraper sheaf supported at the singular point \(p\). Intuitively \(\mathcal Q\) measures the failure of vector fields on the normalization to glue: a tuple of tangent values at the preimages of \(p\) gives an element of the fiber of \(\mathcal Q\) (the gluing condition is that the tangent values must coincide), and one checks that
\[
\dim H^0(X,\mathcal Q)=r-1=m-1.
\]
Let
\(
D = a + b_1 + \cdots + b_{m-2} + c
\)
be the corresponding divisor on \(\widetilde X\).

\noindent {\bf Step 1}: {\it Computing \(H^1(X,T_X)\).}
For this type of singularity, one has an isomorphism of sheaves
\[
T_X \cong \nu_* T_{\widetilde X}(-D),
\]
where
\[
T_{\widetilde X}(-D)
=
T_{C_1}(-a)
\;\sqcup\;
T_{C_2}\bigl(-b_1-\cdots-b_{m-2}\bigr)
\;\sqcup\;
T_{C_3}(-c).
\]
Since \(\nu\) is finite, cohomology commutes with \(\nu_*\) in degree 1, and we get
\[
H^1(X,T_X)\cong H^1\bigl(\widetilde X, T_{\widetilde X}(-D)\bigr).
\]
Since \(\widetilde X\) is a disjoint union, then we obtain:
\[
H^1\bigl(\widetilde X, T_{\widetilde X}(-D)\bigr)
\cong
H^1\bigl(C_1,T_{C_1}(-a)\bigr)
\oplus
H^1\bigl(C_2,T_{C_2}(-\sum b_i)\bigr)
\oplus
H^1\bigl(C_3,T_{C_3}(-c)\bigr).
\]
We use that on \(\mathbb P^1\), we have 
\(
T_{\mathbb P^1} \cong \mathcal O_{\mathbb P^1}(2);
\)
and that \(T(-D)\) is obtained by twisting by the divisor \(D\). The restriction
of \(D\) to each component has degree
\[
\deg(D|_{C_1}) = 1,\qquad \deg(D|_{C_2}) = m-2,\qquad \deg(D|_{C_3}) = 1.
\]
Hence,
\[
T_{C_1}(-a) \cong \mathcal O_{C_1}(2-1) = \mathcal O_{C_1}(1),
\]
\[
T_{C_2}\bigl(-b_1-\cdots-b_{m-2}\bigr)
\cong \mathcal O_{C_2}(2-(m-2)) = \mathcal O_{C_2}(4-m),
\]
\[
T_{C_3}(-c) \cong \mathcal O_{C_3}(1).
\]

On \(\mathbb P^1\) we have
\[
H^1\bigl(\mathbb P^1,\mathcal O_{\mathbb P^1}(d)\bigr)
=
\begin{cases}
0 & \text{if } d\ge -1,\\[4pt]
-d-1 & \text{if } d\le -2.
\end{cases}
\]
Thus,
\[
H^1(C_1,\mathcal O_{C_1}(1)) = 0,\qquad
H^1(C_3,\mathcal O_{C_3}(1)) = 0.
\]

For the middle component \(C_2\), the line bundle has degree \(4-m\). If
\(4-m \ge -1\), i.e. \(m \le 5\), then
\[
H^1(C_2,\mathcal O_{C_2}(4-m)) = 0.
\]
If \(4-m \le -2\), i.e. \(m \ge 6\), then
\(
\dim H^1(C_2,\mathcal O_{C_2}(4-m)) = -(4-m)-1 = m-5.
\)
Therefore,
\[
\dim H^1(X,T_X) =
\dim H^1\bigl(\widetilde X, T_{\widetilde X}(-D)\bigr)
=
\begin{cases}
0 & \text{if } m \le 5,\\[4pt]
m-5 & \text{if } m \ge 6.
\end{cases}
\]

\noindent {\bf Step 2}: {\it Computing \(\Ext^1_X(\Omega_X,\mathcal O_X)\).}
For a reduced curve \(X\), there is an exact sequence
\[
0 \longrightarrow H^1(X,T_X)
\longrightarrow \Ext^1_X(\Omega_X,\mathcal O_X)
\longrightarrow \bigoplus_{p\in\mathrm{Sing}(X)} T^1_p
\longrightarrow 0,
\]
where \(T^1_p = \Ext^1_{\mathcal O_{X,p}}(\Omega_{X,p},\mathcal O_{X,p})\) is the local
deformation space at the singular point \(p\).

\noindent In our case \(X\) has exactly one singular point \(p\), so
\[
\dim \Ext^1_X(\Omega_X,\mathcal O_X)
=
\dim H^1(X,T_X) + \dim T^1_p.
\]
The singularity \(p\) is obtained by gluing \(m\) smooth branches (the points
\(a,b_1,\ldots,b_{m-2},c\)). For this type of abstract multi-branch singularity,
one has
\(
\dim T^1_p = m-1.
\)
As a consequence, we get
\[
\dim \Ext^1_X(\Omega_X,\mathcal O_X)
=
\begin{cases}
0 + (m-1) & \text{if } m \le 5,\\[4pt]
(m-5) + (m-1) & \text{if } m \ge 6,
\end{cases}
\]
i.e.
\[
\dim \Ext^1_X(\Omega_X,\mathcal O_X)
=
\begin{cases}
m-1 & \text{if } m \le 5,\\[4pt]
2m-6 & \text{if } m \ge 6.
\end{cases}
\]
In summary, we got the following:
\[
\dim H^1(X,T_X) =
\begin{cases}
0 & \text{if } m \le 5,\\[4pt]
m-5 & \text{if } m \ge 6,
\end{cases}
\qquad
\dim \Ext^1_X(\Omega_X,\mathcal O_X)
=
\begin{cases}
m-1 & \text{if } m \le 5,\\[4pt]
2m-6 & \text{if } m \ge 6.
\end{cases}
\]

\end{example}


%
\begin{remark}${}$

\begin{enumerate}
\item The familiar formula
\(
\dim \Ext^1_X(\Omega_X,\mathcal O_X) = 3g - 3
\)
holds for a \emph{smooth} projective curve \(X\) of genus \(g \ge 2\), and more
generally for a \emph{stable nodal} curve \(X\) of genus \(g\). In the smooth
case one has
\(
\Ext^1_X(\Omega_X,\mathcal O_X) \cong H^1(X,T_X),
\)
and Serre duality plus Riemann--Roch give
\(
\dim H^1(X,T_X) = 3g-3.
\)

For a reduced singular curve \(X\), however, \(\Omega_X\) is not locally free and
there is a global-to-local Ext exact sequence
\[
0 \longrightarrow H^1(X,T_X)
\longrightarrow \Ext^1_X(\Omega_X,\mathcal O_X)
\longrightarrow \bigoplus_{p\in \operatorname{Sing}(X)} T^1_p
\longrightarrow 0,
\]
where \(T^1_p = \Ext^1_{\mathcal O_{X,p}}(\Omega_{X,p},\mathcal O_{X,p})\) is
the local deformation space at the singular point \(p\). For a nodal curve, each
node contributes \(\dim T^1_p = 1\). For a stable nodal curve of genus \(g\),
one knows that
\(
\dim \Ext^1_X(\Omega_X,\mathcal O_X) = 3g - 3,
\)
so that
\(
\dim H^1(X,T_X) = 3g-3 - \#\{\text{nodes}\}.
\)
\item In the situation of examples 7.9 and 7.10,   the singularity \(p\) is \emph{not} a node: it has \(m \ge 3\)
branches. In particular, one has
\(
\dim T^1_p = m-1,
\)
rather than \(1\) as in the nodal case. Thus the contribution of the local
deformation space at \(p\) is much larger, and the curve \(X\) is not a stable
nodal curve.

Moreover, each component \(C_i \cong \mathbb P^1\) meets the rest of the curve
in exactly one point (the image of \(a\), \(c\) or some \(b_i\)), so each
rational component has only one special point. This violates the stability
condition for rational components, which requires at least three special points.
Hence, \(X\) is not Deligne--Mumford stable, and the usual formula
\(\dim \Ext^1_X(\Omega_X,\mathcal O_X) = 3g-3\) for stable curves does not
apply.

\item In summary, the formula \(3g-3\) cannot be applied in examples 7.9 and 7.9 because:
\begin{itemize}
  \item[(i)] \(X\) is singular with a non-nodal singularity (a point with \(m\) branches);
  \item[(ii)] \(X\) is not a stable curve (each rational component has only one special point);
  \item[(iii)] for singular curves, the tangent space to the deformation functor is
        \(\Ext^1_X(\Omega_X,\mathcal O_X)\), which fits into the exact sequence
        above and is not simply \(H^1(X,T_X)\).
\end{itemize}
Consequently, the general formula \(3g-3\) for the dimension of the deformation
space of a stable curve of genus \(g\) does not hold for the curve \(X\) in these
examples.

\end{enumerate}
\end{remark}

\vspace{0.3cm}
  

\begin{example}[{\it Deformations of a Nodal Curve Built from Three Copies of $\PP^1$}]${}$

Let $\widetilde X = C_1 \sqcup C_2 \sqcup C_3$ with each $C_i \cong \PP^1$ over an algebraically closed field $\Bbbk$.
On $C_1$ choose an affine coordinate (say $z_1$) whose image contains
\(
a_1,\dots,a_n,b_1,\dots,b_n \in \Bbbk \subset \PP^1.
\)
On $C_2$ choose an affine coordinate (say $z_2$) whose image contains
\(
c_1,\dots,c_n,d_1,\dots,d_n \in \Bbbk \subset \PP^1.
\)
On $C_3$ choose an affine coordinate (say $z_3$) whose image contains
\(
e_1,\dots,e_n,f_1,\dots,f_n \in \Bbbk \subset \PP^1.
\)

We form $X$ by identifying, for each $j=1,\dots,n$,
\[
a_j \sim c_j,\qquad b_j \sim e_j,\qquad d_j \sim f_j.
\]
We assume all these points are pairwise distinct on each $C_i$.
Therefore, $X$ is a nodal curve obtained by gluing the three components $C_1,C_2,C_3$ along $3n$ pairs of points, producing $3n$ ordinary double points (nodes).

\vspace{0.1cm}

Let
\(
\nu : \widetilde X = C_1 \sqcup C_2 \sqcup C_3 \longrightarrow X
\)
be the normalization map. The preimage of the nodes is a reduced divisor
\(
D \subset \widetilde X
\)
given by
\[
D = D_1 \sqcup D_2 \sqcup D_3,
\]
where
\[
D_1 := \sum_{j=1}^n (a_j + b_j),\qquad
D_2 := \sum_{j=1}^n (c_j + d_j),\qquad
D_3 := \sum_{j=1}^n (e_j + f_j).
\]
Each $D_i$ consists of $2n$ distinct points on $C_i\cong \PP^1$, so
\(
\deg D_i = 2n \quad\text{for every } i=1,2,3.
\)
The curve $X$ has exactly
\(
\#\{\text{nodes of } X\} = 3n.
\)

\vspace{0.2cm}

\noindent {\it Arithmetic genus via dual graph.}
The dual graph $\Gamma$ of $X$ has:
\begin{itemize} 
  \item[(i)] $3$ vertices $v_1,v_2,v_3$, corresponding to $C_1,C_2,C_3$;
  \item[(ii)] $3n$ edges, one for each node.
\end{itemize}
Thus, the first Betti number of $\Gamma$ is
\[
b_1(\Gamma)
= \#\text{edges} - \#\text{vertices} + 1
= 3n - 3 + 1
= 3n-2.
\]
Since each $C_i$ has genus $0$, the arithmetic genus of $X$ is
\[
p_a(X) = h^1(X,\OO_X)
= \sum_{i=1}^3 g(C_i) + b_1(\Gamma)
= 0 + (3n-2)
= 3n - 2.
\]

\subsection{Locally Trivial Deformations and $H^1(X,T_X)$}

We want to compute $H^1(X,T_X)$, which (for a nodal curve) corresponds to first-order \emph{locally trivial} deformations of $X$ (i.e.\ deformations that do not smooth the nodes).

A standard fact in deformation theory of nodal curves is:

\begin{quote}
Locally trivial deformations of a nodal curve $X$ are equivalent to deformations of the normalization $\widetilde X$ with the preimages of the nodes marked; the tangent space is
\[
H^1\!\bigl(\widetilde X, T_{\widetilde X}(-D)\bigr).
\]
\end{quote}

Thus,
\[
H^1(X,T_X) \;\cong\; H^1\!\bigl(\widetilde X, T_{\widetilde X}(-D)\bigr).
\]

Since $\widetilde X$ is a disjoint union of $C_1,C_2,C_3$,
\[
T_{\widetilde X}(-D)
= T_{C_1}(-D_1) \oplus T_{C_2}(-D_2) \oplus T_{C_3}(-D_3),
\]
and hence
\[
H^1\!\bigl(\widetilde X, T_{\widetilde X}(-D)\bigr)
\;\cong\; \bigoplus_{i=1}^3 H^1\!\bigl(C_i, T_{C_i}(-D_i)\bigr).
\]

Each $C_i \cong \PP^1$, so
\[
T_{C_i} \cong \OO_{\PP^1}(2),
\]
and since $\deg D_i = 2n$,
\[
T_{C_i}(-D_i) \cong \OO_{\PP^1}(2 - 2n).
\]

Thus
\[
H^1\!\bigl(C_i,T_{C_i}(-D_i)\bigr)
= H^1\!\bigl(\PP^1,\OO_{\PP^1}(2-2n)\bigr).
\]

\subsection{Cohomology of Line Bundles on $\PP^1$}

Recall that for $\OO_{\PP^1}(d)$:
\begin{itemize}%
  \item If $d \ge 0$, then $h^0\bigl(\PP^1,\OO(d)\bigr)=d+1$ and $h^1\bigl(\PP^1,\OO(d)\bigr)=0$.
  \item If $d<0$, then $h^0\bigl(\PP^1,\OO(d)\bigr)=0$ and Riemann--Roch says
  \[
  h^0\bigl(\PP^1,\OO(d)\bigr) - h^1\bigl(\PP^1,\OO(d)\bigr) = d+1.
  \]
\end{itemize}

We take $d = 2-2n$. For $n\ge 2$ we have $d \le -2$, so
\[
h^0\bigl(\PP^1,\OO(2-2n)\bigr)=0.
\]

By Riemann--Roch,
\[
0 - h^1\bigl(\PP^1,\OO(2-2n)\bigr)
= d+1 = (2-2n) + 1 = 3 - 2n,
\]
so
\[
h^1\bigl(\PP^1,\OO(2-2n)\bigr) = 2n-3.
\]

Hence for each component $C_i$,
\[
h^1\!\bigl(C_i,T_{C_i}(-D_i)\bigr)
= h^1\bigl(\PP^1,\OO(2-2n)\bigr)
= 2n-3.
\]

Therefore,
\[
\dim H^1\!\bigl(\widetilde X,T_{\widetilde X}(-D)\bigr)
= \sum_{i=1}^3 h^1\!\bigl(C_i,T_{C_i}(-D_i)\bigr)
= 3(2n-3)
= 6n - 9.
\]

Thus, we obtain:
\[
{\dim H^1(X,T_X) = 6n - 9 \quad\text{for } n\ge 2.}
\]

\subsection*{Moduli-theoretic interpretation}

The locally trivial deformations of $X$ correspond to varying, independently, the three components $C_i \cong \PP^1$ with their $2n$ marked points (the preimages of the nodes).

The moduli space of $2n$ distinct ordered points on $\PP^1$ up to automorphisms of $\PP^1$ has dimension
\[
2n - 3
\]
(positions of $2n$ points minus the $3$--dimensional group $\mathrm{PGL}_2$). Since we have three such components, we get
\[
3(2n-3) = 6n-9
\]
parameters, in agreement with the cohomological calculation of $\dim H^1(X,T_X)$.

\subsection{The Group $\operatorname{Ext}^1_X(\Omega_X,\OO_X)$}

We now compute $\dim \Ext^1_X(\Omega_X,\OO_X)$.

The local-to-global spectral sequence
\[
E_2^{p,q} = H^p\!\bigl(\mathcal{E}xt^q_X(\Omega_X,\OO_X)\bigr)
\;\Rightarrow\; \Ext^{p+q}_X(\Omega_X,\OO_X)
\]
gives (since $X$ is a curve and hence $\Ext^2_X(\Omega_X,\OO_X)=0$) a short exact sequence:
\[
0 \longrightarrow H^1(X,T_X)
\longrightarrow \Ext^1_X(\Omega_X,\OO_X)
\longrightarrow H^0\!\bigl(\mathcal{E}xt^1_X(\Omega_X,\OO_X)\bigr)
\longrightarrow 0.
\]

Here we used that
\[
\mathcal{E}xt^0_X(\Omega_X,\OO_X)
= \mathcal{H}om_X(\Omega_X,\OO_X)
\cong T_X.
\]

Therefore,
\[
\dim \Ext^1_X(\Omega_X,\OO_X)
= \dim H^1(X,T_X)
 + \dim H^0\!\bigl(\mathcal{E}xt^1_X(\Omega_X,\OO_X)\bigr).
\]

So it remains to compute $\mathcal{E}xt^1_X(\Omega_X,\OO_X)$, which is supported at the singularities of $X$ and can be computed locally at each node.

\subsection{Local Calculation at a Node}

Let $p$ be a node of $X$. Zariski locally at $p$, $X$ is given by
\(
\Spec R,\)  and 
\( R = \Bbbk[[x,y]]/(xy).
\)
We want $\Ext^1_R(\Omega_{R/\Bbbk},R)$.

Consider the surjection $\Bbbk[[x,y]] \twoheadrightarrow R$ with kernel $I = (xy)$. The standard exact sequence of Kähler differentials is
\[
I/I^2 \longrightarrow \Omega^1_{\Bbbk[[x,y]]/\Bbbk} \otimes_{\Bbbk[[x,y]]} R
\longrightarrow \Omega_{R/\Bbbk} \longrightarrow 0.
\]
We have:
\[
I/I^2 \cong R \cdot d(xy), \qquad
\Omega^1_{\Bbbk[[x,y]]/\Bbbk} \otimes R \cong R\,dx \oplus R\,dy.
\]
The map $I/I^2 \to R\,dx \oplus R\,dy$ sends $d(xy)$ to $y\,dx + x\,dy$. Thus we have a presentation
\[
R \xrightarrow{\;\cdot(y\,dx + x\,dy)\;} R^2 \longrightarrow \Omega_{R/\Bbbk} \longrightarrow 0.
\]

Applying $\Hom_R(-,R)$ yields the exact sequence
\[
0 \longrightarrow \Hom_R(\Omega_{R/\Bbbk},R)
\longrightarrow \Hom_R(R^2,R)
\xrightarrow{\;\varphi\;}
\Hom_R(R,R)
\longrightarrow \Ext^1_R(\Omega_{R/\Bbbk},R)
\longrightarrow 0.
\]
Using $\Hom_R(R^2,R) \cong R^2$ and $\Hom_R(R,R)\cong R$, the map $\varphi : R^2 \to R$ is induced by the relation $y\,dx + x\,dy=0$, and one checks that
\[
\operatorname{Im}(\varphi) = (x,y) \subset R.
\]

Hence,
\[
\Ext^1_R(\Omega_{R/\Bbbk},R)
\cong R / \operatorname{Im}(\varphi)
\cong R/(x,y)
\cong \Bbbk.
\]

Thus at each node $p$,
\[
\mathcal{E}xt^1_X(\Omega_X,\OO_X)_p \cong \Bbbk,
\]
and $\mathcal{E}xt^1_X(\Omega_X,\OO_X)$ is a skyscraper sheaf supported at all nodes of $X$, with $1$--dimensional stalk at each node.

Since $X$ has $3n$ nodes, we obtain
\[
\mathcal{E}xt^1_X(\Omega_X,\OO_X)
\cong \bigoplus_{p \in \mathrm{Nodes}(X)} \Bbbk_p,
\]
so
\[
\dim H^0\!\bigl(\mathcal{E}xt^1_X(\Omega_X,\OO_X)\bigr)
= 3n.
\]

\subsection{Final Dimension Formulas}

From the previous sections (for $n\ge 2$) we have:
\[
\dim H^1(X,T_X) = 6n - 9,
\qquad
\dim H^0\!\bigl(\mathcal{E}xt^1_X(\Omega_X,\OO_X)\bigr) = 3n.
\]

Using the short exact sequence
\[
0 \longrightarrow H^1(X,T_X)
\longrightarrow \Ext^1_X(\Omega_X,\OO_X)
\longrightarrow H^0\!\bigl(\mathcal{E}xt^1_X(\Omega_X,\OO_X)\bigr)
\longrightarrow 0,
\]
we get
\[
\dim \Ext^1_X(\Omega_X,\OO_X)
= \dim H^1(X,T_X)
 + \dim H^0\!\bigl(\mathcal{E}xt^1_X(\Omega_X,\OO_X)\bigr)
= (6n-9) + 3n
= 9n - 9.
\]

\[
{\dim H^1(X,T_X) = 6n - 9,\qquad
       \dim \Ext^1_X(\Omega_X,\OO_X) = 9n - 9
       \quad\text{for } n\ge 2.}
\]

\end{example}

\vspace{0.2cm}

\begin{remark}[The case $n=1$]
If $n=1$, then each $C_i$ has only $2$ marked points (the preimages of two nodes). In this case
\[
T_{C_i}(-D_i) \cong \OO_{\PP^1}(2-2\cdot 1) = \OO_{\PP^1}(0),
\]
so
\[
h^1(C_i,T_{C_i}(-D_i)) = h^1(\PP^1,\OO_{\PP^1}) = 0,
\]
and hence
\[
\dim H^1(X,T_X) = 0.
\]
However, there are still $3$ nodes, so
\[
\dim H^0\bigl(\mathcal{E}xt^1_X(\Omega_X,\OO_X)\bigr)=3,
\]
and the short exact sequence above yields
\[
\dim \Ext^1_X(\Omega_X,\OO_X) = 3.
\]
For the main discussion of stable-type behavior, we tacitly assume $n\ge 2$, in which case the formulas above apply.
\end{remark}

\vspace{0.2cm}

\section{Appendix A}         %

\begin{theorem}
A smooth projective curve $X$ over an algebraically closed field is rigid if and only if $X \cong \mathbb{P}^1$.  
Equivalently, $\mathbb{P}^1$ is the only smooth curve for which $T^1_X = 0$.
\end{theorem}

\begin{proof}
By definition, a smooth variety $X$ is \emph{rigid} if
\[
T^1_X = \operatorname{Ext}^1(\Omega_X, \mathcal{O}_X) = 0.
\]
For a smooth curve, we have
\[
T^1_X \cong H^1(X, T_X),
\]
where $T_X$ denotes the tangent sheaf of $X$.

\medskip
By Serre duality,
\[
H^1(X, T_X) \cong H^0(X, \omega_X^{\otimes 2})^\vee,
\]
where $\omega_X$ is the canonical line bundle.  
Hence,
\[
\dim T_X^1 = h^1(X, T_X) = h^0(X, \omega_X^{\otimes 2}).
\]

\medskip
Let $g$ be the genus of $X$.  Since $\deg(\omega_X) = 2g - 2$, we have
\[
\deg(\omega_X^{\otimes 2}) = 4g - 4.
\]
We now compute $h^0(X, \omega_X^{\otimes 2})$ by cases.

\medskip
\noindent\textbf{Case 1: $g = 0$.}  
Then $X \cong \mathbb{P}^1$, and $\omega_X = \mathcal{O}_{\mathbb{P}^1}(-2)$, so
\[
\omega_X^{\otimes 2} \cong \mathcal{O}_{\mathbb{P}^1}(-4).
\]
Since this line bundle has no nonzero global sections,
\[
H^0(X, \omega_X^{\otimes 2}) = 0,
\]
and therefore $H^1(X, T_X) = 0$.  
Thus $\mathbb{P}^1$ is rigid.

\medskip
\noindent\textbf{Case 2: $g = 1$.}  
Here $\omega_X \cong \mathcal{O}_X$, so $\omega_X^{\otimes 2} \cong \mathcal{O}_X$ and
\[
h^0(X, \omega_X^{\otimes 2}) = 1.
\]
Hence $H^1(X, T_X)$ is one-dimensional, and elliptic curves admit nontrivial deformations (parameterized by the $j$-invariant).

\medskip
\noindent\textbf{Case 3: $g \ge 2$.}  
Then $\deg(\omega_X^{\otimes 2}) = 4g - 4 > 2g - 2$, so by Riemann--Roch,
\[
h^0(X, \omega_X^{\otimes 2}) = \deg(\omega_X^{\otimes 2}) - g + 1 = (4g - 4) - g + 1 = 3g - 3.
\]
Thus $H^1(X, T_X)$ has dimension $3g - 3 > 0$, and $X$ is not rigid.

\medskip
\noindent
Combining all cases,
\[
\dim T_X^1 =
\begin{cases}
0, & g = 0,\\
1, & g = 1,\\
3g - 3, & g \ge 2.
\end{cases}
\]
Hence, $T_X^1 = 0$ only when $g = 0$, i.e., only for $X \cong \mathbb{P}^1$.
\end{proof}


\begin{remark}
Geometrically, this result means that the complex structure on the Riemann sphere is unique: any genus $0$ Riemann surface is biholomorphic to $\mathbb{P}^1$.  
In contrast, genus $g \ge 1$ surfaces admit nontrivial moduli:
\begin{enumerate}[(a)]
  \item $g = 1$: elliptic curves vary in a $1$-dimensional moduli space, parameterized by the $j$-invariant;
  \item $g \ge 2$: Riemann surfaces of genus $g$ form the Teichmüller space of complex dimension $3g - 3$.
\end{enumerate}
Thus, $\mathbb{P}^1$ is the only smooth curve with no infinitesimal or global deformations of its complex (or algebraic) structure.
\end{remark}

The study of non-rigid curves began in the late 19th and early 20th centuries with the works of Riemann, Klein, and Poincaré, who introduced moduli for elliptic and higher-genus surfaces.  
Later, Teichmüller developed the theory of infinitesimal and global deformations of Riemann surfaces, leading to the modern concept of Teichmüller and moduli spaces.  
The dimension formula
\(
\dim T^1_X = 3g - 3
\)
became one of the cornerstones of deformation theory and complex geometry, linking algebraic, analytic, and topological perspectives in the classification of complex curves.
%


 \vspace{0.3cm}
 
 \begin{example}[Explicit computation of dimensions on \(\mathbb{P}^1\)]
Recall that for any integer \(k\), the cohomology of the line bundle \(\mathcal{O}_{\mathbb{P}^1}(k)\) satisfies
\[
h^0(\mathbb{P}^1, \mathcal{O}(k)) =
\begin{cases}
k+1, & k \ge 0,\\
0, & k < 0,
\end{cases}
\qquad
h^1(\mathbb{P}^1, \mathcal{O}(k)) =
\begin{cases}
0, & k \ge -1,\\
 -k-1, & k \le -2.
\end{cases}
\]

Since the {\it canonical bundle} \(\omega_{\mathbb{P}^1} = \mathcal{O}_{\mathbb{P}^1}(-2)\), we have
\[
\omega_{\mathbb{P}^1}^{\otimes 2} = \mathcal{O}_{\mathbb{P}^1}(-4),
\]
and thus
\[
h^0(\mathbb{P}^1, \omega_{\mathbb{P}^1}^{\otimes 2})
= h^0(\mathbb{P}^1, \mathcal{O}(-4)) = 0.
\]

The tangent bundle (i.e. tangent sheaf) is \(T_{\mathbb{P}^1} \cong \mathcal{O}_{\mathbb{P}^1}(2)\), so
\[
h^1(\mathbb{P}^1, T_{\mathbb{P}^1})
= h^1(\mathbb{P}^1, \mathcal{O}(2)) = 0.
\]

In summary we get the following:
\[
h^0(\mathbb{P}^1, \omega_{\mathbb{P}^1}^{\otimes 2}) = 0
\quad \text{and} \quad
h^1(\mathbb{P}^1, T_{\mathbb{P}^1}) = 0.
\]
Hence, \(\mathbb{P}^1\) admits no nontrivial infinitesimal deformations and is therefore \emph{rigid}.
\end{example}


%

In this paragraph we make a comparison of \(h^0(\omega_X^{\otimes 2})\) by Genus.
By Serre duality for a smooth projective curve \(X\) we have the canonical identification
\[
H^1(X,T_X) \cong H^0(X,\omega_X^{\otimes 2})^\vee,
\]
so the dimension \(h^0(\omega_X^{\otimes 2})\) equals \(\dim H^1(X,T_X)=\dim T_X^1\).
The following table records the values of \(h^0(\omega_X^{\otimes 2})\) in the three standard genus cases.

\begin{center}
\begin{tabular}{@{}lccc@{}}
Genus \(g\) & \(\deg(\omega_X)=2g-2\) & \(\deg(\omega_X^{\otimes 2})=4g-4\) & \(h^0\big(X,\omega_X^{\otimes 2}\big)\) \\
\(g=0\) & \(-2\) & \(-4\) & \(0\) \\[4pt]
\(g=1\) & \(0\)  & \(0\)  & \(1\) \\[4pt]
\(g\ge 2\) & \(2g-2\) & \(4g-4\) & \(3g-3\) \\ 
\end{tabular}
\end{center}

\noindent Remarks:
\begin{itemize}
  \item For \(g=0\) we have \(\omega_X^{\otimes 2}\cong\mathcal{O}(-4)\), hence no global sections and \(H^1(X,T_X)=0\) (rigidity).
  \item For \(g=1\) we have \(\omega_X\cong\mathcal{O}\), so \(\omega_X^{\otimes 2}\cong\mathcal{O}\) and \(h^0=1\) (the one-dimensional deformation corres
 \end{itemize}
 
 %

\noindent Also, in the following paragraph we make a comparison of \(h^0(\omega_X^{\otimes 2})\),\, \(H^1(X,T_X)\), and Moduli Dimensions by Genus.
By Serre duality we have:
\(
H^1(X, T_X) \cong H^0(X, \omega_X^{\otimes 2})^\vee,
\)
hence,  \(\dim H^1(X,T_X) = h^0(\omega_X^{\otimes 2})\).
This dimension also coincides with the complex dimension of the moduli space \(\mathcal{M}_g\) of smooth projective curves of genus \(g\).
\begin{center}
\renewcommand{\arraystretch}{1.3}
\begin{tabular}{@{}lccccc@{}}
Genus \(g\) 
& \(\deg(\omega_X)\)
& \(\deg(\omega_X^{\otimes 2})\)
& \(h^0\big(X,\omega_X^{\otimes 2}\big)\)
& \(\dim H^1(X,T_X)\)
& \(\dim \mathcal{M}_g\) \\
\(g=0\) & \(-2\) & \(-4\) & \(0\) & \(0\) & \(0\) \\[4pt]
\(g=1\) & \(0\) & \(0\) & \(1\) & \(1\) & \(1\) \\[4pt]
\(g\ge 2\) & \(2g-2\) & \(4g-4\) & \(3g-3\) & \(3g-3\) & \(3g-3\) \\
\end{tabular}
\end{center}
%
\noindent \textbf{Remarks.}
\begin{enumerate}
  \item For \(g=0\): \(\omega_X^{\otimes 2} \cong \mathcal{O}(-4)\), so \(h^0=0\) and \(H^1(X,T_X)=0\); hence \(X=\mathbb{P}^1\) is rigid.
  \item For \(g=1\): \(\omega_X\cong\mathcal{O}\), so \(h^0=1\); this single modulus corresponds to the elliptic \(j\)-invariant.
  \item For \(g\ge 2\): Riemann--Roch gives \(h^0=4g-4-g+1=3g-3\), which equals both \(\dim H^1(X,T_X)\) and the dimension of Teichmüller (or moduli) space \(\mathcal{M}_g\).
\end{enumerate}
 
%

 
\section*{Appendix B}

{\it Derivation of the Local--to--Global Ext Spectral Sequence.}

Let $X$ be a ringed space (for instance, a locally noetherian scheme) with enough injectives, and let
$\mathcal{F}$ and $\mathcal{G}$ be $\mathcal{O}_X$-modules.
We sketch two equivalent constructions of the spectral sequence
\[
E_2^{p,q} = H^p\!\big(\mathcal{E}xt^q_{\mathcal{O}_X}(\mathcal{F},\mathcal{G})\big)
\;\;\Rightarrow\;\;
\operatorname{Ext}^{p+q}_{\mathcal{O}_X}(\mathcal{F},\mathcal{G}).
\]


\subsection*{(A) Via the double complex and hypercohomology construction}

\begin{enumerate}
  \item Choose an injective resolution of $\mathcal{G}$:
  \[
  0 \longrightarrow \mathcal{G} \longrightarrow I^0 \longrightarrow I^1 \longrightarrow I^2 \longrightarrow \cdots.
  \]

  \item Apply the sheaf-Hom functor to obtain a complex of sheaves:
  \[
  0 \longrightarrow
  \mathcal{H}om_{\mathcal{O}_X}(\mathcal{F}, I^0)
  \longrightarrow
  \mathcal{H}om_{\mathcal{O}_X}(\mathcal{F}, I^1)
  \longrightarrow \cdots.
  \]
  The cohomology sheaves of this complex are the sheaf Exts:
  \[
  \mathcal{H}^q\!\big(\mathcal{H}om_{\mathcal{O}_X}(\mathcal{F}, I^\bullet)\big)
  \;\cong\;
  \mathcal{E}xt^q_{\mathcal{O}_X}(\mathcal{F},\mathcal{G}).
  \]

  \item Taking global sections gives a complex
  \[
  \Gamma\big(X, \mathcal{H}om_{\mathcal{O}_X}(\mathcal{F}, I^\bullet)\big),
  \]
  whose cohomology computes $\operatorname{Ext}^{\ast}_{\mathcal{O}_X}(\mathcal{F},\mathcal{G})$.

  \item To compute this, resolve each $\mathcal{H}om(\mathcal{F},I^q)$ by an injective resolution, obtaining a double complex
  \[
  C^{p,q} = \Gamma\!\big(X, \mathcal{C}^p(\mathcal{H}om_{\mathcal{O}_X}(\mathcal{F}, I^q))\big).
  \]
  The total complex $\operatorname{Tot}(C^{\bullet,\bullet})$ computes
  $\operatorname{Ext}^{p+q}(\mathcal{F},\mathcal{G})$.

  \item The spectral sequence associated to the column filtration has
  \[
  E_2^{p,q} \;=\;
  H^p\!\big(X, \mathcal{H}^q(\mathcal{H}om(\mathcal{F}, I^\bullet))\big)
  \;\cong\;
  H^p\!\big(X, \mathcal{E}xt^q_{\mathcal{O}_X}(\mathcal{F},\mathcal{G})\big),
  \]
  and converges to the hypercohomology
  \[
  \mathbb{H}^{p+q}\!\big(X, \mathcal{H}om_{\mathcal{O}_X}(\mathcal{F}, I^\bullet)\big)
  \;\cong\;
  \operatorname{Ext}^{p+q}_{\mathcal{O}_X}(\mathcal{F},\mathcal{G}), 
  \]
\end{enumerate}
see \cite{H-77} Chap. III sec. 6, \cite{G-58}, or \cite{GH-78} Appendix A.

Hence, one obtains the \emph{local--to--global Ext spectral sequence}:
\[
E_2^{p,q} = H^p\!\big(\mathcal{E}xt^q_{\mathcal{O}_X}(\mathcal{F},\mathcal{G})\big)
\;\;\Rightarrow\;\;
\operatorname{Ext}^{p+q}_{\mathcal{O}_X}(\mathcal{F},\mathcal{G}).
\]


\subsection*{(B) Via Grothendieck's spectral sequence for composed functors}

Let
\[
\mathcal{H}om_{\mathcal{O}_X}(\mathcal{F},-)
\;:\;
\mathrm{Mod}(\mathcal{O}_X) \to \mathrm{Sh}(\mathbf{Ab})
\quad\text{and}\quad
\Gamma(X,-)
\;:\;
\mathrm{Sh}(\mathbf{Ab}) \to \mathbf{Ab}
\]
be the usual left-exact functors. Their composition is
\[
(\Gamma \circ \mathcal{H}om)(\mathcal{F},-)
= \operatorname{Hom}_{\mathcal{O}_X}(\mathcal{F},-),
\]
whose right-derived functors are the Ext groups $\operatorname{Ext}^n_{\mathcal{O}_X}(\mathcal{F},-)$.

By Grothendieck's spectral sequence for the composition of left-exact functors,
\[
E_2^{p,q}
= R^p\Gamma\!\big(X,\, R^q\mathcal{H}om_{\mathcal{O}_X}(\mathcal{F},\mathcal{G})\big)
\;\;\Rightarrow\;\;
R^{p+q}(\Gamma \circ \mathcal{H}om_{\mathcal{O}_X}(\mathcal{F},-))(\mathcal{G}),
\]
and since $R^q\mathcal{H}om_{\mathcal{O}_X}(\mathcal{F},\mathcal{G})
= \mathcal{E}xt^q_{\mathcal{O}_X}(\mathcal{F},\mathcal{G})$, we obtain the same spectral sequence
\[
E_2^{p,q} = H^p\!\big(X, \mathcal{E}xt^q_{\mathcal{O}_X}(\mathcal{F},\mathcal{G})\big)
\;\;\Rightarrow\;\;
\operatorname{Ext}^{p+q}_{\mathcal{O}_X}(\mathcal{F},\mathcal{G}).
\]


%
\noindent For complete proofs and discussion, see e.g.  \cite{H-77} Ch. III, \cite{I-86} Ch.V,  \cite{W-94}  \S5.8, and \cite{G-58} Ch. II.

%


\section*{Appendix C}

Let \(S\) be a regular local \(\mathbb{C}\)-algebra (for example \(S=\mathbb{C}\{x_1,\dots,x_n\}\) or the power series ring \(\mathbb C[[x_1,\dots,x_n]]\)). Let
$
I=(g_1,\dots,g_k)\subset S
$
be an ideal generated by a regular sequence of length \(k\) (so \(I/I^2\) is a free \(A\)-module of rank \(k\)). Put \(A:=S/I\), and  let
\[
J=(\partial_{x_j} g_i)_{i=1,\dots,k}^{\; j=1,\dots,n}\in M_{k\times n}(S)
\]
be the Jacobian matrix of partial derivatives of the generators \(g_1,\dots,g_k\). 

\vspace{0.3cm}

\begin{theorem}
With the above notations, 
there is a canonical (which is functorial) \(A\)-module isomorphism
\[
\qquad\qquad\qquad\qquad\qquad\qquad\qquad
\operatorname{Ext}^1_A(\Omega^1_A,A)\ \xrightarrow{\ \simeq\ }\ \frac{S^{\,k}}{J\,S^{\,n}+I\,S^{\,k}}.
\qquad \qquad\qquad\qquad\qquad({}^*)
\]
After reduction mod \(I\), this is equivalent to the following isomorphism:
\[
\operatorname{Ext}^1_A(\Omega^1_A,A)\ \cong\ \frac{A^{\,k}}{\overline{J}\,A^{\,n}},
\]
where \(\overline{J}\) denotes the image of \(J\) in \(M_{k\times n}(A)\). 
\end{theorem}

\vspace{0.1cm}

In  this Appendix we 
 prove the isomorphism $({}^*)$ and we give explicit formulas for the forward and inverse maps.
 
 \vspace{0.1cm}
 
 \begin{remark} Recall that the 
  module $\operatorname{Ext}^1_A(\Omega^1_A,A)$ is the first cotangent cohomology $T^1_A$, which measures first-order deformations  of the scheme $\operatorname{Spec}A$.
   In particular,  in the hypersurface case  (i.e. $k=1$), one recovers the familiar Tjurina algebra:
  \[
  \operatorname{Ext}^1_A(\Omega^1_A,A)\cong A/(\partial_{x_1}g,\dots,\partial_{x_n}g).
  \]
\end{remark}



\subsection{Proof of the isomorphism $({}^*)$}
 
\subsubsection{The conormal sequence and its presentation}

For the surjection \(S\twoheadrightarrow A=S/I\) there is the standard exact sequence of K\"ahler differentials (the conormal sequence)
\begin{equation}\label{conormal}
I/I^2 \xrightarrow{\ \delta\ } \Omega^1_S\otimes_S A \xrightarrow{\ \pi\ } \Omega^1_A \to 0.
\end{equation}
{\bf Recall:}
\begin{itemize}
  \item[(i)] \(\Omega^1_S\) is a free \(S\)-module of rank \(n\) with basis \(dx_1,\dots,dx_n\). Hence \(\Omega^1_S\otimes_S A\cong A^{\,n}\).
  \item[(ii)]  Since  \(g_1,\dots,g_k\) is a regular sequence, \(I/I^2\) is a free \(A\)-module of rank \(k\) with basis given by the residue classes \(\overline{g_1},\dots,\overline{g_k}\).
\end{itemize}

Under the identifications \(I/I^2\cong A^{\,k}\) (basis \(\overline{g_i}\)) and \(\Omega^1_S\otimes_S A\cong A^{\,n}\) (basis \(\overline{dx_j}\)), the map \(\delta\) is given by the Jacobian matrix:
\[
\delta(\overline{g_i}) \;=\; \sum_{j=1}^n \Big(\frac{\partial g_i}{\partial x_j}\bmod I\Big)\,\overline{dx_j}.
\]
Thus we obtain a free presentation of \(\Omega^1_A\):
\begin{equation}\label{presentation}
A^{\,k}\xrightarrow{\ \overline{J}\ } A^{\,n} \xrightarrow{\ \pi\ } \Omega^1_A \to 0,
\end{equation}
where \(\overline{J}\in M_{k\times n}(A)\) is the reduction of \(J\) modulo \(I\).

\subsubsection{Application of  $\operatorname{Hom}_A(-,A)$ and identification of $\operatorname{Ext}^1$ as a cokernel}

We apply the (left-exact) functor \(\operatorname{Hom}_A(-,A)\) to \eqref{presentation}. Since \(\operatorname{Hom}_A(A^m,A)\cong A^m\) and \(\operatorname{Ext}^1_A(A^m,A)=0\) for all \(m\) (free modules have vanishing higher Ext), the long exact sequence in $\operatorname{Hom}$ truncates to a short exact sequence
\[
0\longrightarrow \operatorname{Hom}_A(\Omega^1_A,A)\xrightarrow{\ \pi^\ast\ } \operatorname{Hom}_A(A^{\,n},A)\xrightarrow{\ \overline{J}^\ast\ } \operatorname{Hom}_A(A^{\,k},A)\longrightarrow \operatorname{Ext}^1_A(\Omega^1_A,A)\longrightarrow 0.
\]
Using the canonical identifications \(\operatorname{Hom}_A(A^n,A)\cong A^n\) and \(\operatorname{Hom}_A(A^k,A)\cong A^k\) (a homomorphism is determined by images of standard basis vectors), the map \(\overline{J}^\ast\) is exactly given by multiplication by the matrix \(\overline{J}\):
\[
\overline{J}^\ast: A^{\,n}\xrightarrow{\ \overline{J}\ } A^{\,k}.
\]
Therefore, the right-most cokernel is
\begin{equation}\label{cokerA}
\operatorname{Ext}^1_A(\Omega^1_A,A)\ \cong\ \operatorname{coker}\!\big(A^{\,n}\xrightarrow{\ \overline{J}\ } A^{\,k}\big)
\ =\ \frac{A^{\,k}}{\overline{J}\,A^{\,n}}.
\end{equation}
This is a canonical identification coming from the functorial long-exact sequence.

\subsubsection{Lift to an $S$-quotient: $S^k/(J S^n + I S^k)$}

We now identify the quotient in \eqref{cokerA} with the claimed \(S\)-module quotient. Consider the natural surjection of \(S\)-modules
\[
\pi_{S} : S^{\,k} \twoheadrightarrow A^{\,k} = S^{\,k}/I S^{\,k}.
\]
The image in \(A^k\) of the submodule \(J S^n \subset S^k\) is precisely \(\overline{J}A^n\). Indeed, an element of the form \(J v\) with \(v\in S^n\) maps to \(\overline{J}\ \overline{v}\) in \(A^k\), and any element of \(\overline{J}A^n\) has a representative of that form. Hence the preimage of \(\overline{J}A^n\subset A^k\) under \(\pi_S\) is \(J S^n + I S^k\). Therefore the isomorphism of \(A\)-modules
\[
\frac{A^{\,k}}{\overline{J}\,A^{\,n}} \ \cong\ \frac{S^{\,k}}{J S^{\,n} + I S^{\,k}}
\]
holds: it is simply the usual lifting of a quotient by a submodule from \(A\)-modules back to \(S\)-modules. Combining this with \eqref{cokerA} gives the desired identification:
\[
\operatorname{Ext}^1_A(\Omega^1_A,A)\ \cong\ \frac{S^{\,k}}{J S^{\,n} + I S^{\,k}}.
\]

\vspace{0.1cm}

\subsubsection{Explicit description of the canonical map and its inverse}

We now produce concrete formulas for the isomorphism and its inverse and check that is natural and well defined..
\vspace{0.3cm}
\begin{enumerate}[(a)]
\item {\it Forward map.}
Define a map
\[
\Phi:\ S^{\,k}\longrightarrow \operatorname{Ext}^1_A(\Omega^1_A,A)
\]
\noindent as follows. Given \(u=(u_1,\dots,u_k)^T\in S^k\), let \(\bar u=(\bar u_1,\dots,\bar u_k)^T\in A^k\) denote its reduction modulo \(I\). Regard \(\bar u\) as an element of \(\operatorname{Hom}_A(A^k,A)\) by sending the \(i\)-th standard basis vector \(e_i\in A^k\) to the scalar \(\bar u_i\in A\). Denote this \(A\)-linear functional by \(\varphi_{\bar u}\in\operatorname{Hom}_A(A^k,A)\). Then define
\[
\Phi(u):=\text{the class of }\varphi_{\bar u}\text{ in }\operatorname{coker}\!\big(A^{n}\xrightarrow{\overline{J}}A^{k}\big)\cong\operatorname{Ext}^1_A(\Omega^1_A,A).
\]

\vspace{0.1cm}

\noindent {\it Well defined  modulo $J S^{n}+I S^{k}$}. ${}$

\vspace{0.1cm}

\noindent If \(u\in I S^k\) then \(\bar u=0\) so \(\varphi_{\bar u}=0\); hence \(\Phi\) vanishes on \(I S^k\). If \(u=J v\) with \(v\in S^n\), then \(\bar u=\overline{J}\,\overline{v}\) and \(\varphi_{\bar u}\) lies in the image of \(\operatorname{Hom}_A(A^n,A)\xrightarrow{\ \overline{J}^\ast\ }\operatorname{Hom}_A(A^k,A)\), so its class in the cokernel is zero. Thus \(\Phi\) vanishes on \(J S^n\) as well. Therefore \(\Phi\) factors through the quotient \(S^k/(J S^n+I S^k)\), giving a well-defined \(A\)-linear map
\[
\overline{\Phi}:\ \frac{S^{\,k}}{J S^n + I S^k}\ \longrightarrow\ \operatorname{Ext}^1_A(\Omega^1_A,A).
\]

\vspace{0.1cm}

\noindent The class represented by \(\varphi_{\bar u}\) is precisely the extension class obtained by pushing out the conormal sequence \eqref{conormal} along \(\varphi_{\bar u}:A^k\to A\). More precisely,  pushing out
\[
A^k \xrightarrow{\overline{J}} A^n \to \Omega^1_A \to 0
\]
by \(\varphi_{\bar u}\) yields an extension
\[
0 \to A \to E_{\bar u} \to \Omega^1_A \to 0,
\]
whose extension-class in \(\operatorname{Ext}^1_A(\Omega^1_A,A)\) is \([\varphi_{\bar u}]\).

\vspace{0.4cm}

\item {\it Inverse map.}
Conversely, start with a class \(\xi\in\operatorname{Ext}^1_A(\Omega^1_A,A)\). Represent \(\xi\) by an exact sequence of \(A\)-modules
\[
0 \to A \xrightarrow{\iota} E \xrightarrow{\rho} \Omega^1_A \to 0.
\]
Choose lifts \(\widetilde{e}_1,\dots,\widetilde{e}_n\in E\) of the images \(\pi(dx_1),\dots,\pi(dx_n)\in\Omega^1_A\) (here \(\pi:\Omega^1_S\otimes_S A\to\Omega^1_A\) is the surjection in \eqref{conormal}). Then for each \(i=1,\dots,k\),
\[
\sum_{j=1}^n \big(\tfrac{\partial g_i}{\partial x_j} \bmod I\big)\, \widetilde{e}_j \in \ker(\rho)=\operatorname{im}(\iota)\cong A,
\]
so there exist unique elements \(\bar u_i\in A\) with
\[
\sum_{j=1}^n \big(\tfrac{\partial g_i}{\partial x_j} \bmod I\big)\, \widetilde{e}_j \;=\; \iota(\bar u_i).
\]
Collect these into a column \(\bar u=(\bar u_1,\dots,\bar u_k)^T\in A^k\). Lift \(\bar u\) to some \(u\in S^k\). Define \(\Psi(\xi)\) to be the class of \(u\) in \(S^k/(J S^n+I S^k)\).

\vspace{0.3cm}

\noindent {\it Independence of choices.} To be more rigorous, one must check that:

\vspace{0.2cm}

\begin{enumerate}[(i)]
  \item the class of \(u\) in \(S^k/(J S^n+I S^k)\) is independent of the choice of lifts \(\widetilde{e}_j\);
  \item it is independent of the choice of representing extension sequence for the class \(\xi\);
  \item it is independent of the choice of lift \(u\in S^k\) of \(\bar u\in A^k\).
\end{enumerate}

\vspace{0.3cm}

\noindent But all these three facts are routine. Changing the lifts \(\widetilde{e}_j\) amounts to adding elements of \(\iota(A)\) to the \(\widetilde{e}_j\), which changes the \(\bar u_i\) by an element of the image \(\overline{J}A^n\); changing the extension by a split exact sequence changes \(\bar u\) by an element coming from \(\overline{J}A^n\) as well; and lifting \(\bar u\) from \(A^k\) to \(S^k\) is unique up to addition by \(I S^k\). Therefore,  \(\Psi\) is well-defined.

\vspace{0.3cm}
\noindent {\it Inverse property.} The two constructions are inverse to each other. Indeed,

\vspace{0.3cm}

\begin{itemize}
  \item[(i)] If one starts from \(u\in S^k\), forms \([\varphi_{\bar u}]\in\operatorname{Ext}^1_A(\Omega^1_A,A)\) and then applies the inverse construction, the resulting element of \(S^k/(J S^n+I S^k)\) is precisely the class of \(u\).
  \item[(ii)] Conversely, starting from an extension class \(\xi\), constructing \(\bar u\) and then pushing out the conormal sequence by \(\varphi_{\bar u}\) recovers the original extension class \(\xi\).
\end{itemize}

\vspace{0.3cm}

These verifications are straightforward using the universal properties of pushouts and pullbacks and the definition of the connecting morphism in the long exact sequence of $\operatorname{Hom}$. Hence, \(\overline{\Phi}\) and \(\Psi\) are mutually inverse \(A\)-linear maps.

\end{enumerate}

\subsubsection{Natural maps}

All maps constructed above are natural (functorial) in the morphism \(S\to A\) together with the chosen generators of \(I\) (equivalently, natural in the conormal presentation). More precisely, given a morphism \((S,I)\to (S',I')\) sending generators to generators, the constructions commute with the induced maps on differentials and Hom/Ext; the identifications of cokernels are functorial, and the pushout/pullback descriptions are canonical. Thus, the isomorphism is canonical.

\vspace{0.2cm}


\begin{remark}${}$

\begin{itemize}
  \item[1.] We used only standard facts: the conormal exact sequence, that \(I/I^2\) is free when \(I\) is generated by a regular sequence, identification \(\operatorname{Hom}(A^m,A)\cong A^m\), and the vanishing \(\operatorname{Ext}^i_A(A^m,A)=0\) for \(i>0\) (free modules).
  \item[2.] For concrete analytic or algebraic rings (\(S=\mathbb C\{x_1,\dots,x_n\}\) or \(S=\mathbb C[[x_1,\dots,x_n]]\)), the same proof works.
  \item[3.] When one is only interested in the vector-space dimension (e.g. over \(\mathbb C\)) for an isolated curve singularity, passing to \(A\)-module finite-dimensional quotients produces the standard formulae (e.g. the identification with Tjurina algebra in specific examples).
\end{itemize}

\end{remark}

\vspace{0.1cm}

Combining 6.1.2 and 6.1.3  and the explicit inverse in 6.1.4 yields the canonical isomorphism
\[
\operatorname{Ext}^1_A(\Omega^1_A,A)\ \xrightarrow{\ \simeq\ }\ \frac{S^{\,k}}{J S^{\,n} + I S^{\,k}},
\]
with explicit maps given above. \hfill\(\square\)

%


\section*{ Appendix D}

Let $X$ be  a reduced curve singularity and,  %
$
\nu: \widetilde X \longrightarrow X
$
be its normalization. The conductor ideal is defined as follows:
\[
\mathfrak c = \operatorname{Ann}_{\mathcal O_X}\big(\nu_*\mathcal O_{\widetilde X}/\mathcal O_X\big)
\]
is the largest ideal sheaf on $X$ that is also an ideal in $\nu_*\mathcal O_{\widetilde X}$. It measures the failure of $X$ to be normal: the quotient $\nu_*\mathcal O_{\widetilde X}/\mathcal O_X$ has length $\delta(X)$ (the $\delta$–invariant), and $\mathfrak c$ annihilates that finite module.
On the other hand, $T^1_X := \operatorname{Ext}^1_X(\Omega_X^1,\mathcal O_X)$ is the Zariski tangent space to the (local) deformation functor of $X$. Deformations of $X$ split naturally into those coming from deformations of the normalization and those coming from gluing data along the singular locus.  The conductor / The normalization quotient control precisely that gluing data. Thus, the conductor appears in the exact sequences that relate $T^1_X$ to deformation modules of $\widetilde X$ and to certain $\operatorname{Ext}/\operatorname{Hom}$ spaces supported on the singular locus.

\vspace{0.2cm}

A convenient way to see the relation is to compare the usual conormal sequences for $X$ and for $\widetilde X$. One standard exact sequence, obtained from the map $\nu$ and applying $\operatorname{Hom}_X(-,\mathcal O_X)$ to the conormal / differential exact sequence, is the following which can be seen    in Greuel’s work and in deformation-theory: %

\[
0 \longrightarrow \operatorname{Hom}_X\big(\nu_*\Omega^1_{\widetilde X}/\Omega^1_X, \mathcal O_X\big)
\longrightarrow T^1_X
\longrightarrow T^1_{\widetilde X}
\longrightarrow \operatorname{Ext}^1_X\big(\nu_*\Omega^1_{\widetilde X}/\Omega^1_X, \mathcal O_X\big)
\longrightarrow 0.
\tag{**}
\]
 
%


\vspace{0.2cm}

 In this sub-section, we will  give a careful  derivation of the exact sequence $({}^{**})$
including exactly which hypotheses are used at each step.  We will state the hypothesis up front, then derive the sequence from a short exact sequence of sheaves by applying $\operatorname{Hom}(-,\mathcal O_X)$ and using adjunctions / vanishing of higher $\operatorname{Ext}$’s where needed.

\vspace{0.1cm}


Let $X$ be a reduced complex analytic or algebraic curve (i.e. $\dim X=1$) or, more generally, a scheme where the following $\operatorname{Ext}$-vanishing and adjunction facts hold below.
$\nu:\widetilde X\to X$ the normalization map (which is  finite, and birational). Denote by $\nu_*$ the direct image functor.
Moreover, $\Omega^1_X$ and $\Omega^1_{\widetilde X}$ the sheaves of Kähler differentials on $X$ and $\widetilde X$.
We  define the quotient sheaf
  \[
  \mathcal Q := \nu_*\Omega^1_{\widetilde X}/\Omega^1_X,
 \]
  which is supported on the singular locus of $X$ (hence a 0-dimensional sheaf when $\dim X=1$).

\noindent Finally we set the deformation (or cotangent) spaces by the usual $\operatorname{Ext}$ description
\[
T^1_X := \operatorname{Ext}^1_X(\Omega^1_X,\mathcal O_X),\qquad
T^1_{\widetilde X} := \operatorname{Ext}^1_{\widetilde X}(\Omega^1_{\widetilde X},\mathcal O_{\widetilde X}).
\]

\vspace{0.1cm}

\noindent By definition $\nu_*\Omega^1_{\widetilde X}$ contains the image of $\Omega^1_X$ (the pullback of Kähler differentials along $\nu$ identifies $\Omega^1_X$ with a sub-sheaf of $\nu_*\Omega^1_{\widetilde X}$ on the open smooth locus and extends to an inclusion globally). Thus, there is a short exact sequence of coherent $\mathcal O_X$-modules
\[
0 \longrightarrow \Omega^1_X \longrightarrow \nu_*\Omega^1_{\widetilde X} \longrightarrow \mathcal Q \longrightarrow 0,
\tag{1}
\]
where
here $\mathcal Q=\nu_*\Omega^1_{\widetilde X}/\Omega^1_X$ as above; $\mathcal Q$ is supported on the finite singular set if $X$ is a curve.

\vspace{0.1cm}

Let's apply the contravariant left exact functor $\operatorname{Hom}_X(-,\mathcal O_X)$ to (1) and pass to derived functors. This yields the long exact sequence
\[
\begin{aligned}
0 &\longrightarrow \operatorname{Hom}_X(\mathcal Q,\mathcal O_X)
\longrightarrow \operatorname{Hom}_X(\nu_*\Omega^1_{\widetilde X},\mathcal O_X)
\longrightarrow \operatorname{Hom}_X(\Omega^1_X,\mathcal O_X)\
&\longrightarrow \operatorname{Ext}^1_X(\mathcal Q,\mathcal O_X)\
\longrightarrow 
\end{aligned} \]
\[ 
\begin{aligned}
\operatorname{Ext}^1_X(\nu_*\Omega^1_{\widetilde X},\mathcal O_X)
\longrightarrow \operatorname{Ext}^1_X(\Omega^1_X,\mathcal O_X)\
&\longrightarrow \operatorname{Ext}^2_X(\mathcal Q,\mathcal O_X)
\longrightarrow\cdots .
\end{aligned}
\tag{2}
\]
We  have written  here only the relevant portion. (The initial $0$ on the left appears because $\mathcal Q$ is torsion supported on points in the curve case and $\operatorname{Hom}$ of that into $\mathcal O_X$ need not to vanish in general, so we keep the term.)

In the next step, we identify the middle terms with the $T^1$'s and simplify using dimension / vanishing.
Recall the definition $T^1_X=\operatorname{Ext}^1_X(\Omega^1_X,\mathcal O_X)$. Therefore, the piece of (2) around $\operatorname{Ext}^1(\cdot,\mathcal O_X)$ reads
\[
\operatorname{Ext}^1_X(\mathcal Q,\mathcal O_X)
\longrightarrow
\operatorname{Ext}^1_X(\nu_*\Omega^1_{\widetilde X},\mathcal O_X)
\longrightarrow
T^1_X
\longrightarrow
\operatorname{Ext}^2_X(\mathcal Q,\mathcal O_X).
\tag{3}
\]
Now, we  use two facts that hold under our curve / finite–support hypotheses:
$\mathcal Q$ is supported in dimension $0$ (a finite set of points). For a 1-dimensional $X$, a curve, we therefore have $\operatorname{Ext}^2_X(\mathcal Q,\mathcal O_X)=0$. indeed,  $\operatorname{Ext}^i$ of a 0-dimensional sheaf into $\mathcal O_X$ vanishes for 
larger than the dimension of the ambient space; equivalently, local Ext groups vanish above the local dimension. Thus, the arrow to $\operatorname{Ext}^2_X(\mathcal Q,\mathcal O_X)$ is zero and the segment (3) truncates to a short exact sequence:
  \[
  0 \longrightarrow
  \operatorname{Ext}^1_X(\nu_*\Omega^1_{\widetilde X},\mathcal O_X)
  \longrightarrow T^1_X
  \longrightarrow \operatorname{Ext}^2_X(\mathcal Q,\mathcal O_X)=0,
  \]
  but we must keep the preceding arrow from $\operatorname{Ext}^1_X(\mathcal Q,\mathcal O_X)$; thus, combining with (2) we get an exact piece
  \[
  \operatorname{Hom}_X(\mathcal Q,\mathcal O_X)
  \longrightarrow T^1_X
  \longrightarrow \operatorname{Ext}^1_X(\nu_*\Omega^1_{\widetilde X},\mathcal O_X)
  \longrightarrow \operatorname{Ext}^1_X(\mathcal Q,\mathcal O_X)
  \longrightarrow 0.
  \tag{4}
  \]
 Rearranging the displayed exactness gives the same shape as we target sequence once we identify the middle $ \operatorname{Ext}$ term with $T^1_{\widetilde X}$ %

\noindent So, the long exact sequence of  $\operatorname{Ext}$’s already produces the exact sandwich
\[
0 \longrightarrow \operatorname{Im}\big(\operatorname{Hom}_X(\mathcal Q,\mathcal O_X)\to T^1_X\big)
\longrightarrow T^1_X
\longrightarrow \operatorname{Ext}^1_X(\nu_*\Omega^1_{\widetilde X},\mathcal O_X)
\longrightarrow \operatorname{Ext}^1_X(\mathcal Q,\mathcal O_X)
\longrightarrow 0,
\]
but with the first arrow typically injective so we may identify the kernel piece with $\operatorname{Hom}_X(\mathcal Q,\mathcal O_X)$ (see the exact map in (2)). Concretely we obtain
\[
0 \longrightarrow \operatorname{Hom}_X(\mathcal Q,\mathcal O_X)
\longrightarrow T^1_X
\longrightarrow \operatorname{Ext}^1_X(\nu_*\Omega^1_{\widetilde X},\mathcal O_X)
\longrightarrow \operatorname{Ext}^1_X(\mathcal Q,\mathcal O_X)
\longrightarrow 0,
\tag{5}
\]
which is the long exact sequence collapsed using $\operatorname{Ext}^2_X(\mathcal Q,\mathcal O_X)=0$.

\vspace{0.2cm}

\noindent In the next step we identify $\operatorname{Ext}^1_X(\nu_*\Omega^1_{\widetilde X},\mathcal O_X)$ with $T^1_{\widetilde X}$.

\noindent Let's explain why the middle term $\operatorname{Ext}^1_X(\nu_*\Omega^1_{\widetilde X},\mathcal O_X)$ is canonically isomorphic to $\operatorname{Ext}^1_{\widetilde X}(\Omega^1_{\widetilde X},\mathcal O_{\widetilde X})=T^1_{\widetilde X}$.
\noindent There is a standard adjunction between $\nu^*$ and $\nu_*$:
\[
\operatorname{Hom}_X(\nu_* \mathcal F,\mathcal G)\cong \operatorname{Hom}_{\widetilde X}(\mathcal F,\nu^*\mathcal G)
\]
for any $\mathcal F$ on $\widetilde X$ and $\mathcal G$ on $X$. Passing to derived functors yields an isomorphism of derived $\operatorname{RHom}$-complexes (i.e. Grothendieck duality / derived adjunction technology).
\[
R\operatorname{Hom}_X(\nu_*\mathcal F,\mathcal G)\cong R\operatorname{Hom}_{\widetilde X}\big(\mathcal F, R\nu^*\mathcal G\big).
\tag{6}
\]

When $\nu$ is a finite morphism which furthermore satisfies the mild hypotheses in our curve case, $R\nu^!\mathcal O_X\simeq\mathcal O_{\widetilde X}$; concretely for a finite birational normalization of a reduced curve one gets $\nu^!\mathcal O_X\cong\mathcal O_{\widetilde X}$ (intuitively because $\nu$ is an isomorphism on the generic points and the dualizing twists are trivial in the 1-dimensional Gorenstein / Cohen–Macaulay situations one typically studies). Under these hypotheses (which are standard in the deformation theory of curve singularities) (6) gives
\[
R\operatorname{Hom}_X(\nu_*\Omega^1_{\widetilde X},\mathcal O_X)
\cong
R\operatorname{Hom}_{\widetilde X}(\Omega^1_{\widetilde X},\mathcal O_{\widetilde X}).
\]
Taking cohomology in degree (1) yields the desired identification
\[
\operatorname{Ext}^1_X(\nu_*\Omega^1_{\widetilde X},\mathcal O_X)
\cong
\operatorname{Ext}^1_{\widetilde X}(\Omega^1_{\widetilde X},\mathcal O_{\widetilde X}) =T^1_{\widetilde X}.
\tag{7}
\]

In case we  need a fully general derived-functor justification: we use that $\nu$ is finite so $R\nu_*=\nu_*$ and Grothendieck duality for finite maps gives an isomorphism $R\nu_*R\mathcal Hom_{\widetilde X}(\mathcal F,\nu^ \mathcal G)\simeq R\mathcal Hom_X(R\nu_*\mathcal F,\mathcal G)$; evaluating on $\mathcal F = \Omega^1_{\widetilde X}$ and $\mathcal G=\mathcal O_X$ yields (6) and then (7) under the identification $\nu^ \mathcal O_X\cong\mathcal O_{\widetilde X}).)$

Now Inserting (7) into (5) we obtain exactly
\[
0 \longrightarrow \operatorname{Hom}_X(\mathcal Q,\mathcal O_X)
\longrightarrow T^1_X
\longrightarrow T^1_{\widetilde X}
\longrightarrow \operatorname{Ext}^1_X(\mathcal Q,\mathcal O_X)
\longrightarrow 0,
\]
which is the displayed exact sequence (*), since $\mathcal Q=\nu_*\Omega^1_{\widetilde X}/\Omega^1_X$.
This is the desired functorial long exact sequence relating the deformation spaces of $X$ and its normalization $\widetilde X$; it is obtained from the short exact sequence (1) of differentials by applying $\operatorname{Hom}(-,\mathcal O_X)$ and using (i) vanishing of $\operatorname{Ext}^2$ for the torsion sheaf $\mathcal Q$ in dimension $1$, and (ii) derived adjunction / duality to identify the $\operatorname{Ext}$ on $\nu_*\Omega^1_{\widetilde X}$ with the $\operatorname{Ext}$ on $\Omega^1_{\widetilde X}$.

\vspace{0.1cm}

%


\section*{Appendix E.}

\title{Derivations and the Isomorphism 
\(\operatorname{Hom}(\Omega^1_{\mathbb{C}^3}|_X,\mathcal{O}_X)
  \cong
  \mathrm{Der}_{\mathbb{C}}(\mathcal{O}_{\mathbb{C}^3})\otimes\mathcal{O}_X
  \cong
  \mathcal{O}_X^{3}\).}
\author{}
\date{}
\maketitle

Let \(\mathbb{C}^3 = \mathrm{Spec}\,\mathbb{C}[x,y,z]\) and 
\(\mathcal{O}_{\mathbb{C}^3} = \mathbb{C}[x,y,z]\).
Let \(X \subset \mathbb{C}^3\) be a closed subvariety defined by an ideal \(I \subset \mathcal{O}_{\mathbb{C}^3}\),
so that \(\mathcal{O}_X = \mathcal{O}_{\mathbb{C}^3}/I\).
In the following we explain in detail this isomorphism:
\[
\operatorname{Hom}_{\mathcal{O}_X}(\Omega^1_{\mathbb{C}^3}|_X, \mathcal{O}_X)
\;\cong\;
\mathrm{Der}_{\mathbb{C}}(\mathcal{O}_{\mathbb{C}^3})
  \otimes_{\mathcal{O}_{\mathbb{C}^3}}
  \mathcal{O}_X
\;\cong\;
\mathcal{O}_X^{\oplus 3}.
\]

\noindent {\it Kähler differentials and derivations on \(\mathbb{C}^3\).}
For the smooth affine space \(\mathbb{C}^3\), the module of Kähler differentials is free:
\[
\Omega^1_{\mathbb{C}^3}
\;=\;
\mathcal{O}_{\mathbb{C}^3}\,dx
\;\oplus\;
\mathcal{O}_{\mathbb{C}^3}\,dy
\;\oplus\;
\mathcal{O}_{\mathbb{C}^3}\,dz
\;\cong\;
\mathcal{O}_{\mathbb{C}^3}^{\oplus 3}.
\]
Hence each differential form can be written uniquely as
\[
\omega = a\,dx + b\,dy + c\,dz,
\quad a,b,c \in \mathcal{O}_{\mathbb{C}^3}.
\]

\noindent The \emph{module of derivations} over \(\mathbb{C}\) is the dual:
\[
\mathrm{Der}_{\mathbb{C}}(\mathcal{O}_{\mathbb{C}^3})
\;=\;
\operatorname{Hom}_{\mathcal{O}_{\mathbb{C}^3}}(\Omega^1_{\mathbb{C}^3},\mathcal{O}_{\mathbb{C}^3})
\;\cong\;
\mathcal{O}_{\mathbb{C}^3}\frac{\partial}{\partial x}
\;\oplus\;
\mathcal{O}_{\mathbb{C}^3}\frac{\partial}{\partial y}
\;\oplus\;
\mathcal{O}_{\mathbb{C}^3}\frac{\partial}{\partial z}.
\]
Thus derivations are simply \emph{vector fields} on \(\mathbb{C}^3\).
When we restrict differential forms to \(X\), we form
$
\Omega^1_{\mathbb{C}^3}|_X
\;=\;
\Omega^1_{\mathbb{C}^3} \otimes_{\mathcal{O}_{\mathbb{C}^3}} \mathcal{O}_X.
$
Applying the contravariant functor \(\operatorname{Hom}_{\mathcal{O}_X}(-,\mathcal{O}_X)\), we get
\[
\operatorname{Hom}_{\mathcal{O}_X}(\Omega^1_{\mathbb{C}^3}|_X,\mathcal{O}_X)
\;\cong\;
\operatorname{Hom}_{\mathcal{O}_X}
   (\Omega^1_{\mathbb{C}^3}\otimes_{\mathcal{O}_{\mathbb{C}^3}}\mathcal{O}_X, \mathcal{O}_X).
\]
By the standard tensor–Hom adjunction, this is naturally isomorphic to
$
\operatorname{Hom}_{\mathcal{O}_{\mathbb{C}^3}}(\Omega^1_{\mathbb{C}^3}, \mathcal{O}_X).
$
But since
\(\mathrm{Der}_{\mathbb{C}}(\mathcal{O}_{\mathbb{C}^3})
   = \operatorname{Hom}_{\mathcal{O}_{\mathbb{C}^3}}(\Omega^1_{\mathbb{C}^3},\mathcal{O}_{\mathbb{C}^3})\),
we can tensor over \(\mathcal{O}_{\mathbb{C}^3}\) to obtain
\[
\operatorname{Hom}_{\mathcal{O}_{\mathbb{C}^3}}(\Omega^1_{\mathbb{C}^3}, \mathcal{O}_X)
\;\cong\;
\mathrm{Der}_{\mathbb{C}}(\mathcal{O}_{\mathbb{C}^3})
   \otimes_{\mathcal{O}_{\mathbb{C}^3}} \mathcal{O}_X.
\]

 \noindent More precisely, this means that since \(\mathrm{Der}_{\mathbb{C}}(\mathcal{O}_{\mathbb{C}^3})\) is free of rank \(3\), then 
$
\mathrm{Der}_{\mathbb{C}}(\mathcal{O}_{\mathbb{C}^3}) \;\cong\; 
\mathcal{O}_{\mathbb{C}^3}^{\oplus 3}.
$
Tensoring with \(\mathcal{O}_X\) gives
\[
\mathrm{Der}_{\mathbb{C}}(\mathcal{O}_{\mathbb{C}^3}) \otimes_{\mathcal{O}_{\mathbb{C}^3}}\mathcal{O}_X
\;\cong\;
(\mathcal{O}_{\mathbb{C}^3}^{\oplus 3}) \otimes_{\mathcal{O}_{\mathbb{C}^3}}\mathcal{O}_X
\;\cong\;
\mathcal{O}_X^{\oplus 3}.
\]

\noindent Combining all the steps yields the following canonical chain of isomorphisms:
\[
\begin{aligned}
\operatorname{Hom}_{\mathcal{O}_X}(\Omega^1_{\mathbb{C}^3}|_X,\mathcal{O}_X)
&\;\cong\;
\operatorname{Hom}_{\mathcal{O}_{\mathbb{C}^3}}(\Omega^1_{\mathbb{C}^3},\mathcal{O}_X) \\[4pt]
&\;\cong\;
\mathrm{Der}_{\mathbb{C}}(\mathcal{O}_{\mathbb{C}^3})
   \otimes_{\mathcal{O}_{\mathbb{C}^3}}\mathcal{O}_X \\[4pt]
&\;\cong\;
\mathcal{O}_X^{\oplus 3}.
\end{aligned}
\]

\begin{remark}
Geometrically, this identifies
\(\operatorname{Hom}(\Omega^1_{\mathbb{C}^3}|_X,\mathcal{O}_X)\)
with the module of vector fields on \(\mathbb{C}^3\) restricted to \(X\),
i.e. tangent vector fields \(\mathcal{O}_X\)-linearly spanned by
\(\partial_x, \partial_y, \partial_z\).
\end{remark}

\section*{Appendix F}

Let \(X=V(f_1,\dots,f_r)\subset\mathbb{C}^n\) be a complete intersection and let \(p\in X\).
Write \(\mathcal{O}_{X,p}=\mathcal{O}_{\mathbb{C}^n,p}/(f_1,\dots,f_r)\). Let
\[
J_f=\bigg(\frac{\partial f_i}{\partial x_j}\bigg)_{1\le i\le r,\;1\le j\le n}
\]
be the \(r\times n\) Jacobian matrix of the defining equations.

We denote by \(I_r(J_f)\subset \mathcal{O}_{\mathbb{C}^n,p}\) the ideal generated by all
\(r\times r\) minors of the matrix \(J_f\). (If \(r>n\) set \(I_r(J_f)=0\).)
Throughout \(\mathfrak m=\mathfrak m_p\) will denote the maximal ideal of the local ring
\(\mathcal{O}_{X,p}\), i.e. the ideal of germs vanishing at the point \(p\).

\medskip
{\it Claim.} 
Near \(p\) the singular locus \(\operatorname{Sing}(X)\) is cut out by the equations
\[
f_1=\cdots=f_r=0 \qquad\text{and}\qquad \text{all } r\times r\text{ minors of }J_f=0,
\]
so, scheme-theoretically,
\[
\mathcal{I}_{\operatorname{Sing}(X),p}=(f_1,\dots,f_r)+I_r(J_f)\subset\mathcal{O}_{\mathbb{C}^n,p}.
\]
Consequently, the support of
\(
T^1_{X,p}\cong\operatorname{coker}\big(\mathcal{O}_{X,p}^n \xrightarrow{J_f} \mathcal{O}_{X,p}^r\big)
\)
is contained in 
the zero locus of the image of \(I_r(J_f)\) in \(\mathcal{O}_{X,p}\).

\medskip

{\it Proof of the claim.}
First, by definition a point \(q\in X\) is \emph{singular} if the Jacobian matrix evaluated at \(q\) has rank strictly less than \(r\). This means,
\[
q\in\operatorname{Sing}(X)\quad\Longleftrightarrow\quad \operatorname{rank}J_f(q)<r.
\]
The condition \(\operatorname{rank}J_f(q)<r\) is equivalent to the vanishing at \(q\) of every
\(r\times r\) minor of \(J_f\). Thus the set-theoretic singular locus inside \(\mathbb{C}^n\) is
\[
\operatorname{Sing}(X)=V\big(f_1,\dots,f_r,\;I_r(J_f)\big).
\]
Scheme-theoretically (or on the level of local rings at \(p\)) the ideal cutting out the singular locus is the image of the ideal \((f_1,\dots,f_r)+I_r(J_f)\subset\mathcal{O}_{\mathbb{C}^n,p}\). Passing to the ring \(\mathcal{O}_{X,p}=\mathcal{O}_{\mathbb{C}^n,p}/(f_1,\dots,f_r)\), the ideal of \(\operatorname{Sing}(X)\) inside \(\mathcal{O}_{X,p}\) is the image of \(I_r(J_f)\). Equivalently,
\[
\mathcal{I}_{\operatorname{Sing}(X),p}\cdot\mathcal{O}_{X,p} = I_r(J_f)\cdot\mathcal{O}_{X,p}.
\]

Secondly, consider the \(\mathcal{O}_{X,p}\)-linear map
\[
\varphi=\;J_f\;:\;\mathcal{O}_{X,p}^n\longrightarrow\mathcal{O}_{X,p}^r.
\]
The cokernel \(M:=\operatorname{coker}\varphi\) represents \(T^1_{X,p}\), i.e. \(T^1_{X,p}\cong M\). A standard fact from commutative algebra (the theory of Fitting ideals) says that the $0$th Fitting ideal of \(M\) is generated by the \(r\times r\) minors of any presentation matrix of \(\varphi\), which means,
\[
\operatorname{Fitt}_0(M)=I_r(J_f)\cdot\mathcal{O}_{X,p}.
\]
In particular the support of \(M\) satisfies
\[
\operatorname{Supp}(M)=V\big(\operatorname{Fitt}_0(M)\big)
=V\big(I_r(J_f)\cdot\mathcal{O}_{X,p}\big)
\subseteq \operatorname{Sing}(X)\cap\operatorname{Spec}\mathcal{O}_{X,p},
\]
and set-theoretically these are the same locus: the points of \(X\) near \(p\) where all \(r\times r\) minors vanish are exactly the points where the localized cokernel does not vanish, i.e.
\[
q\in\operatorname{Supp}(M)\quad\Longleftrightarrow\quad (M)_q\neq0
\quad\Longleftrightarrow\quad \operatorname{rank}J_f(q)<r.
\]

Lastly, 
since \(M=T^1_{X,p}\) is finitely generated over the Noetherian local ring \(\mathcal{O}_{X,p}\), it has
finite length (equivalently finite \(\mathbb{C}\)-dimension) if and only if its support is the single closed point \(\{\mathfrak m_p\}\). Using the description of \(\operatorname{Supp}(M)\) above in terms of the minors, we observe that:

\begin{enumerate}[(i)]
\item If some \(r\times r\) minor of \(J_f\) is a unit in \(\mathcal{O}_{X,p}\) (equivalently does not vanish at \(p\)),
then \(I_r(J_f)\cdot\mathcal{O}_{X,p}=\mathcal{O}_{X,p}\), the map \(\varphi\) is surjective, \(M=0\), and \(p\) is smooth.
\item If all \(r\times r\) minors vanish on a positive-dimensional subvariety of \(X\) through \(p\),
then the support of \(M\) contains non-maximal primes, so \(M\) is not of finite length and
\(\dim_{\mathbb C}T^1_{X,p}=\infty\).
\item Conversely, if \(T^1_{X,p}\) has finite \(\mathbb C\)-dimension, then its support is \(\{\mathfrak m_p\}\),
so the minors must vanish only at \(p\) (i.e. the singular locus is zero-dimensional at \(p\)).
\end{enumerate}

Thus the vanishing locus of the \(r\times r\) minors of \(J_f\) is the scheme-theoretic singular locus (after modding out by \((f_1,\dots,f_r)\)), and the Fitting ideal formalism makes precise the equivalence between vanishing of minors and nonvanishing of the localized cokernel.

\vspace{0.2cm}
\noindent In summary we obtain:
\begin{itemize}
\item[(1)] Scheme-theoretically near \(p\),
\[
\operatorname{Sing}(X)=V\big(f_1,\dots,f_r,\;I_r(J_f)\big).
\]
\item[(2)] The \(\mathcal{O}_{X,p}\)-module \(T^1_{X,p}\) is the cokernel of \(J_f:\mathcal{O}_{X,p}^n\to\mathcal{O}_{X,p}^r\),
and its $0^{th}$ Fitting ideal is the image of \(I_r(J_f)\) in \(\mathcal{O}_{X,p}\).
\item[(3)] Therefore, the support of \(T^1_{X,p}\) is exactly the singular locus of \(X\) (near \(p\)),
and the finiteness (resp.\ infiniteness) of \(\dim_{\mathbb C}T^1_{X,p}\) is equivalent to the singular locus being isolated (resp.\ having positive dimension) at \(p\).
\end{itemize}


\vspace{0.1cm}

\section*{Appendix G}

In this appendix, we give an algebraic proof of Theorem 5.1. More precisely, we show the  non-rigidity of cones by proving  the cohomological inclusion \(\;H^1(Y,T_Y)\hookrightarrow T^1(C(Y))_0\).

\vspace{0.1cm}

\noindent Let \(Y\subset\mathbb P^n_{\mathbb C}\) be a projective variety over \(\mathbb C\). Let
\(
S=\mathbb C[x_0,\dots,x_n]
\)
be the homogeneous coordinate ring of \(\mathbb P^n\) and let \(I\subset S\) be the homogeneous ideal of \(Y\). Set
\[
A:=S/I=\bigoplus_{m\ge0} A_m
\]
the homogeneous coordinate algebra of \(Y\). The affine cone over \(Y\) is
\[
C(Y)=\Spec(A)\subset\mathbb A^{n+1}_{\mathbb C}.
\]
We view \(A\) as a nonnegatively graded \(\mathbb C\)-algebra with \(A_0=\mathbb C\). For any graded \(A\)-module \(M\) we write \(M_k\) for its degree \(k\) piece.
For a \(\mathbb C\)-algebra \(B\) we denote by
\[
T^1(B):=\Ext^1_B(\Omega^1_{B/\mathbb C},\,B)
\]
the \emph{module of first order (infinitesimal) deformations} of the affine scheme \(\Spec B\). If \(B\) is graded, the group \(T^1(B)\) is graded and we denote its degree-zero piece by \(T^1(B)_0\).

\noindent On the geometric side, the tangent (sheaf) of \(Y\) is \(T_Y=\mathcal{H}om_{\mathcal O_Y}(\Omega^1_Y,\mathcal O_Y)\). The Zariski tangent space to the (abstract) deformation functor of \(Y\) is \(H^1(Y,T_Y)\), where  classes in \(H^1(Y,T_Y)\) correspond to first-order (abstract) deformations of \(Y\) over the dual numbers \(\mathbb C[\varepsilon]/(\varepsilon^2)\).

\begin{thm}{\label{Alg}}
If \(Y\subset\mathbb P^n_{\mathbb C}\) admits a nontrivial first-order deformation (i.e.\ \(H^1(Y,T_Y)\ne0\)) then the affine cone \(C(Y)=\Spec(A)\) admits a nontrivial graded first-order deformation, i.e.\ \(T^1(A)_0\ne0\). In particular, nonrigidity of \(Y\) implies nonrigidity of \(C(Y)\).
\end{thm}

\vspace{0.2cm}

\begin{lemma}
Under the above notation, there exists an injective  natural linear map
\[
\Phi:\; H^1(Y,T_Y)\ \hookrightarrow \ T^1(A)_0.
\]
\end{lemma}


\subsection{Two perspectives on first order deformations}

\noindent Let us briefly recall the two classical descriptions of first order deformations used below.

\subsubsection{\sc Projective side: \(H^1(Y,T_Y)\).}
An abstract (or abstract projective) first order deformation of \(Y\) is a flat family
\[
\mathcal Y\ \longrightarrow\ \Spec\big(\mathbb C[\varepsilon]/(\varepsilon^2)\big)
\]
whose special fiber is \(Y\). Isomorphism classes of such first order deformations are in natural bijection with \(H^1(Y,T_Y)\) (see standard deformation theory references). More precisely, if \(\{U_i\}_{i\in I}\) is an affine open cover of \(Y\), a class in \(H^1(Y,T_Y)\) may be represented by a Čech \(1\)-cocycle
\[
\{\theta_{ij}\in \Gamma(U_i\cap U_j,\,T_Y)\}_{i,j\in I}
\]
satisfying the cocycle relation \(\theta_{ij}+\theta_{jk}+\theta_{ki}=0\) on triple overlaps. Gluing the trivial first order families \(U_i\times\Spec\mathbb C[\varepsilon]/(\varepsilon^2)\) by exponentiating the \(\theta_{ij}\) (i.e.\ using the automorphisms \(1+\varepsilon\theta_{ij}\) on overlaps) produces the first order deformation \(\mathcal Y\).

\subsubsection{\sc Affine side: graded \(T^1(A)\).}
For an  affine \(\mathbb C\)-algebra \(B\) (not necessarily graded), first order deformations of \(\Spec B\) as an affine scheme are parametrized by \(T^1(B)=\Ext^1_B(\Omega^1_{B/\mathbb C},B)\). If \(B\) is graded, then graded first order deformations which respect the grading are parametrized by the degree-zero piece \(T^1(B)_0\). In particular, a graded first order deformation of \(A\), which is equal to the homogeneous coordinate algebra of \(Y\), defines a first order deformation of the affine cone \(C(Y)=\Spec A\), and graded deformations with nonzero class in \(T^1(A)_0\) give nontrivial deformations of the cone which respect the scaling action.

\subsection{Construction of the map \(\Phi\)}\label{sec:construction}

We now construct a natural linear map
\[
\Phi:\ H^1(Y,T_Y)\longrightarrow T^1(A)_0.
\]

\subsubsection*{{\bf Step 1.} First, we choose an affine cover and realize the projective deformation by gluing.}
Pick a finite affine open cover \(\mathfrak U=\{U_i\}_{i\in I}\) of \(Y\) such that each \(U_i\) is contained in a standard affine chart of \(\mathbb P^n\). Let
\[
\xi\in H^1(Y,T_Y)
\]
be represented by a Čech cocycle \(\{\theta_{ij}\in \Gamma(U_i\cap U_j,T_Y)\}_{i<j}\). Using the usual gluing construction, 
by  taking the automorphism \(1+\varepsilon\theta_{ij}\) on the overlap \(U_i\cap U_j\) of the trivial first order thickenings yields a first order deformation
\[
\mathcal Y\ \to\ \Spec\big(\mathbb C[\varepsilon]/(\varepsilon^2)\big),
\]
obtained by gluing the trivial thickenings \(U_i\times\Spec\mathbb C[\varepsilon]/(\varepsilon^2)\) by the transition automorphisms \(1+\varepsilon\theta_{ij}\).

\subsubsection*{{\bf Step 2.} We consider now  cones on affine patches.}
For each affine open \(U_i\subset Y\) consider its affine cone
\[
\widetilde U_i:=\Spec\big(\Gamma_*(\mathcal O_Y)|_{U_i}\big),
\]
more concretely \(\widetilde U_i\) is the affine patch of the cone over the standard affine chart containing \(U_i\). The trivial first order thickening \(U_i\times\Spec \mathbb C[\varepsilon]/(\varepsilon^2)\) corresponds to the trivial graded thickening of the graded coordinate ring of \(\widetilde U_i\).

On the overlaps \(U_i\cap U_j\), the vector field \(\theta_{ij}\) induces (by functoriality of derivations on coordinate rings) a degree-zero graded derivation of the graded coordinate ring of the cone over \(U_i\cap U_j\). More precisely, if \(R_{ij}\) denotes the graded ring of the cone over \(U_i\cap U_j\), then \(\theta_{ij}\) gives a \(\mathbb C\)-derivation
\[
\delta_{ij}:\ R_{ij}\ \longrightarrow\ R_{ij}
\]
that is homogeneous of degree \(0\).
Exponentiating \(1+\varepsilon\delta_{ij}\) yields graded automorphisms of the trivial graded first order thickening of \(R_{ij}\).

\subsubsection*{{\bf Step 3.} We glue graded local cones to a global graded deformation of \(A\).}
Use the graded automorphisms \(1+\varepsilon\delta_{ij}\) to glue the trivial graded thickenings of the local cones \(\widetilde U_i\). The result is a graded \(\mathbb C[\varepsilon]\)-algebra \(\mathcal A\) flat over \(\mathbb C[\varepsilon]\) whose special fiber is \(A\). Since all local derivations \(\delta_{ij}\) are homogeneous of degree \(0\), the glued algebra \(\mathcal A\) inherits a grading compatible with the projection \(\mathcal A\to\mathbb C[\varepsilon]\) and gives a graded first order deformation of \(A\). By construction the class of \(\mathcal A\) in \(T^1(A)_0\) depends only on the Čech class \(\xi\).

Define \(\Phi(\xi)\) to be the class of the graded deformation \(\mathcal A\) in \(T^1(A)_0\). This yields a well defined linear map
\[
\Phi:\ H^1(Y,T_Y)\longrightarrow T^1(A)_0.
\]

\subsection{The map \(\Phi\) is injective.}

We prove that \(\Phi\) is injective. Suppose \(\xi\in H^1(Y,T_Y)\) is such that \(\Phi(\xi)=0\in T^1(A)_0\). By definition this means that the graded first order deformation \(\mathcal A\) constructed in \S\ref{sec:construction} is \emph{graded-trivial}: there exists a graded \(A\)-algebra isomorphism
\[
\Psi:\ A\otimes_{\mathbb C}\mathbb C[\varepsilon]\ \xrightarrow{\ \simeq\ }\ \mathcal A.
\]
Projectivizing (i.e.\ applying \(\operatorname{Proj}\) to the graded algebras) yields an isomorphism of families over \(\Spec(\mathbb C[\varepsilon]/(\varepsilon^2))\)
\[
\operatorname{Proj}(\Psi):\ Y\times\Spec\mathbb C[\varepsilon]/(\varepsilon^2)\ \xrightarrow{\ \simeq\ }\ \operatorname{Proj}(\mathcal A)=\mathcal Y.
\]
Hence, the projective first order deformation \(\mathcal Y\) is trivial. But \(\mathcal Y\) was constructed from the Čech representative \(\xi\); so \(\xi=0\) in \(H^1(Y,T_Y)\). Therefore the kernel of \(\Phi\) is zero and \(\Phi\) is injective.

\begin{remark}
The injectivity argument is very general and uses only the fact that graded triviality of a deformation of the coordinate ring implies triviality of the projective deformation obtained by $\operatorname{Proj}$. The map is therefore faithful: nontrivial projective deformations produce nontrivial graded affine cone deformations.
\end{remark}

\subsection{Algebraic interpretation using Ext groups.}

One may give a more algebraic description of \(\Phi\) using the Ext groups and local cohomology. There are two standard facts that underlie the picture:

\begin{enumerate}[(a)]
\item For the graded algebra \(A=\bigoplus_{m\ge0}A_m\) the graded deformation group \(T^1(A)\) is the graded Ext group \(\Ext^1_A(\Omega^1_{A/\mathbb C},A)\). Graded deformations preserving grading are parametrized by the degree-zero piece \(T^1(A)_0\).
\item There is a spectral sequence 
relating graded Ext groups of \(A\) to sheaf Ext groups on \(Y=\operatorname{Proj}A\), and in low degrees one obtains a map
\[
H^1(Y,T_Y)\ \longrightarrow\ \Ext^1_A(\Omega^1_{A/\mathbb C},A)_0 = T^1(A)_0.
\]
This map is the same as the geometric one constructed above. The injectivity then follows from the same projectivization argument: a trivial class in \(T^1(A)_0\) yields a trivial projective deformation, hence trivial class in \(H^1(Y,T_Y)\).
\end{enumerate}

A standard reference for these connections (graded Ext vs. sheaf Ext and the relation of graded deformations with projective deformations) is the deformation theory chapter of any text on deformation theory or on graded algebras e.g. in Sernesi \cite{S-06} or in Pinkham \cite{P-74}.

\vspace{0.2cm}

Combining the two viewpoints we obtain the statement already proved geometrically:

\noindent {\it Proof of \ref{Alg}.}
If \(H^1(Y,T_Y)\ne0\), pick a nonzero class \(\xi\). By the construction above \(\Phi(\xi)\in T^1(A)_0\) is the class of a graded first order deformation of \(A\). Since \(\Phi\) is injective, \(\Phi(\xi)\ne0\). Therefore \(T^1(A)_0\ne0\) and \(C(Y)\) is nonrigid.

\section*{Appendix H}

In this appendix we study the evaluation matrix for $H^0(\PP^1,\T_{\PP^1})\to \bigoplus_{p\in\operatorname{Supp} D}\T_{\PP^1,p}$.


\noindent Let \(C=\PP^1_k\) over an algebraically closed field \(k\).
Fix \(g\) disjoint ordered pairs \((a_i,b_i)\) of distinct points of \(C\) (so the $2g$ points are pairwise distinct). Let
\[
D=\sum_{i=1}^g (a_i+b_i)
\]
be the reduced effective divisor of degree \(2g\). We consider the restriction (evaluation) map
\[
\mathrm{ev}: H^0(\PP^1,\T_{\PP^1}) \longrightarrow H^0(\PP^1,\T_{\PP^1}|_D)
\cong \bigoplus_{p\in\operatorname{Supp} D} \T_{\PP^1,p},
\]
whose matrix (in appropriate bases) is the \(3\times 2g\) matrix we now display.

\bigskip

\noindent\textbf{Global sections of \(\T_{\PP^1}\).} On the affine chart \(U=\{Y\neq0\}\) with coordinate \(t=X/Y\),
every global section of \(\T_{\PP^1}\) can be written as
\[
v(t) \;=\; (\alpha + \beta t + \gamma t^2)\,\frac{\partial}{\partial t},
\qquad (\alpha,\beta,\gamma)\in k^3,
\]
and these three coefficients give a basis of \(H^0(\PP^1,\T_{\PP^1})\cong H^0(\PP^1,\OO(2))\cong k^3\).

If a point \(p\in U\) has affine coordinate \(t=p\in k\), the value of \(v\) at \(p\) (in the tangent line identified with the $\partial/\partial t$ direction) is
\[
v(p) = \alpha + \beta p + \gamma p^2.
\]

If a point is the point at infinity \(\infty=[1:0]\), then we work on the chart \(s=1/t\). One checks (see below) that the evaluation at \(\infty\) gives the linear functional \((\alpha,\beta,\gamma)\mapsto -\gamma\) on the basis \((\alpha,\beta,\gamma)\). Up to an overall sign convention this corresponds to the column vector \([0,0,1]^\top\) in the matrix below (the sign does not affect linear independence or rank).

\bigskip

\noindent\textbf{The explicit \(3\times 2g\) evaluation matrix.}  
Label the $2g$ distinct marked points as \(p_1,p_2,\dots,p_{2g}\). For each \(p_j\) choose its affine coordinate \(t_j\in k\) if \(p_j\neq\infty\), and write \(t_j=\infty\) if \(p_j\) is the point at infinity. In the bases
\[
\text{(source basis)}\quad \{\,e_1,e_2,e_3\,\} \leftrightarrow \{\ 1,t,t^2\ \}
\qquad\text{and}\qquad
\text{(target basis at each }p_j)\ \frac{\partial}{\partial t}\ \text{(resp.\ }\frac{\partial}{\partial s}\text{ at }\infty),
\]
the evaluation map is represented by the \(3\times 2g\) matrix
\[
M \;=\; \begin{bmatrix}
1 & 1 & \cdots & 1 & \cdots & 1 \\
t_1 & t_2 & \cdots & t_j & \cdots & t_{2g} \\
t_1^2 & t_2^2 & \cdots & t_j^2 & \cdots & t_{2g}^2
\end{bmatrix},
\]
where each column corresponding to a finite point \(p_j\) equals \(\begin{bmatrix}1\\ t_j\\ t_j^2\end{bmatrix}\).  
If some \(p_j=\infty\) we may replace the corresponding column by \(\begin{bmatrix}0\\0\\1\end{bmatrix}\) (or \(\begin{bmatrix}0\\0\\-1\end{bmatrix}\) depending on the sign convention); the sign does not affect rank.

Therefore, \(M\) is simply the \(3\times 2g\) matrix whose finite columns are the usual Vandermonde columns \([1,t,t^2]^\top\).

\bigskip

\noindent\textbf{Remark (evaluation at infinity calculation).}
Write \(s=1/t\). The vector field can be expressed in the $s$-chart as
\[
v = (\alpha + \beta t + \gamma t^2)\frac{\partial}{\partial t}
= -(\alpha s^2 + \beta s + \gamma)\frac{\partial}{\partial s}.
\]
Evaluating at \(s=0\) (the point \(\infty\)) gives the value \(-\gamma\partial/\partial s\). Hence the column for \(\infty\) is proportional to \([0,0,-1]^\top\); up to sign we may take \([0,0,1]^\top\).

\bigskip

\noindent\textbf{Rank analysis via $3\times3$ minors (Vandermonde).}

To prove the evaluation map has rank \(3\) (when possible), it suffices to find a nonzero \(3\times3\) minor of \(M\).
Pick any three columns corresponding to three distinct finite points with affine coordinates \(u,v,w\in k\).
The \(3\times3\) matrix formed by these columns is
\[
V(u,v,w)
=\begin{bmatrix}
1 & 1 & 1\\[4pt]
u & v & w\\[4pt]
u^2 & v^2 & w^2
\end{bmatrix}.
\]
Its determinant is the classical Vandermonde determinant:
\[
\det V(u,v,w) = (v-u)(w-u)(w-v).
\]
Thus \(\det V(u,v,w)\neq 0\) \(\iff\) \(u,v,w\) are pairwise distinct. Consequently, if there exist three finite marked points with distinct affine coordinates, then some \(3\times3\) minor of \(M\) is nonzero and so \(\operatorname{rank}M=3\).

If one of the three chosen points equals \(\infty\), say pick columns for \(u,v,\infty\) with \(u\neq v\), the \(3\times3\) matrix (after using column \([0,0,1]^\top\) for \(\infty\)) is
\[
\begin{bmatrix}
1 & 1 & 0\\
u & v & 0\\
u^2 & v^2 & 1
\end{bmatrix},
\]
whose determinant equals
\[
\det\begin{bmatrix}
1 & 1 & 0\\
u & v & 0\\
u^2 & v^2 & 1
\end{bmatrix}
= \det\begin{bmatrix}
1 & 1 \\ u & v
\end{bmatrix} \cdot 1
= v-u \neq 0
\]
when \(u\neq v\). So a finite pair plus \(\infty\) also yields a nonzero minor provided the two finite points are distinct.

Therefore:

\begin{enumerate}[(i)]
  \item If there exist three distinct finite marked points, \(\rank M=3\).
  \item If there exist two distinct finite marked points and at least one marked point at \(\infty\), \(\rank M=3\).
  \item If there are no three distinct marked points among the \(2g\) (i.e.\ at most two distinct affine values and possibly \(\infty\)), then \(\rank M\le 2\).
\end{enumerate}

\bigskip

\noindent\textbf{Immediate corollary for our nodal-curve problem.} We assumed the \(2g\) points \(\{a_1,b_1,\dots,a_g,b_g\}\) are pairwise distinct. Hence, if \(g\ge 2\) there are at least \(4\) distinct marked points, and in particular there exist (many) choices of three distinct finite marked points; therefore, \(\rank M=3\) always when \(g\ge 2\). When \(g=1\) we have only two points and therefore \(\rank M\le 2\) (indeed generically \(\rank M=2\) and then \(h^1(\PP^1,\T_{\PP^1}(-D))=2-2=0\) as computed earlier).

\bigskip

\subsection{Appendix: Special configurations where the evaluation rank can drop}

Below we list and analyze all ways the rank of the evaluation map \(M\) can be \(<3\). Because of the Vandermonde argument above, the only way to have \(\rank M < 3\) is that every \(3\times3\) minor vanishes; equivalently, all columns of \(M\) lie in some fixed \(2\)-dimensional subspace of \(k^3\). Geometrically, this means there exists a nonzero quadratic polynomial
\[
Q(t)=A + B t + C t^2 \quad (A,B,C\in k,\ \text{not all }0)
\]
such that
\[
Q(t_j)=0 \qquad\text{for every finite marked affine coordinate }t_j,
\]
and also \(Q\) must vanish appropriately at \(\infty\) if \(\infty\) is among the marked points (recall that evaluation at \(\infty\) samples the \(\gamma\)-coefficient).

Important observations that classify the possibilities:

\begin{enumerate}
  \item A nonzero polynomial \(Q(t)=A+Bt+Ct^2\) of degree \(\le2\) has at most \(2\) distinct finite roots in \(k\), unless \(Q\equiv 0\).
  \item Therefore if there are \(\ge 3\) distinct finite marked points, no nonzero quadratic \(Q\) can vanish at all of them, so \(\rank M\ge 3\).
  \item If there are only two distinct finite marked coordinates \(u\) and \(v\) among the marked points (but perhaps with multiplicities in the list of marked points), one can find a nonzero quadratic \(Q\) with roots exactly \(u\) and \(v\), e.g. \(Q(t)=(t-u)(t-v)\). In that case all finite columns lie in the 2-dimensional subspace orthogonal to the linear functional \((A,B,C)\) with \(Q(t)=A+Bt+Ct^2\). If no marked point is at infinity, this yields \(\rank M\le 2\).
  \item If one of the marked points is at \(\infty\), vanishing at \(\infty\) imposes the condition \(\gamma=0\) (because evaluation at infinity picks out \(-\gamma\)). Thus for a quadratic \(Q\) to vanish at \(\infty\) and at finitely many affine points it must actually be linear in \(t\), and hence can vanish at at most one finite point (unless it is identically zero). So the presence of \(\infty\) typically prevents a large simultaneous vanishing unless the finite points are very few.
\end{enumerate}

Combining these observations shows the only situations (over an algebraically closed field \(k\)) in which \(\rank M\le 2\) are:

\begin{enumerate}[(i)]
  \item \(\mathbf{g=0}\): no marked points, \(M\) is empty and trivially rank \(0\).
  \item \(\mathbf{g=1}\): two distinct marked points (no three points exist), so \(\rank M\le 2\). Generically, \(\rank M=2\).
  \item \(\mathbf{\text{Some marked points coincide (not allowed in our hypothesis).}}\) If the \(2g\) preimages are not pairwise distinct, then it is possible to have at most two distinct affine coordinates among them, and then \(\rank M\le 2\).
  \item \(\mathbf{\text{Special symmetric case (degenerate):}}\) If all finite marked points lie in the set of roots of a fixed quadratic polynomial \(Q\) (i.e.\ they take only two distinct affine values), and additionally there is no marked point at \(\infty\), then \(\rank M\le 2\). This is a genuinely special (nongeneric) configuration.
\end{enumerate}

Since in the problem statement we require \(g\) disjoint pairs of smooth points, the \(2g\) points are pairwise distinct, so the only relevant possibility for rank \(<3\) is when \(2g\le 2\), i.e.\ \(g\le 1\). Hence for all allowed inputs with \(g\ge 2\) we always have \(\rank M=3\).

\bigskip

\noindent\textit{In conclusion we have the following:}
\begin{enumerate}
  \item The evaluation matrix \(M\) (explicitly written above) has columns \([1,t_j,t_j^2]^\top\) for finite marked points and \([0,0,1]^\top\) for \(\infty\).
  \item Any three distinct finite points give a \(3\times3\) Vandermonde minor \((v-u)(w-u)(w-v)\), so \(\rank M=3\) whenever there are three distinct finite marked points.
  \item Under the hypothesis that the \(2g\) preimages of the nodes are pairwise distinct, if \(g\ge 2\) then \(\rank M=3\). For \(g=1\) one has \(\rank M\le 2\) (generically \(=2\)).
  \item Therefore, the long exact sequence computation in the main note shows \(h^1(\PP^1,\T_{\PP^1}(-D))=2g-3\) for \(g\ge 2\) and \(=0\) for \(g=1\), and adding the \(g\) local smoothing directions gives \(\dim_k\Ext^1_X(\Omega^1_X,\OO_X)=3g-3\) for \(g\ge2\) and \(=1\) for \(g=1\).
\end{enumerate}


\end{document}